\documentclass[a4paper,11pt]{amsart}

\usepackage[T1]{fontenc}
\usepackage[utf8]{inputenc}
\usepackage{lmodern}
\usepackage[english]{babel}
\usepackage[autostyle]{csquotes}

\usepackage[pdfusetitle,linktocpage,hidelinks]{hyperref}
\usepackage{doi}
\usepackage{tikz-cd}

\usepackage{enumitem}

\usepackage{amsmath,amssymb,amsthm,amscd}
\newtheorem{thm}{Theorem}[section]
\newtheorem{prop}[thm]{Proposition}
\newtheorem{lem}[thm]{Lemma}
\newtheorem{cor}[thm]{Corollary}

\newtheorem{ex}[thm]{Example}
\newtheorem{defn}[thm]{Definition}

\theoremstyle{remark}
\newtheorem{rem}[thm]{Remark}

\newcommand{\supp}{\mathrm{supp}}
\newcommand{\insupp}{\mathrm{insupp}}
\newcommand{\outsupp}{\mathrm{outsupp}}
\newcommand{\essupp}{\mathrm{essupp}}
\newcommand{\diag}{\mathrm{diag}}
\newcommand{\sign}{\mathrm{sign}}
\newcommand{\Ad}{\mathrm{Ad}}

\DeclareSymbolFont{symbolsC}{U}{pxsyc}{m}{n}
\DeclareMathSymbol{\coloneqq}{\mathrel}{symbolsC}{"42}

\begin{document}
\title{Ideal structure of $\ell^p$ uniform Roe algebras}
\author{Yeong Chyuan Chung}
\address{School of Mathematics, Jilin University, Changchun 130012, P.R. China}
\email{chungyc@jlu.edu.cn}
\author{Xinhui Du}
\address{School of Mathematics, Jilin University, Changchun 130012, P.R. China}
\email{xhdu24@mails.jlu.edu.cn}

\date{\today}

\begin{abstract}
For a uniformly locally finite coarse space $(X,\mathcal E)$ and $p\in\{0\}\cup[1,\infty]$, we classify the nonzero geometric ideals of the $\ell^p$ uniform Roe algebra $B_u^p(X,\mathcal E)$ by the ideals of $\mathcal E$, equivalently by ideals in the associated family of controlled partial coverings. Hence the geometric ideal lattice is canonically independent of $p$. For $p\in[1,\infty]$, we identify $B_u^p(X,\mathcal E)$ isometrically with the reduced $L^p$ operator algebra of the coarse groupoid and show that this identification preserves essential support, carrying geometric and ghostly ideals to dynamical and restrictive ideals. Under property A, controlled $p$-kernels yield a multiplier approximate identity with controlled propagation; consequently every ideal is geometric and every ghost is trivial. Thus the full ideal lattices are canonically independent of $p$ under property A. For $p\in\{0,1,\infty\}$, the complete ideal classification holds
without property A; on the groupoid side, the corresponding
unconditional conclusion holds for $p\in\{1,\infty\}$.
\end{abstract}

\maketitle

\tableofcontents

\section{Introduction}

One of the fundamental themes in coarse geometry is the interplay
between the large-scale structure of a space and the algebraic
properties of its associated operator algebras. Roe algebras, originally
introduced in connection with index theory on non-compact manifolds,
have become central objects in higher index theory and coarse geometry
\cite{Roe88,Roe96,Yu,BBFKVW,BFV,CW01,CW04,CW05,MV25,Sako20,WZ25}.

For uniform Roe algebras, Chen and Wang
\cite{CW01,CW04,CW05} established a close connection between ideal
structure and coarse geometry. In particular, they introduced geometric
ideals, identified their lattice with the lattice of ideals of the coarse
structure, and showed that property A forces every ideal to be geometric
and every ghost to be trivial in the bounded-geometry metric setting
\cite{CW04}. The coarse groupoid of Skandalis--Tu--Yu
\cite{STY} provides a groupoid framework for these phenomena. More
recently, Wang and Zhang \cite{WZ25} decomposed the full ideal lattice
into essential-support blocks, while the limit operator theory of
\v{S}pakula and Willett \cite{SW17} provides a complementary
description of the corresponding asymptotic behavior.

All of these results lie in the $C^*$-algebraic and Hilbert space framework. Moving from $\ell^2(X)$ to the Banach spaces $\ell^p(X)$ with $p\in\{0\}\cup[1,\infty]$ replaces the $C^*$-algebra by a Banach algebra, connecting the subject to the rapidly developing theory of 
$L^p$ operator algebras. The central question addressed in this paper is how the rich interplay between coarse geometry and operator algebras extends to the 
$L^p$ setting, continuing a line of inquiry developed in \cite{Braga21,BV,ChungLi18,ChungLi21,LWZ19,SZ20}.

We work with uniformly locally finite coarse spaces
$(X,\mathcal E)$, not necessarily arising from metrics; see
Definition~\ref{def:CS}. This general setting has proved useful for
extending metric coarse-geometric results
\cite{HPR,Roe03,DH,LV,MV,Sako}. Our main contributions are organized
into four themes.

1. Partial coverings and coarse ideals.
There is an equivalent description of a coarse structure via families of partial coverings (see Appendix~\ref{appen1}).
In Section~\ref{section3}, we define ideals in such families and prove that the lattice of ideals of $\mathcal{E}$ is isomorphic to the lattice of ideals in the family of $\mathcal{E}$-controlled partial coverings (Theorem~\ref{thm:covideal}). 

2. Geometric and ghostly ideals in $\ell^p$ uniform Roe algebras.
Using controlled truncations, we prove that for every
$p\in\{0\}\cup[1,\infty]$, the lattice of nonzero geometric ideals of
$B^p_u(X,\mathcal E)$ is canonically isomorphic to the lattice of
ideals of $\mathcal E$ (Theorem~\ref{thm:geometric}). In particular,
the geometric ideal lattice is independent of $p$, extending the
$C^*$-algebraic result of Chen and Wang \cite{CW04}. Following
\cite{WZ25}, we also decompose the full ideal lattice into
essential-support blocks whose extremal elements are the associated
geometric and ghostly ideals (Proposition~\ref{prop:ghostly}).

3. Groupoid $L^p$ operator algebras and limit operators.
We extend limit operator theory to uniformly locally finite coarse
spaces and obtain a canonical isometric isomorphism between $F^p_{\mathrm{red}}(G(X))$ and $B^p_u(X,\mathcal E)$ for $p\in[1,\infty]$ (Proposition~\ref{prop:isometric}). The induced ideal correspondence
preserves essential support (Theorem~\ref{thm:isomorphic}); geometric
and ghostly ideals correspond, respectively, to dynamical and
restrictive ideals (Corollary~\ref{cor:isomorphic} and
Proposition~\ref{prop:restrictive}). Limit operators also give a
boundary description of these ideals (Corollary~\ref{prop:limitideal}).

4. Property A and full ideal classification.
For $1<p<\infty$, we characterize property A through controlled
$p$-partitions of unity and controlled $p$-kernels
(Lemma~\ref{lem:propertyA}). We then show that property A yields a
multiplier approximate identity with controlled propagation
(Theorem~\ref{thm:multappr}); consequently every ideal is geometric
and every ghost is trivial.

It follows that, under property A, every nonzero ideal of
$B^p_u(X,\mathcal E)$ is uniquely of the form
$B^p_u(X,\mathcal J)$, and the full ideal lattice is canonically
independent of $p\in\{0\}\cup[1,\infty]$. At the extreme exponents
$p\in\{0,1,\infty\}$, the same classification and
exponent-independence hold without property A. On the groupoid side, the full ideal lattices of
$F^p_{\mathrm{red}}(G(X))$, $p\in[1,\infty]$, are likewise
canonically independent of $p$ under property A; without property A,
the same conclusion holds for $p\in\{1,\infty\}$
(Corollary~\ref{cor:fullclassification}).

Section~2 collects the preliminaries required for the rest of the paper.
Appendix~\ref{appen1} recalls and generalizes the language of partial
coverings used to describe coarse structures, while
Appendix~\ref{appen2} develops the limit coarse spaces and limit operators
used to analyze the boundary of the coarse groupoid.


\section{Preliminaries}

\subsection{Coarse geometry}

Let $X$ be a set. For $A\subseteq X\times X$ and $B\subseteq X\times X$, denote
\begin{gather*}
    A^{-1}=\{(y,x): (x,y)\in A\},\\
    A\circ B=\{(x,z): \exists y\in X, (x,y)\in A\ \text{and}\ (y,z)\in B\}.
\end{gather*}
Thus, $A^{-1}$ is the inverse of $A$, and $A\circ B$ is the product of
$A$ and $B$ in the pair groupoid $X\times X$. We also write $r$ and $s$
for the range and source maps from $X\times X$ to $X$ defined by $r(x,y)=x$ and $s(x,y)=y$.

\begin{defn} \cite[Definition 2.1]{STY}\label{def:CS}
    A coarse structure $\mathcal{E}$ on $X$ is a collection of subsets of $X\times X$, called entourages, that have the following properties:
    \begin{enumerate}
        \item For any entourages $A$ and $B$, $A^{-1}$, $A\circ B$ and $A\cup B$ are entourages;
        \item Every finite subset of $X\times X$ is an entourage;
        \item Any subset of an entourage is an entourage.
    \end{enumerate}
\end{defn}

Property (ii) above is sometimes referred to as being coarse connected.
If $\Delta=\{(x,x): x\in X\}$ is an entourage, then the coarse structure is said to be unital. A set $X$ endowed with a coarse structure $\mathcal{E}$, denoted by $(X,\mathcal{E})$, is called a coarse space. Entourages in $\mathcal{E}$ are also called ($\mathcal{E}$-)controlled sets.

Observe that the intersection of a family of coarse structures is a coarse structure. 
\begin{defn}
    Let $\mathcal{C}$ be a family of subsets of $X\times X$. The intersection of all the coarse structures on $X$ that contain $\mathcal{C}$ is a coarse structure which we call the coarse structure generated by $\mathcal{C}$.
\end{defn}

For any subset $Y\subseteq X$ and any entourage $E\in\mathcal{E}$, we define \[E[Y]=\{x\in X: \exists y\in Y, (x,y)\in E\}\] and define the (one-sided) $E$-neighborhood $N_E(Y)$ of $Y$ by
\[
N_E(Y)=Y\cup E[Y].
\]
When $Y=\{x\}$, we write $E_x$ for $E[\{x\}]$ and $N_E(x)$ for $N_E(\{x\})$. Note that
\[
N_E(x)=\{x\}\cup E_x,
\]
where $E_x=\{y\in X: (y,x)\in E\}=r(E\cap s^{-1}(x))$.
\begin{defn} \cite[Definition 2.4]{STY}
    A coarse structure $\mathcal{E}$ on a set $X$ is said to be uniformly locally finite if 
    \[
    n(E):= \sup_{x\in X} \max(\#N_{E}(x), \#N_{E^{-1}}(x))< \infty
    \]
    for any entourage $E\in \mathcal{E}$.
\end{defn}

We say that $X$, or equivalently $(X,\mathcal{E})$, is a uniformly locally
finite coarse space if the coarse structure $\mathcal{E}$ on $X$ is uniformly locally finite. 

\begin{defn} \cite[Definition 2.3]{CW04}\label{def:partial}
A partial translation is an entourage $E\in\mathcal{E}$ such that the maps $r$ and $s$ are both injective on $E$. 

We use $\Gamma_{\mathcal{E}}$ to denote the collection of all partial translations in $\mathcal{E}$.
\end{defn}

\begin{lem} (cf. \cite[Lemma 2.8]{STY}, \cite[Lemma 4.10]{Roe03}, \cite[Lemma 2.4]{CW04}) \label{lem:CW}
    The coarse structure $\mathcal{E}$ on $X$ is generated by $\Gamma_{\mathcal{E}}$ if and only if it is uniformly locally finite. Moreover, when $\mathcal{E}$ is a uniformly locally finite coarse structure on $X$, and $n(E)=n$, we have:
    \begin{enumerate}
    \item For any entourage $E$, the space $X$ can be partitioned into finitely
many $E$-separated subspaces; more precisely,
    \[
        X=X_1\cup X_2\cup\cdots\cup X_m,
    \]
    where $X_i\cap X_j=\emptyset$ for $i\neq j$, and for any $i$ and any points $x,y\in X_i$ with $x\neq y$, we have $(x,y)\notin E\cup E^{-1}$. The number $m$ depends on $E$, and one may take $m\leq n^2-n+1$.
    \item Every entourage $E\in\mathcal{E}$ can be partitioned into finitely many
partial translations; that is,
    \[
    E=E_1\cup E_2\cup \cdots \cup E_l,
    \]
    where $E_i\in\Gamma_{\mathcal{E}}, E_i\cap E_j= \emptyset$ for $i\neq j$. The number $l$ depends on $E$, and one may take $l\leq 2n-1$.
    \end{enumerate}
\end{lem}

\begin{ex} \label{ex:min}
Let $X$ be a set.
    \begin{enumerate}
        \item Let $\Gamma_{min}$ be the collection of singleton partial translations $\{(x,y)\}\subseteq X\times X$, and let $\mathcal{E}_{min}$ be the coarse structure generated by $\Gamma_{min}$. Then $\mathcal{E}_{min}$ consists of all finite subsets of $X\times X$ and is uniformly locally finite. Moreover, it is clear that $\mathcal{E}_{min}$ is the smallest coarse structure on $X$.
        \item Suppose $X$ is endowed with a uniformly discrete metric $d$ of bounded geometry (that is, for every $R>0$ there exists $N$ such that every ball of radius $R$ in $X$ contains at most $N$ points). For $R>0$, set
\[
E_R:=\{(x,y)\in X\times X:d(x,y)<R\}.
\]
The bounded coarse structure $\mathcal{E}_d$ consists of the subsets of
the entourages $E_R$, $R>0$.
          The coarse structure $\mathcal{E}_d$ is clearly unital and uniformly locally finite.
    \end{enumerate}
\end{ex}

Let $Y\subseteq X$. For any entourage $E\in\mathcal{E}$, the restriction of $E$ to $Y\times Y$ is defined by
\[
E|_Y=E\cap(Y\times Y).
\]
Set $\mathcal{E}|_Y:=\{E|_Y: E\in\mathcal{E}\}$. Then it is clear that $\mathcal{E}|_Y\subseteq\mathcal{E}$, and $\mathcal{E}|_Y$ is a uniformly locally finite coarse structure on $Y$ whenever $\mathcal{E}$ is uniformly locally finite on $X$.

\begin{defn}\cite[Definition 2.6]{CW04}\label{def:subspace}
    A subset $Y\subseteq X$ endowed with the relative coarse structure $\mathcal{E}|_Y$ is called a subspace of the coarse space $(X,\mathcal{E})$.
\end{defn}

\begin{defn}\cite[Definition 4.1]{CW04}
For a coarse space $(X,\mathcal{E})$, let $\mathcal{J}\subseteq\mathcal{E}$
be another coarse structure on $X$. We call $\mathcal{J}$ an ideal of
$\mathcal{E}$ if, for every $E\in\mathcal{E}$ and every
$A\in\mathcal{J}$,
\[
E\circ A\in\mathcal{J}
\ \text{and}\ A\circ E\in\mathcal{J}.
\]
\end{defn}

\begin{ex}
    The smallest uniformly locally finite coarse structure $\mathcal{E}_{min}$ on $X$ is the smallest ideal of any uniformly locally finite coarse structure on $X$.
\end{ex}

    If $\mathcal{J}_1$ and $\mathcal{J}_2$ are ideals of $\mathcal{E}$, then the collections 
    \begin{align*}
        \mathcal{J}_1\wedge\mathcal{J}_2&:=\mathcal{J}_1\cap\mathcal{J}_2,\\
        \mathcal{J}_1\vee\mathcal{J}_2&:=\{A\cup B: A\in\mathcal{J}_1, B\in\mathcal{J}_2\}
    \end{align*}
    are also ideals of $\mathcal{E}$, called the meet and join of $\mathcal{J}_1$ and $\mathcal{J}_2$. The set of all ideals of the coarse structure $\mathcal{E}$ forms a lattice, denoted by $\mathcal{J}[\mathcal{E}]$.

\begin{defn} \cite[Section 3.2]{STY}\label{def:coarsegroupoid}
    Let $(X,\mathcal{E})$ be a uniformly locally finite coarse space. Define its coarse groupoid by
    \[ 
    G(X,\mathcal{E}):=\bigcup_{E\in\mathcal{E}}\overline{E}^{\beta(X\times X)}\subseteq\beta(X\times X),
    \]
    where $\beta(X\times X)$ is the Stone--\v{C}ech compactification of $X\times X$. The maps $r,s:X\times X\to X$ extend to maps
$G(X,\mathcal{E})\to\beta X$, again denoted by $r$ and $s$. It was shown in \cite[Proposition 3.2]{STY} that $G(X,\mathcal{E})$ is \'{e}tale, locally compact, Hausdorff and principal. Its unit space 
    \[
    G(X,\mathcal{E})^{(0)}=\bigcup_{ A\subseteq X,\Delta_A\in\mathcal{E}}\overline{A}^{\beta X}=\bigcup_{ E\in\mathcal{E}}\overline{r(E)}^{\beta X}\subseteq\beta X,
    \]
    is totally disconnected, where $\Delta_A:=\{(x,x):x\in A\}$.
    \end{defn}
    Hence, $G(X,\mathcal{E})$ is also ample. Since every finite subset of $X\times X$ is an entourage, it is clear that $X\subseteq G(X,\mathcal{E})^{(0)}$. Moreover, if $\mathcal{E}$ is unital, then $G(X,\mathcal{E})^{(0)}=\beta X$. In general, the groupoid $G(X,\mathcal{E})$ is not second countable.

We write $G(X)$ for $G(X,\mathcal{E})$ and $G(X)^{(0)}$ for $G(X,\mathcal{E})^{(0)}$ when the coarse space $(X,\mathcal{E})$ is given.

\subsection{$\ell^p$ uniform Roe algebras}\label{subsection1}

Let $X$ be a discrete set. We use the following notation for the relevant Banach spaces:
\begin{enumerate}
    \item $\ell^p(X):=\{f:X\to\mathbb{C}:\sum\limits_{x\in X}|f(x)|^p<\infty\}$, for $p\in[1,\infty)$;
    \item $\ell^\infty(X):=C_b(X)$, which denotes the Banach space of bounded (continuous) functions from $X$ to $\mathbb{C}$;
    \item $\ell^0(X):=C_0(X)$, which denotes the Banach space of continuous functions from $X$ to $\mathbb{C}$ vanishing at infinity.
\end{enumerate}
For a given $p\in\{0\}\cup[1,\infty]$, we denote by $B(\ell^p(X))$ the Banach algebra of all bounded linear operators on the Banach space $\ell^p(X)$, and by $K(\ell^p(X))$ the space of compact operators. The representation
\[
\rho:\ell^\infty(X)\to B(\ell^p(X))
\]
is defined by pointwise multiplication:
\[
(\rho(f)\xi)(x):=f(x)\xi(x),\ \text{for}\ f\in\ell^\infty(X), \xi\in\ell^p(X).
\]
To simplify notation, we write $f\xi$ instead of $\rho(f)\xi$ and write $fT$ instead of $\rho(f)T$ for $T\in B(\ell^p(X))$. Let $\mathcal{F}$ be the set of finite subsets of $X$, directed by inclusion.
For $F\in\mathcal{F}$, let $\chi_F$ denote the characteristic function of
$F$. We write
\[
\mathcal{P}:=\{\chi_F\}_{F\in\mathcal{F}}.
\]

\begin{defn} \cite{CL}, \cite[Definition 2.3]{Zhang}\label{def:subalgebras}
    For any operator $K\in B(\ell^p(X))$, if
        \[
        \lim\limits_{F\in\mathcal{F}}\|K-K\chi_F\|=0\ and\ \lim\limits_{F\in\mathcal{F}}\|K-\chi_F K\|=0,
        \]
        we say $K$ is a $\mathcal{P}$-compact operator.
\end{defn}
Let $\mathcal{K}^p(X)$ denote the collection of all $\mathcal{P}$-compact operators.

We can always represent an operator $T\in B(\ell^p(X))$ as an $X\times X$ matrix:
\[
T\mapsto[T(x,y)]_{x,y\in X},
\]
with entries $T(x,y)=(T\delta_y)(x)\in\mathbb{C}$. When $p=\infty$, this map is not injective.

In \cite{Zhang}, Zhang studied $\ell^p$ uniform Roe algebras for
$p\in\{0\}\cup[1,\infty]$ when $X$ is a strongly discrete metric space
with bounded geometry, an important example of a uniformly locally finite
coarse space, and obtained several results specific to the extreme exponents
$p\in\{0,1,\infty\}$. We now let $X$ be a discrete set equipped with a
uniformly locally finite coarse structure $\mathcal{E}$ and introduce the
analogous definitions.

\begin{defn}
    Let $A=[A(x,y)]_{x,y\in X}$ be an $X$-by-$X$ matrix with entries in $\mathbb{C}$. We say that $A$ is a controlled propagation matrix if
    \begin{enumerate}
        \item the function $A:(x,y)\mapsto A(x,y)\in\mathbb{C}$ is in $\ell^\infty(X\times X)$;
        \item the support of $A$, $\supp(A)=\{(x,y)\in X\times X: A(x,y)\neq 0\}$ is an entourage in $\mathcal{E}$.
    \end{enumerate}
\end{defn}

Let $\mathbb{C}_u[X,\mathcal{E}]$ denote the collection of all controlled propagation matrices on $X$.

\begin{ex}\label{ex:partial}
    Let $v\in\ell^\infty(X\times X)$ and suppose $E:=\supp(v)=\{(x_i,y_i),i\in I\}\in\mathcal{E}$ is a partial translation. Then $V=[v(x,y)]_{x,y\in X}$ is called a partial translation matrix.
\end{ex}

For any $p\in\{0\}\cup[1,\infty]$ and any $\xi\in\ell^p(X)$, the partial translation matrix $V=[v(x,y)]_{x,y\in X}$ defines an operator on $\ell^p(X)$ by matrix multiplication:
\[
(V\xi)(x)=\sum_{y\in X}v(x,y)\xi(y)=\begin{cases}
    v(x_i,y_i)\xi(y_i) &\text{if}\ x=x_i\ \text{for some}\ i\in I,\\
    0 &\text{otherwise}.
\end{cases}
\]
The operator, also denoted by $V$, is called a partial translation operator. Since 
\[
\|V\xi\|_p \leq \sup_{x,y\in X}|V(x,y)|\cdot\|\xi\|_p,
\]
$V$ is bounded, with norm at most $\sup_{x,y\in X}|V(x,y)|$. 

\begin{lem}\label{lem:approximate1}
For $p\in\{0\}\cup[1,\infty]$, every controlled propagation matrix
$A=[A(x,y)]_{x,y\in X}$ defines an operator on $\ell^p(X)$ by matrix
multiplication:
\[
(A\xi)(x)=\sum_{y\in X}A(x,y)\xi(y).
\]
The sum is finite for each $x$. The resulting bounded operator is called a
controlled propagation operator. Moreover,
    \[
    \|A\|_p\leq N\cdot\sup_{x,y\in X}|A(x,y)|,
    \]
    where $N=2n(supp(A))-1$ is a finite positive number since the coarse structure $\mathcal{E}$ is uniformly locally finite.
\end{lem}

\begin{proof}
    By Lemma~\ref{lem:CW}(ii), $\supp(A)\in\mathcal{E}$ can be partitioned into the union of finitely many partial translations:
    \[
    \supp(A)=E_1\cup E_2\cup \cdots \cup E_N,
    \]
    where $N=2n(supp(A))-1$. Thus, we have $A\xi=\sum\limits^N_{i=1}A_i\xi\in\ell^p(X)$,
    \[
    A=A_1+A_2+\cdots+A_N
    \]
    and 
    \[
    \|A\|_p\leq N\cdot\max_{1\leq i\leq N}\|A_i\|_p\leq N\cdot\max_{1\leq i\leq N}\sup_{x,y\in X}|A_i(x,y)|=N\cdot\sup_{x,y\in X}|A(x,y)|,
    \]
    where $A_i$ are partial translation operators.
\end{proof}

We use the same notation $\mathbb{C}_u[X,\mathcal{E}]$ for the controlled
propagation operators on $\ell^p(X)$ at each exponent $p$. 

For each $p\in\{0\}\cup[1,\infty]$, $\mathbb{C}_u[X,\mathcal{E}]$ is a subalgebra of $B(\ell^p(X))$.
\begin{defn}
     The norm completion of $\mathbb{C}_u[X,\mathcal{E}]$ in $B(\ell^p(X))$ is called the $\ell^p$ uniform Roe algebra of the coarse space $(X, \mathcal{E})$, denoted by $B^p_u(X, \mathcal{E})$ or simply $B^p_u(X)$.
\end{defn}

When $p=2$, the uniform Roe algebra is a $C^*$-algebra, usually denoted by $C^*_u(X)$.
Also, if $X$ is infinite, then there is no algebra isomorphism between $B^p_u(X)$ and $B^q_u(X)$ whenever $1\leq p< q<\infty$ \cite[Remark 2.2]{ChungLi18}.

Recall that the natural map from $\ell^0(X)$ to its second dual
\[
\ell^0(X)\to \ell^0(X)^{**}\cong\ell^\infty(X)
\]
is precisely the inclusion map, which induces a canonical isometric homomorphism
\[
\Gamma: B(\ell^0(X))\to B(\ell^\infty(X)),\ T\mapsto T^{''}. 
\]
By definition, we have $T^{''}|_{\ell^0(X)}=T$ and the two operators have the same matrix form $[T(x,y)]_{x,y\in X}=[T^{''}(x,y)]_{x,y\in X}$. Moreover, for any $\xi\in\ell^\infty(X)$ and $\delta_x\in\ell^1(X)$, we have
\[
    (T^{''}\xi)(x)=\langle\ T^{''}\xi, \delta_x\ \rangle=\langle\ \xi, T{'}\delta_x\ \rangle=\sum_{y\in X}T(x,y)\xi(y)=\sum_{y\in X}T^{''}(x,y)\xi(y).
\]
Hence, there is an isometric isomorphism
\[
\Gamma|_{\mathbb{C}_u[X,\mathcal{E}]}: \mathbb{C}_u[X,\mathcal{E}]\subseteq B(\ell^0(X))\to \mathbb{C}_u[X,\mathcal{E}]\subseteq B(\ell^\infty(X)),\ A\mapsto A^{''}.
\]
Extending to the norm-closure, we have $B^0_u(X)\cong B^\infty_u(X)$. 
Thus, for any $p\in\{0\}\cup[1,\infty]$, we can always identify the operators in $B^p_u(X)$ with their infinite matrices. This identification allows us to use the controlled truncation techniques
developed below. All algebraic and ideal-theoretic results for $p=0$ will
therefore be obtained from the corresponding results for $p=\infty$ through
the canonical isometric isomorphism above. The two cases will be distinguished only when the underlying spaces $\ell^0(X)$ and $\ell^\infty(X)$ are involved.

\begin{rem}\label{rem:cutoff-ai}
The net
\[
\{\chi_{r(E)}\}_{E\in\mathcal E},
\]
indexed by $\mathcal E$ under inclusion, is a contractive two-sided
approximate identity for $B^p_u(X)$ for every
$p\in\{0\}\cup[1,\infty]$.

Indeed, for each $E\in\mathcal E$,
\[
\Delta_{r(E)}\subseteq E\circ E^{-1}\in\mathcal E,
\]
so $\chi_{r(E)}\in\mathbb C_u[X,\mathcal E]$. If
$A\in\mathbb C_u[X,\mathcal E]$ and
$F:=\supp(A)$, put
\[
E_0:=F\cup F^{-1}\in\mathcal E.
\]
Then
\[
r(E_0)=r(F)\cup s(F),
\]
and consequently
\[
\chi_{r(E)}A=A=A\chi_{r(E)}
\]
whenever $E\supseteq E_0$. Since
$\mathbb C_u[X,\mathcal E]$ is dense in $B^p_u(X)$ and
$\|\chi_{r(E)}\|_p\leq 1$, it follows that
\[
\chi_{r(E)}T\to T,
\ T\chi_{r(E)}\to T
\]
in operator norm for every $T\in B^p_u(X)$.
\end{rem}

Throughout the paper, all coarse spaces $(X,\mathcal{E})$ are assumed to be
uniformly locally finite. By an ideal of a Banach algebra we mean a closed, two-sided ideal.

\section{Ideals of a family of partial coverings}\label{section3}

Relations on $X$ can be encoded by partial coverings and conversely.
Accordingly, unital coarse structures can be described in terms of systems
of controlled partial coverings \cite{LV}. The relevant material is recalled
and generalized in Appendix~\ref{appen1}. Here we recall only the notions
needed below; see \cite{LV} for further details.

\subsection{Notions from partial coverings}

For a set $X$, a partial covering of $X$, denoted by $\boldsymbol{a}$, is a family of subsets of $X$; if $\bigcup_{A\in\boldsymbol{a}}A= X$, then it is a covering. A partial covering $\boldsymbol{a}$ is a refinement of $\boldsymbol{a'}$ if for any $A\in\boldsymbol{a}$, there exists $A'\in\boldsymbol{a'}$ such that $A\subseteq A'$. We denote this by $\boldsymbol{a}\leq \boldsymbol{a'}$. 

Let $\boldsymbol{i}_X =\{\{x\}: x\in X\}$ denote the minimal covering of $X$ (it is a refinement of every other covering). 
For a relation $E\subseteq X\times X$, the canonical partial covering induced by $E$ is defined by $E(\boldsymbol{i}_X) =\{E_{x}: x\in s(E)\subseteq X\}$ . 

For a partial covering $\boldsymbol{a}$, denote by $\diag(\boldsymbol{a}) =\bigcup\{A\times A: A\in\boldsymbol{a}\}$ the ``blocky diagonal'' relation on $X$ induced by $\boldsymbol{a}$. 
If there is another partial covering $\boldsymbol{b}$ such that $\diag(\boldsymbol{a})=\diag(\boldsymbol{b})$, then we say $\boldsymbol{a}\sim \boldsymbol{b}$; it is easy to check that this is an equivalence relation on all partial coverings of $X$. Denote by $[\boldsymbol{a}]$ the equivalence class of $\boldsymbol{a}$. 

To construct new partial coverings from old ones, for a subset
$E\subseteq X$ and a partial covering $\boldsymbol{b}$ of $X$, define
\[
St(E,\boldsymbol{b})
:=\bigcup\{B\in\boldsymbol{b}:E\cap B\neq\emptyset\}.
\]
For partial coverings $\boldsymbol{a}$ and $\boldsymbol{b}$, set
\[
St(\boldsymbol{a},\boldsymbol{b})
:=\{St(A,\boldsymbol{b}):A\in\boldsymbol{a}\},
\]
the star of $\boldsymbol{a}$ with respect to $\boldsymbol{b}$. We also set
\[
\boldsymbol{a}\cup\boldsymbol{b}
:=\{A:A\in\boldsymbol{a}\}\cup\{B:B\in\boldsymbol{b}\},
\]
which is again a partial covering.

\subsection{Ideals in the family of $\mathcal{E}$-controlled partial coverings}

From the correspondence in Proposition~\ref{prop:LV}, we see that for any family of partial coverings of $X$ satisfying (C0)-(C3) (see Proposition~\ref{prop:cs}), we can find a corresponding coarse structure $\mathcal{E}$. We now define ideals in such families and show that the lattice of ideals of a coarse structure is isomorphic to the lattice of ideals of the family of controlled partial coverings.

\begin{defn}
    Let $(X,\mathcal{E})$ be a coarse space, and let  $\mathfrak{C}(\mathcal{E})$ be the family of $\mathcal{E}$-controlled partial coverings of $X$ (see Definition~\ref{def:LV}). An ideal in $\mathfrak{C}(\mathcal{E})$ is a subset $\mathfrak{K}\subseteq\mathfrak{C}(\mathcal{E})$ satisfying:
    \begin{enumerate}
        \item if $\boldsymbol{a}\in\mathfrak{K}$, then $[\boldsymbol{a}]\subseteq\mathfrak{K}$;
        \item if $\boldsymbol{a}\in\mathfrak{K}$ and $\boldsymbol{a'}$ is a refinement of $\boldsymbol{a}$, then $\boldsymbol{a'}\in\mathfrak{K}$;
        \item if $\boldsymbol{a}\in\mathfrak{C}(\mathcal{E})$ and $\boldsymbol{b},\boldsymbol{c}\in\mathfrak{K}$, then $St(\boldsymbol{a},\boldsymbol{b}), St(\boldsymbol{b},\boldsymbol{a}), \boldsymbol{b}\cup\boldsymbol{c}\in\mathfrak{K}$;
        \item $\{\{x,y\}\}\in\mathfrak{K}$ for every $x,y\in X$. 
    \end{enumerate}
    
    If $\mathfrak{K_1}$ and $\mathfrak{K_2}$ are ideals in $\mathfrak{C}(\mathcal{E})$, the sets
    \begin{align*}
        \mathfrak{K_1}\wedge\mathfrak{K_2}&:=\mathfrak{K_1}\cap\mathfrak{K_2};\\
		\mathfrak{K_1}\vee\mathfrak{K_2}&:=\{\boldsymbol{c}: \exists\boldsymbol{a}\in\mathfrak{K_1}, \boldsymbol{b}\in\mathfrak{K_2}, \boldsymbol{c}\in[\boldsymbol{a}\cup\boldsymbol{b}]\},
    \end{align*}
    are ideals in $\mathfrak{C}(\mathcal{E})$, called the meet and join of $\mathfrak{K_1}$ and $\mathfrak{K_2}$. 
    
    The collection of all ideals in $\mathfrak{C}(\mathcal{E})$ forms a lattice, denoted by $\mathfrak{K}[\mathfrak{C}]$.
\end{defn}        

\begin{thm}\label{thm:covideal}
    For an ideal $\mathfrak{K}$ in $\mathfrak{C}(\mathcal{E})$, define 
    \[
    \mathcal{J}(\mathfrak{K})=\{E\in\mathcal{E}: \exists\boldsymbol{a}\in\mathfrak{K}, E\subseteq\diag(\boldsymbol{a})\},
    \]
    then $\mathcal{J}(\mathfrak{K})$ is an ideal in $\mathcal{E}$.

    For an ideal $\mathcal{J}$ in $\mathcal{E}$, define 
    \[
    \mathfrak{K}(\mathcal{J})=\{\boldsymbol{a}\in \mathfrak{C}(\mathcal{E}): \diag(\boldsymbol{a})\in\mathcal{J}\},
    \]
    then $\mathfrak{K}(\mathcal{J})$ is an ideal in $\mathfrak{C}(\mathcal{E})$. 
    
    Moreover, $\mathfrak{K}(\mathcal{J}(\mathfrak{K}))=\mathfrak{K}$, $\mathcal{J}(\mathfrak{K}(\mathcal{J}))=\mathcal{J}$, and the lattice $\mathfrak{K}[\mathfrak{C}]$ is isomorphic to the lattice $\mathcal{J}[\mathcal{E}]$.
\end{thm}    

\begin{proof}
    It is clear that $\mathcal{J}(\mathfrak{K})\subseteq\mathcal{E}$ is a coarse structure on $X$ (cf. Proposition~\ref{prop:cs}), and $\mathfrak{K}(\mathcal{J})\subseteq\mathfrak{C}(\mathcal{E})$ is a family of partial coverings that satisfies (C0)-(C3). It remains to prove that
    \begin{align*}
        E\circ F, F\circ E\in\mathcal{J}(\mathfrak{K}),\ &\forall E\in\mathcal{E}, \forall F\in\mathcal{J}(\mathfrak{K});\\
        St(\boldsymbol{a},\boldsymbol{b}), St(\boldsymbol{b},\boldsymbol{a})\in \mathfrak{K}(\mathcal{J}),\ &\forall\boldsymbol{a}\in\mathfrak{C}(\mathcal{E}), \forall
        \boldsymbol{b}\in\mathfrak{K}(\mathcal{J}).
    \end{align*}
    It follows from Lemma~\ref{lem:DH} that $E\circ F, \ F\circ E\subseteq \diag(St(\boldsymbol{b},\boldsymbol{a}\cup\boldsymbol{b}))\in\mathcal{J}(\mathfrak{K})$, so $E\circ F, F\circ E\in\mathcal{J}(\mathfrak{K})$. It follows from Lemma~\ref{lem:LV} that $\diag(St(\boldsymbol{a},\boldsymbol{b}))=\diag(\boldsymbol{b})\circ \diag(\boldsymbol{a})\circ \diag(\boldsymbol{b})\in\mathcal{J}$ and $\diag(St(\boldsymbol{b},\boldsymbol{a}))=\diag(\boldsymbol{a})\circ \diag(\boldsymbol{b})\circ \diag(\boldsymbol{a})\in\mathcal{J}$, so $St(\boldsymbol{a},\boldsymbol{b}), St(\boldsymbol{b},\boldsymbol{a})\in\mathfrak{K}(\mathcal{J})$. Using the correspondence in Proposition~\ref{prop:LV}, we have $\mathfrak{K}(\mathcal{J}(\mathfrak{K}))=\mathfrak{K}$ and $\mathcal{J}(\mathfrak{K}(\mathcal{J}))=\mathcal{J}$.  

    Both correspondences are order-preserving. Hence they define a lattice
isomorphism between $\mathfrak{K}[\mathfrak{C}]$ and
$\mathcal{J}[\mathcal{E}]$.
\end{proof}

\section{Characterization of ideals in $\ell^p$ uniform Roe algebras}\label{section4}

In this section, we use controlled truncations to associate ideals of the $\ell^p$ uniform Roe algebra $B^p_u(X)$ with ideals of the underlying coarse space $(X,\mathcal{E})$, obtaining a description for the lattice structure of all geometric ideals in $B^p_u(X)$ for each $p\in\{0\}\cup[1,\infty]$. We also attach a sublattice of ghostly ideals in $B^p_u(X)$ to each geometric ideal. This allows us to establish connections between the ideal structures of $B^p_u(X)$ for different values of $p$ in Theorem~\ref{thm:geometric}. For bookkeeping, we adjoin a formal zero element $\mathbf{0}$ to
$\mathcal{J}[\mathcal{E}]$, corresponding to the zero ideal of
$B^p_u(X)$.

For $p\in\{0\}\cup[1,\infty]$ and $T\in B(\ell^p(X))$, we denote by $\|T\|_p$ the operator norm of $T$, and by $\|T\|_{\sup}$ the sup-norm of $T$ when $T$ is viewed as a bounded function $T(x,y)\in \ell^\infty(X\times X)$. 

\begin{lem}\label{prop:norm}
    For any $T\in B(\ell^p(X))$, we have $\|T\|_{\sup}\leq\|T\|_p$. Equality holds when $T$ is a partial translation operator. 
\end{lem}

\begin{proof}
For any $T\in B(\ell^p(X))$ and $x,y\in X$,
\[
\|T\|_p=\sup_{\|f\|_p=1}\|Tf\|_p\geq \|T\delta_y\|_p\geq|T(x,y)|.\ 
\]
Hence, $\|T\|_p\geq\|T\|_{\sup}$. When $T$ is a partial translation operator, we have observed in Example~\ref{ex:partial} that the norm of $T$ is at most $\sup_{x,y\in X}|T(x,y)|$, i.e., $\|T\|_p\leq\|T\|_{\sup}$. Therefore, $\|T\|_p=\|T\|_{\sup}$.
\end{proof}

For operators in $B^p_u(X)$, the following notions are useful for investigating their properties, especially the controlled truncations in Section~\ref{subsection2}.
\begin{lem}\label{lem:Schur}
        For any $\varphi\in\mathbb{C}_u[X,\mathcal{E}]$, we define the Schur multiplier $\mathcal{M}_\varphi:B^p_u(X)\to B^p_u(X)$ by
    \[
    \mathcal{M}_\varphi(T):=\varphi\circ T:=[\varphi(x,y)T(x, y)]_{x,y\in X }.
    \]
    Then $\mathcal{M}_\varphi$ is bounded for any $\varphi\in\mathbb{C}_u[X,\mathcal{E}]$.
\end{lem}

\begin{proof}
    Viewing elements of $B^p_u(X)$ through their matrix coefficients,
Lemma~\ref{lem:approximate1} shows that
$\mathcal{M}_\varphi(T)$ is a well-defined controlled propagation operator.
By Lemma~\ref{lem:CW}(ii), write 
    \[
    \varphi=\varphi_1+\varphi_2+\cdots+\varphi_m,
    \]
    where the $\varphi_i$ are partial translation operators. Then each
$\mathcal{M}_{\varphi_i}(T)$ is also a partial translation operator, and
    \[
    \|\mathcal{M}_\varphi(T)\|_p\leq\sum_{i=1}^m\|\mathcal{M}_{\varphi_i}(T)\|_p\leq\sum_{i=1}^m\|\varphi\|_{\sup}\cdot\|T\|_{\sup}\leq m\|\varphi\|_{\sup}\cdot\|T\|_p.
    \]
    Hence, $\mathcal{M}_\varphi$ is bounded.
\end{proof}

\begin{defn} (cf. \cite[Definition 3.1]{CW04})
    For an operator $T=[T(x,y)]_{x,y\in X}$ in $B^p_u(X, \mathcal{E})$, an entourage $E\in\mathcal{E}$, and a number $\varepsilon>0$, the $\varepsilon$-support of $T$ is defined to be 
    \[
    \supp_\varepsilon(T)=\{(x,y)\in X\times X: |T(x,y)|\geq \varepsilon\},
    \]
    and the $(E,\varepsilon)$-support of $T$ is defined to be
    \[
    \supp_{(E,\varepsilon)}(T)=E\cap\supp_\varepsilon(T).
    \]
\end{defn}

\begin{lem} (cf. \cite[Lemma 3.2]{CW04})\label{lem:supp}
    $\supp_\varepsilon(T)\in\mathcal{E}$ for any $T\in B^p_u(X, \mathcal{E})$ and $\varepsilon> 0$.
\end{lem}

\begin{proof}
    For $T\in B^p_u(X, \mathcal{E})$ and $\varepsilon> 0$, there exists $R\in\mathbb{C}_u[X,\mathcal{E}]$ such that
    \[
    \|T-R\|_{\sup}\leq\|T-R\|_p<\varepsilon/2.
    \]
    It follows that
    \[
    \supp_\varepsilon(T)\subseteq\supp(R)\in\mathcal{E}.
    \]
    Hence, $\supp_\varepsilon(T)\in\mathcal{E}$.
\end{proof}

\subsection{Controlled truncations}\label{subsection2}

\begin{defn} (cf. \cite[Definition 3.3]{CW04})
    Suppose $T\in B^p_u(X), E\in\mathcal{E}, \varepsilon>0$.
    \begin{enumerate}
    \item The $E$-truncation of $T$ is defined to be the operator $T_E=[T_E(x,y)]$ where
    \[
    T_E(x,y)=\begin{cases}
            T(x,y) &\text{if}\ (x,y)\in E,\\
		  0 &\text{otherwise}.
            \end{cases}
    \]
    \item The $\varepsilon$-truncation of $T$ is defined to be the operator $T_\varepsilon=T_{\supp_\varepsilon(T)}$. That is,
    \[
    T_\varepsilon(x,y)=\begin{cases}
        T(x,y) &\text{if}\ |T(x,y)|\geq\varepsilon,\\
			0 &\text{if}\ |T(x,y)|<\varepsilon.
    \end{cases}
    \]
    \item The $(E, \varepsilon)$-truncation of T is
    \[
    T_{(E,\varepsilon)}=T_{supp_{(E,\varepsilon)}(T)}=(T_E)_\varepsilon=(T_\varepsilon)_E
    \]
    \end{enumerate}
\end{defn}

For a subset $E\subseteq X\times X$, denote by $\chi_E$ the characteristic function of $E$. Then $\chi_E\in\mathbb{C}_u[X,\mathcal{E}]$ for all $E\in\mathcal{E}$. For $Y\subseteq X$, we continue to write $\chi_Y$ for the diagonal
operator $\rho(\chi_Y)=\chi_{\Delta_Y}$, as in
Section~\ref{subsection1}. By Lemma~\ref{lem:Schur} and Lemma~\ref{lem:supp}, we can see that $T_E, T_\varepsilon, T_{(E,\varepsilon)}\in\mathbb{C}_u[X,\mathcal{E}]$ for all $T\in B^p_u(X)$, $E\in\mathcal{E}$ and $\varepsilon>0$. 
\begin{defn}
    For an operator $T\in B^p_u(X)$, denote by $\langle\ T\ \rangle$ the closed, two-sided ideal of $B^p_u(X)$ generated by $T$. For any ideal $I$ of $B^p_u(X)$, denote
\[
    \mathbb{C}_u(I)=I\cap\mathbb{C}_u[X,\mathcal{E}],
\]
the controlled propagation operators in $I$. If $\mathbb{C}_u(I)$ is dense in $I$, we say that $I$ is a geometric ideal of $B^p_u(X)$.
\end{defn}

\begin{lem} (cf. \cite[Lemma 3.4]{CW04})\label{lem:ideal}
    Let $E\in\mathcal{E}$ be a partial translation. For any $T\in B^p_u(X)$ and any $\varepsilon>0$, we have
    \[
    \chi_{r(\supp_{(E,\varepsilon)}(T))}\in\mathbb{C}_u(\langle\ T\ \rangle).
    \]
\end{lem}

\begin{proof}
    Let $F=\supp_{(E,\varepsilon)}(T)$ and $Z=r(F)$, then we have
    \[
    \supp\chi_Z=\Delta_Z=\{(x,x): x\in Z=r(F)\}= F\circ F^{-1}.
    \]
    It follows that $\Delta_Z\in\mathcal{E}$, $\chi_Z\in\mathbb{C}_u[X,\mathcal{E}]$. It remains to prove $\chi_Z\in\langle\ T\ \rangle$. As $F\subseteq E$, $F$ is also a partial translation, i.e., for any $x\in Z=r(F)$, there exists a unique $F(x)\in X$ such that $(x,F(x))\in F$ and $|T(x,F(x))|\geq\varepsilon$. Let 
    \[
    T^*=[\overline{T(y,x)}]_{x,y\in X}.
    \]
    Even though $T^*$ may not be bounded as an operator on $\ell^p(X)$, its $F^{-1}$-truncation $T^*_{F^{-1}}$ belongs to $\mathbb{C}_u[X,\mathcal{E}]$ as $F^{-1}\in\mathcal{E}$. Let $W=\frac{1}{\varepsilon^2}TT^*_{F^{-1}}\in\langle\ T\ \rangle$. Then for any $x\in Z=r(F)$, we have
    \[
    W(x,x)=\frac{1}{\varepsilon^2}|T(x, F(x))|^2\geq1
    \]
    and as $W\in\langle\ T\ \rangle\subseteq B^p_u(X)$, there exists $R\in\mathbb{C}_u[X,\mathcal{E}]$ such that
    \[
        \|R-W\|_{\sup}\leq\|R-W\|_p\leq 1/3.
    \]
    Hence, we have $|R(x,x)|\geq 2/3$ for any $x\in Z=r(F)$. By Lemma~\ref{lem:CW}(i), the space $X$ can be partitioned into the union of finitely many $\supp(R)$-separated subspaces
    \[
        X=X_1\cup X_2\cup\cdots\cup X_m,
    \]
    where $X_i\cap X_j=\emptyset$ for $i\neq j$, and for any $i$ and any points $x,y\in X_i$ with $x\neq y$, we have $(x,y)\notin \supp(R)\cup \supp(R)^{-1}$. Thus,
    \[
    Z=(Z\cap X_1)\cup(Z\cap X_2)\cup\cdots\cup(Z\cap X_m)
    \]
    and
    \[
    \chi_Z=\chi_{Z\cap X_1}+\chi_{Z\cap X_2}+\cdots+\chi_{Z\cap X_m}.
    \]
    It remains to prove that $\chi_{Z\cap X_i}\in\langle\ T\ \rangle$, for $i=1,2,\dots,m$. Let
    \[
       \widetilde{R}=\chi_{Z\cap X_i}R\chi_{Z\cap X_i},\ \widetilde{W}=\chi_{Z\cap X_i}W\chi_{Z\cap X_i}
    \]
    then we have $\widetilde{W}\in\langle\ T\ \rangle$, and $\widetilde{R}, \widetilde{W}$ can be naturally viewed as bounded operators on $\ell^p(Z\cap X_i)$. As $Z\cap X_i$ is a $\supp(R)$-separated subspace, $\widetilde{R}$ is a diagonal operator:
    \[
       \widetilde{R}(x,y)=\begin{cases}
           R(x,y)\ & \text{if}\ x=y\in Z\cap X_i,\\
           0\ &\text{otherwise}.
       \end{cases}
    \]
    Since $|R(x,x)|\geq2/3$ for any $x\in Z\cap X_i$, we have
    \[
        2/3\leq|R(x,x)|=|\widetilde{R}(x,x)|\leq\|\widetilde{R}\|_{\sup}=\|\widetilde{R}\|_p.
    \]
    If $Z\cap X_i=\emptyset$, there is nothing to prove, so suppose that
$Z\cap X_i\neq\emptyset$. Let
$(Z\cap X_i,\mathcal{E}|_{Z\cap X_i})$ be the corresponding subspace
of $(X,\mathcal{E})$. It is unital since
$\Delta_{Z\cap X_i}\in\mathcal{E}$. Thus
$B^p_u(Z\cap X_i,\mathcal{E}|_{Z\cap X_i})$ has unit
$\chi_{Z\cap X_i}$, and $\widetilde{R}$ and $\widetilde{W}$ may be
viewed as elements of this algebra.

Since $\widetilde{R}$ is diagonal and
\[
    |\widetilde{R}(x,x)|=|R(x,x)|\geq\frac23
    \ (x\in Z\cap X_i),
\]
it is invertible, with
\[
    \|\widetilde{R}^{-1}\|_p
    =\sup_{x\in Z\cap X_i}|R(x,x)|^{-1}
    \leq\frac32.
\]
Consequently,
\[
    \|\widetilde{R}-\widetilde{W}\|_p
    \leq\|R-W\|_p
    \leq\frac13
    <\frac23
    \leq\|\widetilde{R}^{-1}\|_p^{-1}.
\]
    It follows that $\widetilde{W}$ is also invertible in $B^p_u(Z\cap X_i, \mathcal{E}|_{Z\cap X_i})$. Denote by $\widetilde{A}$ the inverse of $\widetilde{W}$ in $B^p_u(Z\cap X_i, \mathcal{E}|_{Z\cap X_i})$. When viewing both $\widetilde{W}$ and $\widetilde{A}$ as elements of $B^p_u(X)$, we have
    \[
    \chi_{Z\cap X_i}=\widetilde{A}\widetilde{W}\in\langle\ T\ \rangle.
    \]
    Hence, $\chi_Z=\sum_{i=1}^{m}\chi_{Z\cap X_i}\in\langle\ T\ \rangle$.
\end{proof}

\begin{lem} (cf. \cite[Lemma 3.4]{CW04})\label{lem:keylem}
    Let $E\in\mathcal{E}$ be a partial translation. For any $T\in B^p_u(X)$ and any $\varepsilon>0$, the operators $\chi_{\supp_{(E,\varepsilon)}(T)}, T_{(E,\varepsilon)}$, and $T_E$ belong to $\mathbb{C}_u(\langle\ T\ \rangle)$.
\end{lem}

\begin{proof}
    Let $Z=r(\supp_{(E,\varepsilon)}(T))$. By Lemma~\ref{lem:ideal}, we have $\chi_Z\in\mathbb{C}_u(\langle\ T\ \rangle)$ (in fact, $\chi_Z\in\langle\ T\ \rangle$). Since
    \begin{align*}
        \chi_{\supp_{(E,\varepsilon)}(T)}&=\chi_Z\chi_{\supp_{(E,\varepsilon)}(T)};\\
        T_{(E,\varepsilon)}&=\chi_Z T_{(E,\varepsilon)},
    \end{align*}
    we have $\chi_{\supp_{(E,\varepsilon)}(T)}, T_{(E,\varepsilon)}\in\mathbb{C}_u(\langle\ T\ \rangle)$. As $E\in\mathcal{E}$ is a partial translation,
    \[
    \|T_{(E,\varepsilon)}-T_E\|_p=\|T_{(E,\varepsilon)}-T_E\|_{\sup}\leq\varepsilon.
    \]
    Hence, $T_E=\lim\limits_{\varepsilon\to 0}T_{(E,\varepsilon)}\in\mathbb{C}_u(\langle\ T\ \rangle)$.
\end{proof}

\begin{rem}\label{rem:lims}
    From the proof of Lemma~\ref{lem:keylem}, we observe that for any $T\in\mathbb{C}_u[X,\mathcal{E}]$, $T=\lim\limits_{\varepsilon\to 0} T_\varepsilon$. In particular, $T_E=\lim\limits_{\varepsilon\to 0} T_{(E,\varepsilon)}$ for any $T\in B^p_u(X)$ and $E\in\mathcal{E}$.
\end{rem}
     
\begin{prop} (cf. \cite[Theorem 3.5]{CW04})\label{thm:keythm}
    For any $T\in B^p_u(X)$, $E\in\mathcal{E}$ and $\varepsilon>0$, the following operators belong to $\mathbb{C}_u(\langle\ T\ \rangle)$:
    \begin{enumerate}
		\item $T_{(E,\varepsilon)}, \chi_{supp_{(E,\varepsilon)}(T)}, T_E, \chi_{r(supp_{(E,\varepsilon)}(T))}$;
		\item $T_{\varepsilon},\chi_{supp_{\varepsilon}(T)},\chi_{r(supp_{\varepsilon}(T))}$;
		\item $T^*_{(E,\varepsilon)}, T^*_{E}, T^*_\varepsilon$;
		\item $\varphi\circ T$ for any $\varphi\in \mathbb{C}_u[X,\mathcal{E}]$.	
	\end{enumerate}
\end{prop}

\begin{proof}
    By Lemma~\ref{lem:CW}(ii), any entourage $E\in\mathcal{E}$ can be partitioned into the union of finitely many partial translations 
    \[
    E=E_1\cup E_2\cup \cdots\cup E_l.
    \]
    Hence,
    \begin{align*}
        \chi_{supp_{(E,\varepsilon)}(T)}&=\chi_{supp_{(E_1,\varepsilon)}(T)}+\chi_{supp_{(E_2,\varepsilon)}(T)}+\cdots+\chi_{supp_{(E_l,\varepsilon)}(T)},\\
        T_{(E,\varepsilon)}&=T_{(E_1,\varepsilon)}+T_{(E_2,\varepsilon)}+\cdots+T_{(E_l,\varepsilon)},\\
        T_E&=T_{E_1}+T_{E_2}+\cdots+T_{E_l},
    \end{align*}
    and 
    $T_E=\lim\limits_{\varepsilon\to 0}T_{(E,\varepsilon)}$ (Remark~\ref{rem:lims}).\\
    It follows from Lemma~\ref{lem:keylem} that $T_{(E,\varepsilon)}, \chi_{supp_{(E,\varepsilon)}(T)}, T_{E}\in  \mathbb{C}_u(\langle\ T\ \rangle)$. 
    Let $V=\left(\chi_{supp_{(E,\varepsilon)}(T)}\right)\left(\chi_{supp_{(E,\varepsilon)}(T)}\right)^*$ then we have $V\in\langle\ T\ \rangle\subseteq B^p_u(X)$, and
    \[
    \widetilde{\Delta}=\Delta_{r(supp_{(E,\varepsilon)}(T))}\subseteq E\circ E^{-1}
    \]
    is a partial translation. Hence, from Lemma~\ref{lem:ideal} it follows that \[
    \chi_{r(supp_{(E,\varepsilon)}(T))}=\chi_{supp_{(\widetilde{\Delta},1)}(V)}\in \mathbb{C}_u(\langle\ V\ \rangle)\subseteq \mathbb{C}_u(\langle\ T\ \rangle).
    \]
    
    Since $\supp_\varepsilon(T)\in\mathcal{E}$ (Lemma~\ref{lem:supp}), (ii) is just a special case of (i).

    As
    \begin{align*}
        T^*_{(E,\varepsilon)}&=T^*_{(E,\varepsilon)}\chi_{r(supp_{(E^{-1},\varepsilon)}(T))},\\
        T^*_\varepsilon&=T^*_\varepsilon\chi_{r(supp_{\varepsilon}(T))}
    \end{align*}
    and
    \[
    T^*_E=\lim_{\varepsilon\to 0}(T^*_E)_\varepsilon=\lim_{\varepsilon\to 0}T^*_{(E,\varepsilon)},
    \]
    we have $T^*_{(E,\varepsilon)}, T^*_{E}, T^*_\varepsilon\in  \mathbb{C}_u(\langle\ T\ \rangle)$.
    
    It remains to prove (iv). Let $\supp(\varphi)=F_1\cup F_2\cup\cdots\cup F_m$ be the union of some partial translations. Then we have
    \[
    \varphi\circ T=\varphi_{F_1}\circ T_{F_1}+\dots+\varphi_{F_m}\circ T_{F_m}.
    \]
    It suffices to show that $\varphi_{F_i}\circ T_{F_i}\in \mathbb{C}_u(\langle\ T\ \rangle)$ for all $i=1, 2,\dots, m$. For any $x\in r(F_i)$, there exists a unique $F_i(x)\in X$ such that $(x,F_i(x))\in F_i$. Define a diagonal operator $D_i=[D_i(x,y)]_{x,y\in X}$ with
    \[
    D_i(x,y)=\begin{cases}
        \varphi_{F_i}(x,F_i(x))\ &x=y\in r(F_i),\\
        0\ &\text{otherwise}.
    \end{cases}
    \]
    Then $D_i\in\mathbb{C}_u[X,\mathcal{E}]$ and $T_{F_i}\in\mathbb{C}_u(\langle\ T\ \rangle)$ (Lemma~\ref{lem:keylem}).\\
    Hence,
    \[
    \varphi_{F_i}\circ T_{F_i}=D_iT_{F_i}\in \mathbb{C}_u(\langle\ T\ \rangle).
    \]   
\end{proof}

\begin{prop} (cf. \cite[Corollary 3.6]{CW04})\label{prop:truncation}
    For any ideal $I$ of $B^p_u(X)$, we have
    \begin{enumerate}
        \item \begin{align*}
        \mathbb{C}_u(I)&=\{T_E: T\in I, E\in\mathcal{E}\}\\
        &=\{\varphi\circ T: T\in I, \varphi\in\mathbb{C}_u[X,\mathcal{E}]\};         \end{align*}
        \item If $I$ is a geometric ideal of $B^p_u(X)$ (i.e., $\mathbb{C}_u(I)$ is dense in $I$), then the collection
    \[
    \{T_{(E,\varepsilon)}: T\in I, E\in\mathcal{E}, \varepsilon>0\},
    \]
    or equivalently, the collection
    \[
    \{T_\varepsilon: T\in I, \varepsilon>0\}
    \]
    is also dense in $I$;
        \item For any controlled propagation function $\varphi\in\ell^\infty(X\times X)$ and any ideal $I$ of $B^p_u(X)$, $I$ is an invariant subspace of the Schur multiplier
    \[
    \mathcal{M}_\varphi: B^p_u(X)\to B^p_u(X);
    \]
        \item For any operator $A\in B^p_u(X)$, if there exist $T\in I$ and $\varepsilon>0$ such that $r(\supp(A))=r(\supp_\varepsilon(T))$ then we have $A\in I$. In particular, if $\supp(A)=\supp_\varepsilon(T)$, then $A\in I$.
    \end{enumerate}
\end{prop}

\begin{proof}
    (i) By Proposition~\ref{thm:keythm}(i), for every $T\in I$ and $E\in\mathcal E$,
\[
T_E\in \mathbb{C}_u(\langle\ T\ \rangle)\subseteq \mathbb{C}_u(I).
\]
Conversely, if $A\in \mathbb{C}_u(I)$, then
$E:=\operatorname{supp}(A)\in\mathcal E$ and $A=A_E$. Hence
\[
\mathbb{C}_u(I)=\{T_E:T\in I,\ E\in\mathcal E\}.
\]

    Similarly, Proposition~\ref{thm:keythm}(iv) gives
\[
\varphi\circ T\in \mathbb{C}_u(\langle \ T\ \rangle)\subseteq \mathbb{C}_u(I)
\]
for $T\in I$ and $\varphi\in \mathbb{C}_u[X,\mathcal E]$. Conversely,
\[
T_E=\chi_E\circ T,
\ \chi_E\in \mathbb{C}_u[X,\mathcal E].
\]
Therefore
\[
\mathbb{C}_u(I)
 =\{T_E:T\in I,\ E\in\mathcal E\}
 =\{\varphi\circ T:T\in I,\ \varphi\in \mathbb{C}_u[X,\mathcal E]\}.
\]
  
    (ii) For $T\in I$, $E\in\mathcal E$, and $\varepsilon>0$,
\[
T_{(E,\varepsilon)}=(T_E)_\varepsilon,
\]
and $T_E\in I$ by part~(i). Conversely, for $T\in I$ and
$\varepsilon>0$, Lemma~\ref{lem:supp} gives
\[
E:=\operatorname{supp}_\varepsilon(T)\in\mathcal E,
\]
and
\[
T_\varepsilon=T_{(E,\varepsilon)}.
\]
Thus
\[
\{T_{(E,\varepsilon)}:
  T\in I,\ E\in\mathcal E,\ \varepsilon>0\}
=
\{T_\varepsilon:T\in I,\ \varepsilon>0\}.
\]
    Now let $A\in \mathbb{C}_u(I)$. Since
$E:=\operatorname{supp}(A)\in\mathcal E$, Remark~\ref{rem:lims} gives $A=\lim_{\varepsilon\to0}A_\varepsilon$.
Hence the common set above is dense in $\mathbb{C}_u(I)$. If $I$ is
geometric, then $\mathbb{C}_u(I)$ is dense in $I$, and therefore the same
set is dense in $I$.

    (iii) For any $T\in I\subseteq B^p_u(X)$, it follows from Proposition~\ref{thm:keythm}(iv) that $\varphi\circ T\in\langle\ T\ \rangle\subseteq I$. Hence $\mathcal{M}_\varphi(I)\subseteq I$.

    (iv) It follows from Proposition~\ref{thm:keythm}(ii) that 
    \[
    \chi_{r(\supp(A))}=\chi_{r(\supp_\varepsilon(T))}\in\mathbb{C}_u(\langle\ T\ \rangle)\subseteq I.
    \]
    Thus,
    \[
    A=\chi_{r(\supp(A))}A\in I.
    \]
\end{proof}

\subsection{Geometric ideals in $\ell^p$ uniform Roe algebras}

\begin{lem}\label{lem:finite}
    For every nonzero ideal $I$ of $B^p_u(X)$, we have $\chi_F\in I$ for any finite subset $F\subseteq X\times X$. In particular, $\chi_Y\in I$ for any finite subset $Y\subseteq X$. 
\end{lem}

\begin{proof}
    For $x_i,x_j\in X$, set
$F_{ij}:=\{(x_i,x_j)\}\in\mathcal{E}$. Then
$\chi_{F_{ij}}\in\mathbb{C}_u[X,\mathcal{E}]$.
Since $I$ is nonzero, there is a nonzero operator $T=[T(x,y)]_{x,y\in X}\in I$. Assuming $T(x_i,x_j)\neq 0$, we have
    \[
    \chi_{F_{mn}}=\frac{1}{T(x_i,x_j)}\chi_{F_{mi}}T\chi_{F_{jn}}\in I
    \] 
    for any $x_m,x_n\in X$.
    
    Thus, for any finite subset $F\subseteq X\times X$, we have 
    \[
    \chi_F=\sum_{(x_m,x_n)\in F}\chi_{F_{mn}}\in I.
    \]
\end{proof}

\begin{defn}
    Let $I[B^p_u(X)]$ denote the lattice of all ideals in $B^p_u(X)$, and let $I_0[B^p_u(X)]$ denote the sublattice of all the geometric ideals. 
\end{defn}

Let $I$ be an ideal of $B^p_u(X)$, and define
\[
\mathcal{J}(I)=\begin{cases}
    \mathbf{0} &I=0,\\
    \{\supp_\varepsilon(T): T\in I, \varepsilon>0\} &\text{otherwise}.
\end{cases}
\]
For the nonzero ideals $I$, by Lemma~\ref{lem:supp} and the proof of Proposition~\ref{prop:truncation}(ii), we have
\begin{align*}
    \mathcal{J}(I)&=\{\supp_\varepsilon(T): T\in\mathbb{C}_u(I), \varepsilon>0\}\\
    &=\{\supp_{(E,\varepsilon)}(T): T\in I, E\in\mathcal{E}, \varepsilon>0\}
\end{align*}

\begin{prop}\label{prop:lattice1}
    For a nonzero ideal $I$, $\mathcal{J}(I)$ is an ideal of $\mathcal{E}$, and the correspondence
    \begin{align*}
        \lambda_3: I[B^p_u(X)]&\to\mathcal{J}[\mathcal{E}]\\
        I&\mapsto \mathcal{J}(I)
    \end{align*}
    is an order-preserving map.
\end{prop}

\begin{proof}
    Let $T,T'\in I$, let $\varepsilon>0$, and suppose
$B\in\mathcal{E}$ satisfies
\[
B\subseteq A:=\supp_\varepsilon(T).
\] It follows from Proposition~\ref{thm:keythm} that
    \[
    \chi_{\supp_\varepsilon(T)}, T_B, T^*_\varepsilon\in\mathbb{C}_u(\langle\ T\ \rangle)\subseteq I.
    \]
    Hence,
    \begin{align*}
        \supp_\varepsilon(T)^{-1}&=\supp_\varepsilon(T^*_\varepsilon)\in\mathcal{J}(I);\\
        \supp_\varepsilon(T)\circ\supp_\varepsilon(T')&=\supp_1(\chi_{\supp_\varepsilon(T)}\chi_{\supp_\varepsilon(T')})\in\mathcal{J}(I);\\
        \supp_\varepsilon(T)\cup\supp_\varepsilon(T')&=\supp_1(\chi_{\supp_\varepsilon(T)}+\chi_{\supp_\varepsilon(T')})\in\mathcal{J}(I);\\
        B=\supp(T_B)&=\supp_{\varepsilon}(T_B)\in\mathcal{J}(I),
    \end{align*}
    and by Lemma~\ref{lem:finite} we have
    \[
    F=\supp_1(\chi_F)\in\mathcal{J}(I),
    \]
    for any finite subset $F\subseteq X\times X$. Thus, $\mathcal{J}(I)$ is a coarse structure on $X$ contained in $\mathcal{E}$, so it remains to show that it is an ideal of $\mathcal{E}$, i.e.,
    for any $A=\supp_\varepsilon(T)\in\mathcal{J}(I)$ and any entourage $E\in\mathcal{E}$, we have $E\circ A, A\circ E\in\mathcal{J}(I)$.

    By Lemma~\ref{lem:CW}(ii) and $\mathcal{J}(I)$ being closed under taking finite unions, without loss of generality, we assume that $E$ is a partial translation. Then we have $\chi_ET, T\chi_E\in I$ and
    \begin{align*}
        E\circ A&=\supp_\varepsilon(\chi_ET)\in\mathcal{J}(I),\\
        A\circ E&=\supp_\varepsilon(T\chi_E)\in\mathcal{J}(I).
    \end{align*}
    
    It is easy to check that this correspondence is order-preserving.
\end{proof}

Denote by $\ell^\infty_\mathcal{E}(X)$ the collection of diagonal operators in $B^p_u(X)$, i.e., $\ell^\infty_\mathcal{E}(X)=B^p_u(X)\cap\ell^\infty(X)$. 

\begin{cor}\label{cor:innerouter}
    For a nonzero ideal $I$ in $B^p_u(X)$, let $\mathcal{C}_I:=\{\supp_\varepsilon(D):D\in I\cap\ell^\infty_\mathcal{E}(X), \varepsilon>0\}$. The coarse ideal $\mathcal{J}(I)$ is generated by $\mathcal{C}_I$.
\end{cor}

\begin{proof}
    By definition, we have $\mathcal{C}_I\subseteq\mathcal{J}(I)$. Let $T\in I$ and $\varepsilon>0$, and set
$E:=\supp_\varepsilon(T)\in\mathcal{E}$. It follows from Proposition~\ref{thm:keythm}(ii) that 
    \[
    \chi_{r(\supp_{\varepsilon}(T))}=\chi_{r(E)}\in I\cap\ell^\infty_\mathcal{E}(X)
    \]
    and $\supp_1\chi_{r(E)}=\Delta_{r(E)}\in\mathcal{C}_I$. Then we have
    \[
    E=\Delta_{r(E)}\circ E,
    \]
    which completes the proof. 
\end{proof}

\begin{rem}\label{rem:otherideals}
    $\lambda_3: I[B^p_u(X)]\to\mathcal{J}[\mathcal{E}]$ may not be injective, but Propositions~\ref{prop:lattice2} and~\ref{prop:lattice3} show that it is a bijection when restricted to $I_0[B^p_u(X)]$.
\end{rem}

Conversely, set $I(\mathbf{0})=0$. Every ideal $\mathcal{J}$ of
$\mathcal{E}$ gives rise to a nonzero ideal of $B^p_u(X)$ as follows. By definition, $\mathbb{C}_u[X,\mathcal{J}]=\{ T\in B^p_u(X): \supp(T)\in\mathcal{J}\}$.
Let $I(\mathcal{J})$ be the closure of $\mathbb{C}_u[X,\mathcal{J}]$ in $B^p_u(X)$:
\[
I(\mathcal{J})=\overline{\mathbb{C}_u[X,\mathcal{J}]}^{\|\cdot\|_p}.
\]

\begin{prop}\label{prop:lattice2}
    $I(\mathcal{J})$ is a geometric ideal of $B^p_u(X)$ and the correspondence
    \begin{align*}
        \lambda_4: \mathcal{J}[\mathcal{E}]&\to I_0[B^p_u(X)]\\
        \mathcal{J}&\mapsto I(\mathcal{J})
    \end{align*}
    is an order-preserving map.
\end{prop}

\begin{proof}
    It is clear that $\mathbb{C}_u[X,\mathcal{J}]$ is a subalgebra of $B^p_u(X)$ and $I(\mathcal{J})$ is a closed subalgebra. For any $\varepsilon>0$, $T\in I(\mathcal{J})$ and $T'\in B^p_u(X)$, there exist $B\in\mathbb{C}_u[X,\mathcal{J}]$ and $B'\in\mathbb{C}_u[X,\mathcal{E}]$ such that 
    \[
    \|T-B\|_p\leq\varepsilon,\|T'-B'\|_p\leq\varepsilon.
    \]
    Then we have
    \begin{align*}
        \|TT'-BB'\|_p&\leq\|TT'-TB'\|_p+\|TB'-BB'\|_p\\
        &\leq\|T\|_p\|T'-B'\|_p+\|T-B\|_p\|B'\|_p\\
        &\leq(\|T\|_p+\|T'\|_p)\varepsilon+\varepsilon^2
    \end{align*}
    and
    \[\supp(BB')\subseteq\supp(B)\circ\supp(B')\in\mathcal{J}.
    \]
    Hence, $BB'\in\mathbb{C}_u[X,\mathcal{J}]$ and then $TT'\in I(\mathcal{J})$. Similarly, $T'T\in I(\mathcal{J})$. Thus, $I(\mathcal{J})$ is an ideal of $B^p_u(X)$ and by definition is a geometric ideal.

    It is easy to check that this correspondence is order-preserving.
\end{proof}

\begin{rem}\label{rem:ideal}
    Let $\mathcal{J}$ be an ideal of $\mathcal{E}$.
    \begin{enumerate}
    \item By definition, $\mathbb{C}_u[X,\mathcal{J}]$ is an ideal of
$\mathbb{C}_u[X,\mathcal{E}]$.
    
    \item We have
\[
I(\mathcal{J})=B^p_u(X,\mathcal{J}),
\]
so $I(\mathcal{J})$ is both a geometric ideal of $B^p_u(X)$ and the
$\ell^p$ uniform Roe algebra associated with the coarse structure
$\mathcal{J}\subseteq\mathcal{E}$.
    \end{enumerate}
\end{rem}

A natural question is whether every nonzero geometric ideal of $B^p_u(X,\mathcal{E})$ is itself an $\ell^p$ uniform Roe algebra $B^p_u(X, \mathcal{J})$ for some ideal $\mathcal{J}$ of $\mathcal{E}$. We have the following:

\begin{prop} (cf. \cite[Proposition 4.2]{CW04})\label{prop:lattice3}
    There is a one-to-one correspondence
    \[
    \begin{tikzcd}
  \{\text{ideals of the coarse structure $\mathcal{E}$}\} \arrow[d, "\lambda_4", shift left=1ex] \\
  \{\text{nonzero geometric ideals of the $\ell^p$ uniform Roe algebra $B^p_u(X)$}\} \arrow[u, "\lambda_3", shift left=1ex]
    \end{tikzcd}
    \]
    i.e., $\lambda_3|_{I_0[B^p_u(X)]}\lambda_4=id_{\mathcal{J}[\mathcal{E}]}$ and $\lambda_4\lambda_3|_{I_0[B^p_u(X)]}=id_{I_0[B^p_u(X)]}$. Equivalently, for any ideal $\mathcal{J}$ of $\mathcal{E}$, we have
    \[
    \mathcal{J}=\mathcal{J}(I(\mathcal{J})),
    \]
    and for any nonzero geometric ideal $I$ of $B^p_u(X)$, we have
    \[
    I=I(\mathcal{J}(I)).
    \]
\end{prop}
\begin{proof}
    For any $E\in\mathcal{J}\subseteq\mathcal{E}$, we have $\chi_E\in\mathbb{C}_u[X,\mathcal{J}]\subseteq I(\mathcal{J})$. Since $E=\supp_1(\chi_E)\in\mathcal{J}(I(\mathcal{J}))$, it follows that $\mathcal{J}\subseteq\mathcal{J}(I(\mathcal{J}))$.

    For any $E'\in\mathcal{J}(I(\mathcal{J}))$, there exist $T\in I(\mathcal{J})$ and $\varepsilon>0$ such that
    \[
    \supp_\varepsilon(T)=E'.
    \] It follows from Proposition~\ref{thm:keythm}(ii) that
    \[
    \chi_{E'}=\chi_{\supp_\varepsilon(T)}\in\mathbb{C}_u(\langle\ T\ \rangle)\subseteq I(\mathcal{J})
    \]
    and there exists $B\in\mathbb{C}_u[X,\mathcal{J}]$ such that
    \[
    \|\chi_{E'}-B\|_{\sup}\leq\|\chi_{E'}-B\|_p\leq 1/2.
    \]
    Thus, $E'\subseteq\supp(B)\in\mathcal{J}$,
    i.e., $\mathcal{J}(I(\mathcal{J}))\subseteq\mathcal{J}$. Hence, $\mathcal{J}=\mathcal{J}(I(\mathcal{J}))$.

    For any $T\in I, \varepsilon>0$, we have
    \[
    \supp(T_\varepsilon)=\supp_\varepsilon(T)\in\mathcal{J}(I).
    \]
    Then $T_\varepsilon\in\mathbb{C}_u[X,\mathcal{J}(I)]\subseteq I(\mathcal{J}(I))$,
    i.e.,
    \[
    \{T_\varepsilon: T\in I, \varepsilon>0\}\subseteq\mathbb{C}_u[X,\mathcal{J}(I)].
    \]
    Since $I$ is a geometric ideal, $\mathbb{C}_u(I)=\mathbb{C}_u[X,\mathcal{E}]\cap I$ is dense in $I$, and it follows from Proposition~\ref{prop:truncation}(ii) that $\{T_\varepsilon: T\in I, \varepsilon>0\}$ is dense in $I$.
    Thus,
    \[
    I=\overline{\{T_\varepsilon: T\in I, \varepsilon>0\}}^{\|\cdot\|_p}\subseteq\overline{\mathbb{C}_u[X,\mathcal{J}(I)]}^{\|\cdot\|_p}=I(\mathcal{J}(I)).
    \]
    It remains to prove $I(\mathcal{J}(I))\subseteq I$. For any $T\in\mathbb{C}_u[X,\mathcal{J}(I)]$, by definition, $\supp(T)\in \mathcal{J}(I)$. Then there exist $M\in I$ and $\varepsilon>0$ such that $\supp(T)=\supp_\varepsilon(M)$. It follows from Proposition~\ref{prop:truncation}(iv) that $T\in I$. Therefore, $\mathbb{C}_u[X,\mathcal{J}(I)]\subseteq I$, which implies that $I(\mathcal{J}(I))\subseteq I$.
\end{proof}

Thus the nonzero geometric ideals of $B^p_u(X,\mathcal{E})$ are precisely
the algebras $B^p_u(X,\mathcal{J})$, as $\mathcal{J}$ ranges over the
ideals of $\mathcal{E}$.

\begin{cor}\label{cor:geome}
The geometric ideals of $B^p_u(X,\mathcal{E})$ are exactly the
    ideals generated by their intersection with
    $\ell^\infty_\mathcal{E}(X)$.
\end{cor}

\begin{proof}
Suppose first that $I$ is geometric, and let $J$ be the closed,
    two-sided ideal generated by
    $I\cap\ell^\infty_\mathcal{E}(X)$. Clearly $J\subseteq I$.

    Let $T\in\mathbb{C}_u(I)$. For every $\varepsilon>0$,
    Proposition~\ref{thm:keythm}(ii) gives
    \[
        \chi_{r(\supp_\varepsilon(T))}
        \in I\cap\ell^\infty_\mathcal{E}(X)
        \subseteq J.
    \]
    Since
    \[
        T_\varepsilon
        =\chi_{r(\supp_\varepsilon(T))}T_\varepsilon,
    \]
    we have $T_\varepsilon\in J$. By Remark~\ref{rem:lims},
    $T_\varepsilon\to T$ in operator norm, so $T\in J$. Thus
    $\mathbb{C}_u(I)\subseteq J$. Since $I$ is geometric,
    $\mathbb{C}_u(I)$ is dense in $I$, and hence $I\subseteq J$.
    Therefore $I=J$.

    Conversely, suppose that $I$ is generated by
    $I\cap\ell^\infty_\mathcal{E}(X)$, and set
    \[
        J:=\overline{\mathbb{C}_u(I)}^{\|\cdot\|_p}.
    \]
    We first observe that $J$ is a closed, two-sided ideal of
    $B^p_u(X)$. Indeed, if
    $A\in\mathbb{C}_u(I)$ and $S\in B^p_u(X)$, choose
    $S_n\in\mathbb{C}_u[X,\mathcal{E}]$ with $S_n\to S$. Then
    \[
        S_nA,\ AS_n\in\mathbb{C}_u(I)
    \]
    for every $n$, so $SA,AS\in J$. Passing to norm limits shows that
    $J$ is a two-sided ideal.

    Now let $D\in I\cap\ell^\infty_\mathcal{E}(X)$. By Proposition~\ref{thm:keythm}(ii),
    \[
        D_\varepsilon\in\mathbb{C}_u(I)
    \]
    for every $\varepsilon>0$, and, since $D$ is diagonal,
    \[
        \|D-D_\varepsilon\|_p
        =\|D-D_\varepsilon\|_{\sup}\leq\varepsilon.
    \]
    Hence $D\in J$. Thus $J$ contains
    $I\cap\ell^\infty_\mathcal{E}(X)$, and therefore contains the ideal
    generated by this intersection, namely $I$. Since $J\subseteq I$
    by definition, $I=J$. Hence $I$ is geometric.
\end{proof}

\begin{thm}\label{thm:geometric}
For every $p\in\{0\}\cup[1,\infty]$:
\begin{enumerate}
\item the map
\[
\lambda_4:\mathcal{J}[\mathcal{E}]
\to I_0[B_u^p(X)]
\]
is a lattice isomorphism;
\item for any $p,q\in\{0\}\cup[1,\infty]$, the lattices
$I_0[B_u^p(X)]$ and $I_0[B_u^q(X)]$ are canonically isomorphic.
\end{enumerate}
In particular, every geometric ideal lattice is canonically isomorphic to
that of $C_u^*(X)=B_u^2(X)$.
\end{thm}

\begin{proof}
    (i) It follows from Propositions~\ref{prop:lattice1},~\ref{prop:lattice2}, and~\ref{prop:lattice3} that $\lambda_4$ is a lattice isomorphism from $\mathcal{J}[\mathcal{E}]$ to $I_0[B^p_u(X)]$.
    
    (ii) For $p,q\in[1,\infty]$, the canonical isomorphism is
    \[
    \lambda_4\lambda_3:
    I_0[B_u^p(X)]\to I_0[B_u^q(X)],
    \]
    so it sends the geometric ideal $B_u^p(X,\mathcal{J})$ to
    $B_u^q(X,\mathcal{J})$ for every
    $\mathcal{J}\in\mathcal{J}[\mathcal{E}]$, and sends zero to zero.
    For $p=0$, use the canonical isometric isomorphism
    \[
    \Gamma:B_u^0(X)\to B_u^\infty(X),\ T\mapsto T^{''},
    \]
    from Section~\ref{subsection1}. Since $\Gamma$ preserves the matrix of a controlled propagation operator, it carries
    $B_u^0(X,\mathcal{J})$ onto $B_u^\infty(X,\mathcal{J})$.
    Composing $\Gamma$ (or $\Gamma^{-1}$) with the preceding lattice isomorphisms gives the assertion whenever one of the exponents is zero. 
\end{proof}

    Concretely, the correspondence between geometric ideals is:
    \[       \overline{\mathbb{C}_u[X,\mathcal{J}]}^{\|\cdot\|_{p}}=B^p_u(X,\mathcal{J})\longleftrightarrow B^q_u(X,\mathcal{J})=\overline{\mathbb{C}_u[X,\mathcal{J}]}^{\|\cdot\|_{q}}
    \]
    for every ideal $\mathcal{J}$ of $\mathcal{E}$.

\begin{rem}
   Here $\mathbb{C}_u[X,\mathcal{E}]$ is used uniformly across different
exponents $p$. For different $p$ it represents operators on different $\ell^p$ spaces, all of which have the same controlled propagation matrix form with respect to $\{\delta_x:x\in X\}$. 
\end{rem}

\subsection{Ghostly ideals in $\ell^p$ uniform Roe algebras}

For $p\in\{0\}\cup[1,\infty]$, any operator $T\in B^p_u(X)$ with matrix form $T=[T(x,y)]_{x,y\in X}$ can be considered a function in $\ell^\infty(X\times X)$. We denote by $C_0(X\times X)\subset\ell^\infty(X\times X)$ the functions vanishing at infinity. When $T\in C_0(X\times X)$, we say that the operator $T$ is a ghost element in $B^p_u(X)$. Denote by $I^p_G(X)$ the collection of all ghosts in $B^p_u(X)$. Then $I^p_G(X)$ is an ideal of $B^p_u(X)$. An ideal $I$ of $B^p_u(X)$ is called a ghost ideal if every element of $I$ is a ghost.

Propositions~\ref{prop:lattice1} and~\ref{prop:lattice3} give an
order-preserving surjection from the lattice $I[B^p_u(X)]$ of all ideals in $B^p_u(X)$ to the lattice $\mathcal{J}[\mathcal{E}]$. Thus, we can divide the entire lattice $I[B^p_u(X)]$ into disjoint blocks of sublattices 
\[
I[\mathcal{J}]:=\{I\in I[B^p_u(X)]:\mathcal{J}(I)=\mathcal{J}\},
\]
for each $\mathcal{J}\in\mathcal{J}[\mathcal{E}]$. It remains to understand the lattice structure of $I[\mathcal{J}]$.

\begin{defn}
    For any ideal $I$ of $B^p_u(X)$, we define the essential support of $I$ by
    \[
    \essupp(I):=\mathcal{J}(I).
    \]
\end{defn} 
By definition, $\essupp(0)=\mathbf{0}$. It follows from Proposition~\ref{prop:lattice1} that the essential support $\essupp(I)$ of a nonzero ideal $I$ is an ideal of the underlying coarse structure $\mathcal{E}$, and each sublattice $I[\mathcal{J}]$ is the collection of all ideals in $B^p_u(X)$ with essential support $\mathcal{J}$. From Corollary~\ref{cor:innerouter}, we can see that if two ideals have the same intersection with $\ell^\infty_\mathcal{E}(X)$ then they have the same essential supports.

Moreover, from the proof of Proposition~\ref{prop:lattice3} we can see that the geometric ideal $I(\mathcal{J})$ associated to $\mathcal{J}$ is the smallest element in $I[\mathcal{J}]$.
We may then naturally ask whether there exists a largest element in $I[\mathcal{J}]$. For $\mathcal{J}\in\mathcal{J}[\mathcal{E}]\setminus\{\mathbf{0}\}$, denote by 
    \[
    \widetilde{I}(\mathcal{J}):=\{T\in B^p_u(X):\supp_\varepsilon(T)\in\mathcal{J}\ \text{for every}\ \varepsilon>0\},
    \]
the collection of operators whose $\varepsilon$-support lies in
$\mathcal{J}$ for every $\varepsilon>0$ and let $\widetilde{I}(\mathbf{0})=0$. Then we have:

\begin{prop}\label{prop:ghostly}
    $\widetilde{I}(\mathcal{J})$ is an ideal in $B^p_u(X)$ with essential support 
    \[
    \essupp(\widetilde{I}(\mathcal{J}))=\mathcal{J},
    \]
    which we call a ghostly ideal associated to $\mathcal{J}$. Moreover, $\widetilde{I}(\mathcal{J})$ is the largest element in $I[\mathcal{J}]$.
\end{prop}

\begin{proof}
    When $\mathcal{J}=\mathbf{0}$, this is obvious. For $\mathcal{J}\in\mathcal{J}[\mathcal{E}]\setminus\{\mathbf{0}\}$, it is clear that $\widetilde{I}(\mathcal{J})$ is a closed linear subspace of $B^p_u(X)$. Given $T\in\widetilde{I}(\mathcal{J})$ and $S\in B^p_u(X)$, we need to show $TS,ST\in\widetilde{I}(\mathcal{J})$. As $\widetilde{I}(\mathcal{J})$ is closed, we only need to consider $S\in \mathbb{C}_u[X]$. Without loss of generality, assume that $\supp(S)$ is a partial
translation. Then $\|S\|_{\sup}=\|S\|_p$.
    
    Since $\mathcal{J}$ is an ideal of $\mathcal{E}$, we have 
    \[
    \supp_\varepsilon(TS)\subseteq\supp_{\varepsilon/\|S\|_p}(T)\circ\supp(S)\in\mathcal{J},
    \]
    \[
    \supp_\varepsilon(ST)\subseteq\supp(S)\circ\supp_{\varepsilon/\|S\|_p}(T)\in\mathcal{J},
    \]
    for every $\varepsilon>0$. Thus, $TS,ST\in\widetilde{I}(\mathcal{J})$, $\widetilde{I}(\mathcal{J})$ is an ideal in $B^p_u(X)$. Since $I(\mathcal{J})\subseteq\widetilde{I}(\mathcal{J})$, by definition, we have $\essupp(\widetilde{I}(\mathcal{J}))=\mathcal{J}(\widetilde{I}(\mathcal{J}))=\mathcal{J}$ and any ideal $I$ in $I[\mathcal{J}]$ is contained in $\widetilde{I}(\mathcal{J})$. 
\end{proof}

\begin{rem}
The definition of ghostly ideals is inspired by \cite{WZ25}, which treats
the case $p=2$ for uniformly discrete metric spaces of bounded geometry.
Although the present definition differs in form from
\cite[Definition 4.1]{WZ25}, Proposition~\ref{prop:restrictive} shows that
the two notions agree.
\end{rem}

\section{Relating to ideals in groupoid $L^p$ operator algebras}\label{section5}

In this section, we relate our results to the ideal theory of groupoid $L^p$ operator algebras when $p\in[1,\infty]$.

For a locally compact Hausdorff \'{e}tale groupoid $G$, let $G^{(0)}$
denote its unit space and let $C_c(G)$ be the space of compactly supported continuous functions on $G$.
If $u\in G^{(0)}$, then we denote by $G_u$ the source fiber $s^{-1}(u)$ and by $\lambda_u$ the representation of $C_c(G)$ on $\ell^p(G_u)$ given by
\[ \lambda_u(f)\xi(\alpha)=\sum_{\gamma\in G_u}f(\alpha\gamma^{-1})\xi(\gamma) \]
for $f\in C_c(G)$, $\xi\in\ell^p(G_u)$, and $\alpha\in G_u$.

The $p$-reduced norm on $C_c(G)$ is defined to be
\[ \Vert f\Vert_{p,red}:=\sup_{u\in G^{(0)}}\Vert \lambda_u(f)\Vert_{B(\ell^p(G_u))}, \]
and the reduced $L^p$ operator algebra of $G$, denoted by $F^p_{red}(G)$, is defined to be the completion of $C_c(G)$ with respect to the $p$-reduced norm.

A subset $U$ of the unit space $G^{(0)}$ is called invariant if
every element of $G$ whose range belongs to $U$ also has its source in $U$.

\subsection{Isometric isomorphism between $B^p_u(X)$ and $F^p_{red}(G(X))$}

We now turn to the coarse groupoid $G(X)$ associated with
$(X,\mathcal{E})$. Before identifying the $\ell^p$ uniform Roe algebra
$B^p_u(X)$ with the reduced $L^p$ operator algebra
$F^p_{\mathrm{red}}(G(X))$, we recall some basic properties of $G(X)$ and
the necessary notions from limit coarse spaces and limit operators. The
details are collected in Appendix~\ref{appen2}.

For each non-principal ultrafilter $\omega$ (i.e., $\omega\in\partial G(X)^{(0)}:=G(X)^{(0)}\setminus X$) in $G(X)^{(0)}$, we denote by $(X(\omega),\widetilde{\mathcal{E}_\omega})$ the limit coarse space of $(X,\mathcal{E})$ at $\omega$ (see Definition~\ref{def:limitcoarse2}) and by $\Phi_\omega: B^p_u(X,\mathcal{E})\to B^p_u(X(\omega),\widetilde{\mathcal{E}_\omega})$ the homomorphism that sends the operator $T$ to its limit operator $\Phi_\omega(T)$ at $\omega$ (see Proposition~\ref{lem:norm1}).

By \cite[Proposition 10.28]{Roe03}, there is an algebra isomorphism 
\begin{align*}
\Psi_X:C_c(G(X))&\to \mathbb{C}_u[X,\mathcal{E}]\\
        f&\mapsto T_{f}
\end{align*}
given as follows:
A function $f\in C_c(G(X))$ with compact support on $G(X)$ is, by definition, a continuous function on some $\overline{E}$, where $E$ is a controlled set. By the Stone--\v{C}ech extension property, such a function is equivalently
determined by a bounded function on $E\subseteq X\times X$, and hence by
an element of $\mathbb{C}_u[X,\mathcal{E}]$. 

\begin{prop}\label{prop:isometric}
The algebra isomorphism $\Psi_X:C_c(G(X))\to\mathbb{C}_u[X,\mathcal{E}]$ naturally extends to an isometric isomorphism $\Psi:F^p_{red}(G(X))\to B^p_u(X)$ for $p\in[1,\infty]$.
\end{prop}

\begin{proof}
    Proposition~\ref{lem:norm1}, Corollary~\ref{lem:norm2}, and
Remark~\ref{rem:norm}(i)--(ii) give us
    \[
    ||\lambda_\omega(f)||=||\Phi_\omega(\Psi_X(f))||\leq||\Psi_X(f)||=||\lambda_x(f)||
    \]
    for all $f\in C_c(G(X))$, $\omega\in \partial G(X)^{(0)}$, and $x\in X$.
    Hence \[ ||f||_{p,red}=\sup_{\omega\in G(X)^{(0)}}||\lambda_\omega(f)||=\sup_{x\in X}||\lambda_x(f)||=||\Psi_X(f)|| \] for all $f\in C_c(G(X))$.

    Taking norm closures therefore extends $\Psi_X$ to the asserted isometric
isomorphism of Banach algebras.
\end{proof}
Alternatively, for a locally compact Hausdorff \'{e}tale groupoid $G$ and $f\in C_c(G)$, the map $G^{(0)}\to\mathbb{R},\ u\mapsto\|\lambda_u(f)\|$ is lower semicontinuous. Since $X$ is dense in $G(X)^{(0)}\subseteq\beta X$, we immediately have $\sup_{\omega\in G(X)^{(0)}}||\lambda_\omega(f)||=\sup_{x\in X}||\lambda_x(f)||$. We retain the above limit-operator proof because it will also be used to analyze the boundary $\partial G(X)^{(0)}$ (see Appendix~\ref{appen2}).

\begin{lem}\label{lem:diagonal-identification}
    Under the isometric isomorphism
    \[
        \Psi:F^p_{red}(G(X))\to B^p_u(X),
    \]
    we have
    \[
        \Psi\big(C_0(G(X)^{(0)})\big)
        =\ell^\infty_\mathcal{E}(X).
    \]
\end{lem}

\begin{proof}
    Let
    \[
        \mathbb{C}_{u,\Delta}[X,\mathcal{E}]
        :=\mathbb{C}_u[X,\mathcal{E}]\cap\ell^\infty(X)
    \]
    denote the controlled propagation diagonal operators. We first
    claim that
    \[
        \Psi_X\big(C_c(G(X)^{(0)})\big)
        =\mathbb{C}_{u,\Delta}[X,\mathcal{E}].
    \]

    Indeed, if $f\in C_c(G(X)^{(0)})$, then $\Psi_X(f)$ is diagonal,
    and compactness of $\supp(f)$ gives controlled propagation. Thus       $\Psi_X(f)\in\mathbb{C}_{u,\Delta}[X,\mathcal{E}]$.
    Conversely, let $D\in\mathbb{C}_{u,\Delta}[X,\mathcal{E}]$, and
    write
    $\supp(D)=\Delta_A$
    for some $A\subseteq X$. Since $\Delta_A\in\mathcal{E}$, the
    closure $\overline{A}$ is a compact open subset of
    $G(X)^{(0)}$. The bounded function
    \[
        A\to\mathbb{C},\ x\mapsto D(x,x),
    \]
    extends continuously to $\overline{A}$ by the Stone--\v{C}ech
    property. Extending this function by zero outside $\overline{A}$
    gives an element $f\in C_c(G(X)^{(0)})$ satisfying
    $\Psi_X(f)=D$. This proves the claim.

    The reduced norm of a function supported on $G(X)^{(0)}$ is its
    supremum norm, so
    \[
        \overline{C_c(G(X)^{(0)})}^{\|\cdot\|_{p,red}}
        =C_0(G(X)^{(0)}).
    \]
    It remains only to observe that
    \[
        \overline{\mathbb{C}_{u,\Delta}[X,\mathcal{E}]}^{\|\cdot\|_p}
        =\ell^\infty_\mathcal{E}(X).
    \]
    One inclusion is immediate. For the other, let
    $D\in\ell^\infty_\mathcal{E}(X)$ and choose
    $A_n\in\mathbb{C}_u[X,\mathcal{E}]$ with
    $\|A_n-D\|_p\to0$. Let $D_n$ be the diagonal part of $A_n$.
    Then
    $D_n\in\mathbb{C}_{u,\Delta}[X,\mathcal{E}]$
    and
    \[
        \|D-D_n\|_p
        =\sup_{x\in X}|D(x,x)-A_n(x,x)|
        \leq\|D-A_n\|_p\to0.
    \]
    Taking norm closures in the algebraic identity above now gives
    \[
        \Psi\big(C_0(G(X)^{(0)})\big)
        =\ell^\infty_\mathcal{E}(X).
    \]
\end{proof}

\begin{rem}
If $\mathcal E$ is unital, then
$G(X)^{(0)}=\beta X$ and
$\ell^\infty_{\mathcal E}(X)=\ell^\infty(X)$. Hence
$F^p_{\mathrm{red}}(G(X))$ and $B^p_u(X)$ are unital and
\[
\Psi\big(C(\beta X)\big)=\ell^\infty(X).
\]
For $p\in[1,\infty)\setminus\{2\}$, this also agrees with the
$C^*$-cores:
\[
\operatorname{core}\big(F^p_{\mathrm{red}}(G(X))\big)
=C(\beta X),
\ \operatorname{core}\big(B^p_u(X)\big)=\ell^\infty(X);
\]
see \cite[Propositions~2.7 and~5.1]{CGT}.
\end{rem}

\subsection{Dynamical ideals in $F^p_{red}(G)$}

Having established the above isometric isomorphism, it is natural to relate geometric ideals in $B^p_u(X)$ to their counterparts in $F^p_{red}(G(X))$.

Let $G$ be a locally compact Hausdorff \'{e}tale groupoid.
\begin{defn} (cf. \cite[Definition 3.1]{BCS24})\label{def:dynam}
A closed, two-sided ideal $I$ in $F^p_{red}(G)$ is called a dynamical ideal if it is generated as an ideal by its intersection with $C_0(G^{(0)})$.
\end{defn}

Given an invariant open subset $U$ of $G^{(0)}$, the reduction $G|_U=\{ \gamma\in G:s(\gamma),r(\gamma)\in U \}$ is an open subgroupoid of $G$, and $C_c(G|_U)$ can be viewed as an ideal in $C_c(G)$.
Let $I_U=\overline{C_c(G|_U)}^{F^p_{red}(G)}$.
Then $I_U$ is a closed ideal in $F^p_{red}(G)$.

By \cite[Proposition 3.6 and Remark 3.8]{BKM2}, for each $p\in[1,\infty]$, the inclusion $C_c(G)\subset C_0(G)$ extends to an injective contractive map $j_p:F^p_{red}(G)\to C_0(G)$ that is isometric on $C_0(G^{(0)})$.
This $j$-map allows us to define the support of an ideal in $F^p_{red}(G)$.

\begin{defn} (cf. \cite{BLi}, \cite[Page 5]{BCS24})
    Let $I\subseteq F^p_{red}(G)$ be an ideal.
    Define 
    \[
    \supp(I):=\{ \gamma\in G:j_p(I)(\gamma)\neq\{0\} \}\subseteq G
    \]
    to be the support of $I$,
    \begin{align*}
    \insupp(I):=&\{x\in G^{(0)}:f(x)\neq0\ \text{for some}\ f\in I\cap C_0(G^{(0)})\}\\
    =&\bigcup_{f\in I\cap C_0(G^{(0)})}\{x\in G^{(0)}:f(x)\neq0 \}\subseteq G^{(0)}
    \end{align*}
    to be the inner support of $I$ and
    \[
    \outsupp(I):=\supp(I)\cap G^{(0)}\subseteq G^{(0)}
    \]
    to be the outer support of $I$.
\end{defn}

\begin{lem} (cf. \cite[Proposition 3.3]{BCS24}, \cite[Lemma 3.4]{BCS24})
    The inner support $\insupp(I)$ of an ideal $I\subseteq F^p_{red}(G)$ is an invariant open subset of $G^{(0)}$. Moreover,
    for each invariant open subset $U\subseteq G^{(0)}$, we have 
    \begin{align*}
     I_U\cap C_0(G^{(0)})&=C_0(U), \\       
     \insupp(I_U)&=U,\\
     \supp(I_U)&=G|_U.
    \end{align*}
\end{lem}

\begin{proof}
$\insupp(I)$ is open because every $f\in C_0(G^{(0)})$ is continuous. To see that $\insupp(I)$ is invariant, take $x\in\insupp(I)$ and fix $\gamma\in G_x$. We will show that $r(\gamma)\in\insupp(I)$. As $x\in\insupp(I)$ there exists $f\in I\cap C_0(G^{(0)})$ such that $f(x)\neq0$. Let $B$ be an open bisection containing $\gamma$ and fix $h\in C_c(B)$ such that $h(\gamma)=1$. Define $h^*\in C_c(B^{-1})$ by $h^*(\alpha^{-1}):=\overline{h(\alpha)}$ for any $\alpha\in B$, where $B^{-1}:=\{\alpha^{-1}:\alpha\in B\}$. Then $hfh^*\in I\cap C_c(r(B))\subseteq I\cap C_c(G^{(0)})$. Moreover, \[hfh^*(r(\gamma))=h(\gamma)f(x)h^*(\gamma^{-1})=f(x)\neq0,\] so we have $r(\gamma)\in \insupp(I)$.

It is straightforward from the definition of $I_U$ that $I_U\cap C_0(G^{(0)})\subseteq C_0(U)$.
On the other hand, $C_c(U)\subseteq C_c(G|_U)$, and $U$ is a bisection, so $\Vert f\Vert_\infty=\Vert f\Vert_{F^p_{red}(G)}$ for $f\in C_c(U)$, and we have $C_0(U)\subseteq I_U\cap C_0(G^{(0)})$. By definition, we have $\insupp(I_U)=U$.

If $\gamma\in G\setminus G|_U$, then $j_p(f)(\gamma)=0$ for all $f\in C_c(G|_U)$, so $j_p(I_U)(\gamma)=\{0\}$ by continuity, thus $\gamma\notin\supp(I_U)$.
On the other hand, if $\gamma\in G|_U$, then there exists $f\in C_c(G|_U)$ such that $f(\gamma)\neq 0$, so $\gamma\in\supp(I_U)$.
\end{proof}

\begin{lem}\label{lem:dynamical}
$I\subseteq F^p_{red}(G)$ is a dynamical ideal if and only if $I=I_U$ for the invariant open subset $U=\insupp(I)$.
\end{lem}

\begin{proof}
    ($\Leftarrow$)
    $C_c(U)$ contains a bounded approximate identity for $F^p_{red}(G|_U)$, so $I_U$ is generated as an ideal by $C_c(U)$. It then follows from the previous lemma that $I_U$ is a dynamical ideal.

    ($\Rightarrow$)
    If $I$ is a dynamical ideal, then $I$ is generated as an ideal by $I\cap C_0(G^{(0)})=
    C_0(U)$ for some invariant open subset $U\subseteq G^{(0)}$ (cf. \cite[Lemma 10.3.1]{Sims}). Hence $I=I_U$, and $U=\insupp(I_U)=\insupp(I)$.
\end{proof}

\begin{rem}\label{rem:subgroupoid}\leavevmode
    \begin{enumerate}
        \item For every invariant open subset $U\subseteq G^{(0)}$, $G|_U$ is a subgroupoid of $G$, and the dynamical ideal $I_U$ of $F^p_{red}(G)$ is the reduced groupoid $L^p$ operator algebra $F^p_{red}(G|_U)$ on $G|_U$.
        \item Both open subsets $\insupp(I)$ and $\outsupp(I)$ are invariant in $G^{(0)}$, and clearly $\insupp(I)\subseteq\outsupp(I)$. If $I=I_U$ is a dynamical ideal, then $\insupp(I)=\outsupp(I)$; the converse need not hold.
    \end{enumerate}
\end{rem}

Denote by $\mathcal{O}[G]$ the collection of all invariant open subsets of $G^{(0)}$.

\begin{prop} (cf. \cite[Proposition 3.3]{BCS24})\label{prop:lattice4}
    Denote by $\mathbb{I}[F^p_{red}(G)]$ the lattice of all ideals of $F^p_{red}(G)$ and by $\mathbb{I}_0[F^p_{red}(G)]$ the sublattice of all dynamical ideals. Define
\begin{align*}
    \lambda_7:\mathcal{O}[G]&\to \mathbb{I}_0[F^p_{red}(G)]\\
    U&\mapsto I_U,
\end{align*}
and
\begin{align*}
    \lambda_8:\mathbb{I}[F^p_{red}(G)]&\to\mathcal{O}[G]\\
    I&\mapsto U_I:=\insupp(I).
\end{align*}
Then the maps $\lambda_7,\lambda_8$ are order-preserving, and  $\lambda_8|_{\mathbb{I}_0[F^p_{red}(G)]}=\lambda_7^{-1}$ is a lattice isomorphism from the lattice $\mathbb{I}_0[F^p_{red}(G)]$ of dynamical ideals to the lattice $\mathcal{O}[G]$ of invariant open subsets of $G^{(0)}$.
\end{prop}

\begin{proof}
Both $\lambda_7$ and $\lambda_8$ are order-preserving. The previous two lemmas imply that $\lambda_7:U\mapsto I_U$ is a bijection and $\lambda_7^{-1}=\lambda_8|_{\mathbb{I}_0[F^p_{red}(G)]}$. 
\end{proof}

For reduced groupoid $C^*$-algebras, many properties of ideals and their inner or outer supports have been studied in \cite{BLi,BCS24} and in unpublished work of Kang Li and Jiawen Zhang. For the present $L^p$ application, it suffices to consider principal ample groupoids. The next proposition is a groupoid analogue of
Lemma~\ref{lem:ideal}.

\begin{prop}\label{lem:groupoidideal}
    For a principal and ample groupoid $G$, if $a\in F^p_{red}(G)$ and $j_p(a)(\gamma)\neq0$ for some $\gamma\in G$, then there exists a compact open subset $D\subseteq G^{(0)}$ containing $r(\gamma)$ such that $\chi_D\in\langle\ a\ \rangle$.
\end{prop}

\begin{proof}
    Put $c=j_p(a)(\gamma)$. Choose an open bisection $B\ni\gamma$ and $h\in C_c(B^{-1})$ such that $h(\gamma^{-1})=c^{-1}$. For $b=ah\in\langle\ a\ \rangle$, the convolution coefficient formula and $B^{-1}\cap G_{r(\gamma)}=\{\gamma^{-1}\}$ give
    \[
    j_p(b)(r(\gamma))=\sum_{\eta\in G_{r(\gamma)}}j_p(a)(\eta^{-1})h(\eta)=j_p(a)(\gamma)h(\gamma^{-1})=1.
    \]
    Choose $f\in C_c(G)$ with $\|b-f\|_{F^p_{red}(G)}<1/4$. Then $|f(r(\gamma))|>3/4$. Set
    \[
    K=\operatorname{supp}(f)\setminus G^{(0)}.
    \]
    Since $G$ is principal, $(r(\gamma),r(\gamma))\notin(r,s)(K)$. By ampleness and continuity, there is a compact open neighborhood $D\subseteq G^{(0)}$ of $r(\gamma)$ such that
    \[
    K\cap r^{-1}(D)\cap s^{-1}(D)=\emptyset\ \text{and}\ |f(u)|>2/3,
    \]
    for any $u\in D$. Hence, $d:=\chi_Df\chi_D$ can be viewed as an element in $F_{\mathrm{red}}^p(G|_D)$, which is invertible and $\|d^{-1}\|_{F_{\mathrm{red}}^p(G|_D)}\leq3/2$. Now, let $w=\chi_Db\chi_D\in\langle\ a\ \rangle$. Since
    \[
    \|w-d\|_{F_{\mathrm{red}}^p(G|_D)}\leq\|b-f\|_{F^p_{red}(G)}<1/4,
    \]
    we have
    \[
    \|d^{-1}(w-d)\|_{F_{\mathrm{red}}^p(G|_D)}<3/8<1.
    \]
    Thus, $w$ is invertible in $F_{\mathrm{red}}^p(G|_D)$. Consequently, $\chi_D=w^{-1}w\in\langle\ a\ \rangle$.
\end{proof}

For each ideal $I$ in $F^p_{red}(G)$, we immediately have the following:
\begin{cor}\leavevmode
    \begin{enumerate}
        \item The inner and outer supports of $I$ coincide; we call their common value $\essupp(I):=\insupp(I)=\outsupp(I)$ the essential support of $I$.
        \item If $I$ is nonzero, then $I\cap C_0(G^{(0)})\neq\{0\}$ and $\essupp(I)$ is nonempty.
    \end{enumerate} 
\end{cor}
General locally compact Hausdorff \'{e}tale groupoids need not have the two properties above.

\subsection{Preservation of essential support} 
Note that the coarse groupoid $G(X)$ is principal and ample. In \cite{CW04}, Chen and Wang established the correspondence between $\mathcal{J}[\mathcal{E}]$ and $\mathcal{O}[G(X)]$ as follows: 

Denote $U(\mathbf{0})=\emptyset$. Let $\mathcal{J}$ be an ideal of $\mathcal{E}$, and define
\[
U(\mathcal{J}):=\bigcup_{E\in\mathcal{J}}\overline{r(E)}.
\]

\begin{lem}\cite[Lemma 5.2]{CW04}
    $U(\mathcal{J})$ is an invariant open subset of $G(X)^{(0)}$.
\end{lem}

The map 
\begin{align*}
    \lambda_5:\mathcal{J}[\mathcal{E}]&\to\mathcal{O}[G(X)]\\
    \mathcal{J}&\mapsto U(\mathcal{J}),
\end{align*} 
is order-preserving.

Conversely, let $U\subseteq G(X)^{(0)}$ be an invariant open subset, and define
\[
\mathcal{J}(U):=\begin{cases}
    \mathbf{0}\ &U=\emptyset,\\
    \{E\subseteq X\times X: \overline{E}^{\beta(X\times X)}\subseteq G(X)|_U\}\ &\text{otherwise}.
\end{cases}
\]

\begin{lem}\cite[Lemma 5.3]{CW04}
    For a nonempty invariant open subset $U$, $\mathcal{J}(U)$ is an ideal of $\mathcal{E}$.
\end{lem}

Moreover, the map 
\begin{align*}
    \lambda_6:\mathcal{O}[G(X)]&\to\mathcal{J}[\mathcal{E}]\\
    U&\mapsto \mathcal{J}(U),
\end{align*} 
is order-preserving.

\begin{lem}\cite[Proposition 5.4]{CW04}\label{lem:open}
    $\lambda_5\lambda_6=id_{\mathcal{O}[G(X)]}$ and $\lambda_6\lambda_5=id_{\mathcal{J}[\mathcal{E}]}$. In other words, for any invariant open subset $U\subseteq G(X)^{(0)}$, we have
    \[
    U=U(\mathcal{J}(U)),
    \]
    and for any ideal $\mathcal{J}$ of $\mathcal{E}$, we have
    \[
    \mathcal{J}=\mathcal{J}(U(\mathcal{J})).
    \]
\end{lem}

Hence, $\lambda_5$ is a lattice isomorphism between $\mathcal{J}[\mathcal{E}]$ and $\mathcal{O}[G(X)]$.

\begin{rem}\label{rem:exactly}
Under the correspondence
\[
\mathcal J
\longleftrightarrow
U(\mathcal J),
\]
a nonzero invariant open subset of $G(X,\mathcal E)^{(0)}$ is the
unit space of the coarse groupoid $G(X,\mathcal J)$, while the
corresponding geometric and dynamical ideals are
\[
B^p_u(X,\mathcal J)
\ \text{and}\ F^p_{\mathrm{red}}(G(X,\mathcal J)),
\]
respectively. For the smallest ideal $\mathcal E_{\min}$, we have
\[
U(\mathcal E_{\min})=X,
\ G(X,\mathcal E_{\min})=X\times X,
\]
so $B^p_u(X,\mathcal E_{\min})$ and
$F^p_{\mathrm{red}}(X\times X)$ are the smallest nonzero geometric
and dynamical ideals, respectively.
\end{rem}

Let
\[
\Lambda:\mathbb{I}[F^p_{red}(G(X))]\to I[B^p_u(X)]
\]
denote the ideal correspondence induced by $\Psi:F^p_{red}(G(X))\to B^p_u(X)$. It is clear that $B^p_u(X,\mathcal{E}_{min})\cong F^p_{red}(X\times X)$, i.e., $\Lambda$ sends the smallest nonzero dynamical ideal to the smallest nonzero geometric ideal.

By Lemma~\ref{lem:diagonal-identification}, Corollary~\ref{cor:geome}, and Definition~\ref{def:dynam}, we immediately have:
\begin{cor}\label{cor:isomorphic}
The isometric isomorphism
\[
\Psi:F^p_{red}(G(X))\to B^p_u(X)
\]
induces a lattice isomorphism from the dynamical ideals of
$F^p_{red}(G(X))$ onto the geometric ideals of $B^p_u(X)$.
\end{cor}

For orientation, the four lattices appearing above are related by the
canonical identifications
\[
\mathbb{I}_0[F^p_{\mathrm{red}}(G(X))]
\;\cong\;
\mathcal{O}[G(X)]
\;\cong\;
\mathcal{J}[\mathcal{E}]
\;\cong\;
I_0[B^p_u(X)].
\]
The first identification is implemented by
$\lambda_8$ with inverse $\lambda_7$, the second by
$\lambda_6$ with inverse $\lambda_5$, and the third by
$\lambda_4$ with inverse $\lambda_3$. Under these identifications,
$\Lambda$ is the canonical correspondence between the dynamical and
geometric ideal lattices.

\begin{thm}\label{thm:isomorphic} 
     The correspondence $\Lambda$ preserves essential support under the lattice identification $\mathcal{O}[G(X)]\cong\mathcal{J}[\mathcal{E}]$, that is, the following diagram commutes:
    
    \centering
    \begin{tikzcd}
        \mathbb{I}[F^p_{red}(G(X))]\ar[r,"\Lambda", "\cong"']\ar[d,"\lambda_8"]& I[B^p_u(X)]\ar[d,"\lambda_3"]\\
        \mathcal{O}[G(X)]\ar[r,"\lambda_6"]& \mathcal{J}[\mathcal{E}]\ar[l,"\lambda_5"]
    \end{tikzcd}
\end{thm}

\begin{proof}
     Let $J\subseteq F^p_{red}(G(X))$ be an ideal and set
$I:=\Psi(J)\subseteq B^p_u(X)$. It suffices to prove $\essupp(J)=U(\mathcal{J}(I))$. On the one hand, for any $T\in I$ and $\varepsilon>0$, it follows from Proposition~\ref{thm:keythm}(ii) that 
    \[
    \chi_{r(\supp_{\varepsilon}(T))}\in I\cap\ell^\infty_\mathcal{E}(X).
    \]
    Hence, since $I=\Psi(J)$, Lemma~\ref{lem:diagonal-identification} gives
    \[
    \chi_{\overline{r(\supp_{\varepsilon}(T))}}\in J\cap C_0(G(X)^{(0)}),
    \]
    so $U(\mathcal{J}(I))\subseteq\insupp(J)=\essupp(J)$.

    On the other hand, for any $\omega\in\essupp(J)$, there exists $f\in J\cap C_0(G(X)^{(0)})$ such that $f(\omega)\neq0$. Let $S:=\Psi(f)\in I\cap\ell^\infty_\mathcal{E}(X)$ and $0<\varepsilon<|f(\omega)|$, then 
    \[
    \omega\in r(\overline{\supp_\varepsilon(S)})=\overline{r(\supp_\varepsilon(S))}\subseteq U(\mathcal{J}(I)).
    \]
    Hence, 
    \[
    \essupp(J)\subseteq U(\mathcal{J}(I)),
    \]
    which completes the proof.
\end{proof}

\begin{rem}\leavevmode
    \begin{enumerate}
        \item The ideal lattice $\mathbb{I}[F^p_{red}(G(X))]$ (or $I[B^p_u(X)]$) decomposes into blocks, where each block $\mathbb{I}[U]$ (or $I[\mathcal{J}]$) is a sublattice of ideals with the same essential support $U$ (or $\mathcal{J}$), i.e.,
        \begin{align*}
            \mathbb{I}[U]&=\{I\in \mathbb{I}[F^p_{red}(G(X))]: U_I=U\};\\
            I[\mathcal{J}]&=\{I\in I[B^p_u(X)]:\mathcal{J}(I)=\mathcal{J}\}.
        \end{align*}
        \item By definition, the dynamical ideal $I_U$ is the smallest element in the sublattice $\mathbb{I}[U]$ of all ideals with essential support $U$.
    \end{enumerate}
\end{rem}

Regarding an operator $T\in B^p_u(X)$ as the bounded function
\[
T:X\times X\to\mathbb{C},
\ (x,y)\mapsto T(x,y),
\]
let $\overline{T}$ denote the restriction to $G(X)$ of its continuous
Stone--\v{C}ech extension to $\beta(X\times X)$. Then
$\overline{T}\in C_0(G(X))$.

Indeed, choose $T_n\in\mathbb{C}_u[X,\mathcal{E}]$ such that
$\|T_n-T\|_p\to0$. Since
\[
\|T_n-T\|_{\sup}\leq\|T_n-T\|_p,
\]
the functions $\overline{T_n}|_{G(X)}$ converge uniformly to
$\overline{T}$. Each $\overline{T_n}|_{G(X)}$ has compact support,
because $\supp(T_n)$ is controlled. Hence
$\overline{T}\in C_0(G(X))$.

Moreover, if $\supp_{G(X)}(\overline{T})$ denotes the topological
support of $\overline{T}$, then
\[
\supp_{G(X)}(\overline{T})
=
\overline{\supp(T)}^{\,G(X)},
\]
where $\supp(T)=\{(x,y)\in X\times X:T(x,y)\neq0\}$ and the closure on
the right is taken in $G(X)$.

Denote the resulting injective contractive map by
\[
\beta:B^p_u(X)\to C_0(G(X)),
\ T\mapsto\overline{T}.
\]

\begin{lem}\label{lem:jmap}
    There is a commutative diagram:
    \begin{center}
    \begin{tikzcd}
        F^p_{red}(G(X))\rar["\Psi", "\cong"'] \drar["j_p"']& B^p_u(X)\dar["\beta"]\\
        &C_0(G(X))
    \end{tikzcd}
    \end{center}
    Moreover, since $j_p$ and $\beta$ are injective, correspondence under
$\Psi$ can be checked after applying these maps: two corresponding
elements (respectively, ideals) have the same image in $C_0(G(X))$.
\end{lem}

\begin{proof}
    By definition,
\[
\|\overline{T}\|_{\sup}=\|T\|_{\sup}\leq\|T\|_p,
\]
so $\beta$ is a contractive injection. Since
\[
j_p(f)=\beta(\Psi(f))
\]
for $f\in C_c(G(X))$, the identity extends by continuity.
\end{proof}
The map $j_p$ (resp. $\beta$) packages the regular representations and limit operators $\{\widetilde{\lambda}_\omega\}_{\omega\in G(X)^{(0)}}$ (resp. $\{\Phi_\omega\}_{\omega\in G(X)^{(0)}}$) considered in Appendix~\ref{appen2}.

We will use the following elementary observation without proof.
\begin{lem}\label{lem:supp2}
    For any $T\in B^p_u(X)$ and $\varepsilon>0$, we have
    \[
    \supp(\overline{T_\varepsilon})\subseteq\{\delta\in\beta(X\times X):|\overline{T}(\delta)|\geq\varepsilon\}\subseteq\supp(\overline{T_{\varepsilon/2}})
    \]
\end{lem}

For each invariant open subset $U\subseteq G(X)^{(0)}$, define 
\[
\widetilde{I}_U:=\{f\in F^p_{red}(G(X)): j_p(f)(\gamma)=0, \forall\gamma\notin G(X)|_U\}.
\]
Then $\widetilde{I}_U\supseteq I_U$. We show that $\widetilde{I}_U$ is the largest element in the sublattice $\mathbb{I}[U]$, which corresponds to a ghostly ideal defined in Proposition~\ref{prop:ghostly}.

\begin{prop}\label{prop:restrictive}
    $\widetilde{I}_U$ is an ideal of $F^p_{red}(G(X))$ with essential support 
    \[
    \essupp(\widetilde{I}_U)=U.
    \]
    Thus, $\widetilde{I}_U\in\mathbb{I}[U]$. Moreover, $\Lambda$ sends $\widetilde{I}_U$ to $\widetilde{I}(\mathcal{J}(U))\in I[\mathcal{J}(U)]$, the ghostly ideal associated to $\mathcal{J}(U)$. Moreover, $\widetilde{I}_U$ is the largest element of
$\mathbb{I}[U]$; we call it the restrictive ideal associated with $U$. 
\end{prop}

\begin{proof}
    From Proposition~\ref{prop:ghostly} and Theorem~\ref{thm:isomorphic}, it remains to prove $\Lambda(\widetilde{I}_U)=\widetilde{I}(\mathcal{J}(U))$. Since
    \begin{align*}
        \widetilde{I}(\mathcal{J}(U)):=&\{T\in B^p_u(X):\supp_\varepsilon(T)\in\mathcal{J}(U)\ \text{for every}\ \varepsilon>0\}\\
        =&\{T\in B^p_u(X):\overline{r(\supp_\varepsilon(T))}\subseteq U\ \text{for every}\ \varepsilon>0\}\\
        =&\{T\in B^p_u(X):\overline{T}(\delta)=0, \forall\delta\notin G(X)|_U\},
    \end{align*}
    where the last equality follows from Lemma~\ref{lem:supp2}. It follows from Lemma~\ref{lem:jmap} that $\Lambda(\widetilde{I}_U)=\widetilde{I}(\mathcal{J}(U))$.
\end{proof}

Combining this with the limit operator theory of
Definition~\ref{def:limitoperator} yields the following description of
ideals in $B^p_u(X)$.
\begin{cor}\label{prop:limitideal}\leavevmode
    \begin{enumerate}
        \item For any invariant open subset $U\subseteq G(X)^{(0)}$, the ideals in $B^p_u(X)$ with essential support $\mathcal{J}(U)$ correspond to the ideals in $F^p_{red}(G(X))$ with essential support $U$. Thus, those ideals have limit operators that vanish in the $G(X)^{(0)}\setminus U$-directions.
        \item $\widetilde{I}(\mathcal{J})=\{T\in B^p_u(X):\Phi_\omega(T)=0,\omega\in G(X)^{(0)}\setminus U(\mathcal{J})\}$;
        \item An operator $T\in B^p_u(X)$ is a ghost if and only if $\Phi_\omega(T)=0$ for any non-principal ultrafilter $\omega\in G(X)^{(0)}\setminus X$. Hence,
        \[
        I^p_G(X)=\bigcap_{\omega\in G(X)^{(0)}\setminus X}\ker(\Phi_\omega).
        \]
    \end{enumerate} 
\end{cor}

\begin{rem}\leavevmode
    \begin{enumerate}
        \item Every ideal $I$ of $F^p_{red}(G(X))$ lies between $I_{U_I}$ and $\widetilde{I}_{U_I}$. The study of the ideal structure for $F^p_{red}(G(X))$ (or $B^p_u(X)$) can be reduced to the analysis of the sublattice $\mathbb{I}[U]$ (or $I[\mathcal{J}]$);
        \item $I_{U_I}$ (or $I(\mathcal{J}(I))$) is the largest dynamical (geometric) ideal contained in $I$ and $\widetilde{I}_{U_I}$ (or $\widetilde{I}(\mathcal{J}(I))$) is the smallest restrictive (ghostly) ideal containing $I$. 
        \item Kang Li and Jiawen Zhang have studied (i) and (ii) above in the context of ghostly ideals in reduced groupoid $C^*$-algebras.
    \end{enumerate}
\end{rem}

\begin{ex}
    The ideals $B^p_u(X,\mathcal{E}_{min})$ and $I^p_G(X)$ in $B^p_u(X)$ both correspond to $\mathcal{E}_{min}$, which is the smallest ideal of $\mathcal{E}$:
    \[
    \mathcal{J}(B^p_u(X,\mathcal{E}_{min}))=\mathcal{J}(I^p_G(X))=\mathcal{E}_{min}.
    \]
    Moreover, $B^p_u(X,\mathcal{E}_{min})=I(\mathcal{E}_{min})$ is the geometric ideal associated to $\mathcal{E}_{min}$, $I^p_G(X)=\widetilde{I}(\mathcal{E}_{min})$ is the ghostly ideal associated to $\mathcal{E}_{min}$ and 
    \[
    U(\mathcal{E}_{min})=X\subseteq G(X)^{(0)}.
    \]
\end{ex}

\begin{cor}\label{cor:crossed-product-ideals}
Let $\Gamma$ be a countable discrete group equipped with a proper
left-invariant metric, and let $1<p<\infty$.  Under the canonical
isometric identifications
\[
 \ell^\infty(\Gamma)\rtimes_{\mathrm{red},p}\Gamma
 \cong B_u^p(\Gamma)
 \cong F^p_{\mathrm{red}}(\beta\Gamma\rtimes\Gamma),
\]
the dynamical ideals are parametrized by the
$\Gamma$-invariant open subsets of $\beta\Gamma$.

More precisely, if $U\subseteq\beta\Gamma$ is invariant and open, then
the corresponding dynamical ideal is the closed ideal generated by
$C_0(U)$, and the assignment
\[
 U\mapsto I_U
\]
is a lattice isomorphism from the lattice of invariant open subsets of
$\beta\Gamma$ onto the lattice of dynamical ideals.  Under the
identification with $B_u^p(\Gamma)$, these correspond exactly to the
geometric ideals.
\end{cor}

\begin{proof}
The canonical isometric isomorphism
\[
 \ell^\infty(\Gamma)\rtimes_{\mathrm{red},p}\Gamma
 \cong B_u^p(\Gamma)
\]
is noted in \cite[Page~1066]{ChungLi18}, while the coarse groupoid
of $\Gamma$ is the transformation groupoid
$\beta\Gamma\rtimes\Gamma$ by
\cite[Proposition~3.4]{STY}.
The conclusion therefore follows from
Corollary~\ref{cor:isomorphic} and
Proposition~\ref{prop:lattice4}.
\end{proof}

\section{Property A and full ideal classification}\label{section6}

In this section, we continue to work with a uniformly locally finite coarse space $(X,\mathcal{E})$, which need not be unital. For a non-unital coarse structure $\mathcal{E}$, let $\mathcal{E}^\Delta$ denote the coarse structure obtained by adjoining the diagonal $\Delta_X$. Property A was introduced by Yu for discrete metric spaces. For uniformly locally finite unital coarse spaces we use the formulation of \cite[Definition 2.7]{Sako}, and for non-unital coarse spaces we define property A by passing to the unitization $(X,\mathcal{E}^\Delta)$.

For $1<p<\infty$, property A gives controlled $p$-kernels
through Lemma~\ref{lem:propertyA}. At the extreme exponents
$p\in\{0,1,\infty\}$, Proposition~\ref{prop:endpointpkernel}
shows that the corresponding controlled kernels exist for every
uniformly locally finite coarse space after passing, when necessary,
to the unitization.

For unital coarse spaces, the common mechanism developed in this
section is
\[
\begin{aligned}
\text{controlled $p$-kernels}
&\Longrightarrow
\text{contractive controlled Schur multipliers} \\
&\Longrightarrow
\text{a multiplier approximate identity} \\
&\Longrightarrow
\text{all ideals are geometric} \\
&\Longrightarrow
\text{full ideal classification}.
\end{aligned}
\]
For $1<p<\infty$, property A provides the first step; for
$p\in\{0,1,\infty\}$, it is automatic by
Proposition~\ref{prop:endpointpkernel}. For a non-unital coarse
structure, we apply the kernel construction to $\mathcal E^\Delta$
and compress the resulting multipliers back to $\mathcal E$.

\subsection{Controlled $p$-partitions of unity and $p$-kernels}

\begin{defn}\cite[Definition 2.7]{Sako}\label{def:propertyA}
    A uniformly locally finite unital coarse space $(X,\mathcal{E})$ is said to have property A if for every $\varepsilon>0$ and every controlled set $E\in\mathcal{E}$, there exist a controlled set $F\in\mathcal{E}$ and a subset $A\subset F\times\mathbb{N}$ such that
    \begin{enumerate}
        \item For $x\in X$, $A_x:=\{(y,n)\in X\times\mathbb{N}: (x,y,n)\in A\}$ is finite;
        \item For $x\in X$, $A_x$ is nonempty;
        \item $\#(A_x\bigtriangleup A_y)<\varepsilon\#(A_x\cap A_y)$ for any $(x,y)\in E$. 
    \end{enumerate}
    For a non-unital coarse space $(X,\mathcal{E})$, we say that $(X,\mathcal{E})$ has property A if its unitization $(X,\mathcal{E}^\Delta)$ has property A.
\end{defn}

We next introduce several equivalent formulations of property A for
uniformly locally finite unital coarse spaces.

Recall from Definition~\ref{def:LV} that a (partial) covering $\boldsymbol{a}$ of $X$ is controlled if $\mathrm{diag}(\boldsymbol{a}):=\bigcup\{A\times A: A\in\boldsymbol{a}\}\in\mathcal{E}$ and from Lemma~\ref{lem:Eix} that the canonical partial covering $E(\boldsymbol{i}_X)=\{E_x:x\in s(E)\subseteq X\}$ induced by an entourage $E\in\mathcal{E}$ is a controlled partial covering.

\begin{defn}(cf. \cite{HR})
    For $1\leq p<\infty$, a $p$-Higson--Roe function on $X$ is a map
$\eta:X\to\ell^p(X)_1$ for which there exists an entourage
$F\in\mathcal{E}$ such that
    \[
    F=\{(x,y)\in X\times X: \ y\in\supp(\eta_x)\}.
    \]
    Moreover, for $E\in\mathcal{E}$ and $\varepsilon>0$, we say $\eta:X\to\ell^p(X)_{1}$ has $(E,\varepsilon)$-variation if $\|\eta_x-\eta_y\|_p<\varepsilon$ whenever $(x,y)\in E$.
    Here, $\ell^p(X)_1$ denotes the set of unit vectors in $\ell^p(X)$.
\end{defn}

\begin{defn} (cf. \cite[Definition 1.2.3]{Willett}, \cite[Definition 6.1]{SW17})\label{def:partition}
    For $1\leq p<\infty$, a controlled $p$-partition of unity on $X$ is a collection $\{\phi_i:X\to [0,1]\}_{i\in I}$ of functions on $X$ satisfying:
        \begin{enumerate}
            \item There exists $N\in\mathbb{N}$ such that for each $x\in X$, at most $N$ of the values $\phi_i(x)$ are nonzero;
            \item There exists a controlled covering $\{U_i\}_{i\in I}$ of $X$ such that $\{\phi_i\}_{i\in I}$ is subordinated to $\{U_i\}_{i\in I}$, i.e.,  $\supp(\phi_i)\subseteq U_i$ for each $i\in I$;
            \item For each $x\in X$, $\sum_{i\in I}(\phi_i(x))^p=1$.
        \end{enumerate}
        Moreover, for $E\in\mathcal{E}$ and $\varepsilon>0$, we say $\{\phi_i\}_{i\in I}$ has $(E,\varepsilon)$-variation if 
            $\sum_{i\in I}|\phi_i(x)-\phi_i(y)|^p<\varepsilon$ whenever $(x,y)\in E$.
\end{defn}

Every $p$-Higson-Roe function $\eta:X\to\ell^p(X)_{1}$ induces a $p$-partition of unity $\{\phi_y:X\to [0,1]\}_{y\in X}$, where \[\phi_y(x):=|\eta_x(y)|\] for every $x, y\in X$. We will show in Lemma~\ref{lem:propertyA} that it is a controlled $p$-partition of unity.

A $p$-Higson-Roe function $\eta:X\to\ell^p(X)_{1}$ also determines a controlled propagation $p$-kernel as follows.
Recall that given a set $X$ and $p,q\in[1,\infty)$, we have the Mazur map $M_{p,q}:\ell^p(X)\to\ell^q(X)$ given by the formula
\[
M_{p,q}(f)(x)=|f(x)|^{p/q}\sign(f(x)),
\]
where 
\[
    \sign(f(x))=\begin{cases}
        \frac{f(x)}{|f(x)|}\ &f(x)\neq0,\\
        0\ &f(x)=0.
    \end{cases}
\] 
It is clear that $\|M_{p,q}(f)\|_q^q=\|f\|_p^p$ for every $f\in\ell^p(X)$, and that $M_{p,q}$ and $M_{q,p}$ are inverses of each other. Moreover, its restriction to the unit sphere is a uniformly continuous
homeomorphism whose inverse is uniformly continuous.

When $p=1, q=\infty$, define $M_{1,\infty}:\ell^1(X)\to\ell^\infty(X)$ by 
$M_{1,\infty}(f)(x):=\sign(f(x))$.
For $p\in[1,\infty)$ and $\frac{1}{p}+\frac{1}{q}=1$, define 
\[
[f,g]:=\langle\ f,\overline{M_{p,q}(g)}\ \rangle= \sum_{x\in X}f(x)\overline{M_{p,q}(g)(x)}
\]
for $f,g\in\ell^p(X)$. It is clear that $[f,f]=\|f\|_p^p$.

Moreover, we have the following two inequalities.

\begin{lem}(cf. \cite{BL}, \cite[Proposition 5.4.6]{NY})\label{lem:inequ}\leavevmode 
    \begin{enumerate}
        \item Let $1\leq p<q<\infty$. There exists a constant $C_1>0$, which depends only on $p$ and $q$, such that
        \[
        \frac{p}{q}\|f-g\|_p\leq\|M_{p,q}(f)-M_{p,q}(g)\|_q\leq C_1\|f-g\|_p^{p/q}
        \]
        for every $f,g\in\ell^p(X)_1$;
        \item For $p\in(1,\infty)$, there exists a constant $C_2>0$, which depends only on $p$, and a constant $\alpha=\max\{2,p\}$, such that
        \[
        C_2\|f-g\|_p^\alpha\leq1-\Re[f,g]\leq\|f-g\|_p
        \]
        for every $f,g\in\ell^p(X)_1$. 
        Here, $\Re z$ denotes the real part of $z\in\mathbb{C}$.
    \end{enumerate}
\end{lem}

\begin{proof}
    Part~(i) is \cite[Proposition 5.4.6]{NY}; see also \cite{BL}. 
    (ii) Let $p\in(1,\infty)$ and let $q$ be its conjugate exponent. For
$f,g\in\ell^p(X)_1$, 
    \[
    1\pm\Re[f,g]=\Re\langle\ g\pm f,\overline{M_{p,q}(g)}\ \rangle\leq\|f\pm g\|_p.
    \]
    By Clarkson's inequalities in \cite{Clark}, there exists a constant $C_2>0$, which depends only on $p$ and $\alpha=\max\{2,p\}$, such that
    \[
    1-\|\frac{f+g}{2}\|_p\geq\frac{C_2}{2}\|f-g\|_p^\alpha.
    \]
    Hence, $1-\Re[f,g]=2-(1+\Re[f,g])\geq 2-\|f+g\|_p\geq C_2\|f-g\|_p^\alpha$.
\end{proof}

\begin{rem}\label{rem:inequ}\leavevmode
    \begin{enumerate}
        \item The estimate in Lemma~\ref{lem:inequ}(i) also yields the corresponding
bounds when $p>q$. Indeed, represent $f,g\in\ell^p(X)_1$ as images under
$M_{q,p}$ of elements of $\ell^q(X)$. Then
        \[
        C'_1\|f-g\|_p^{p/q}\leq\|M_{p,q}(f)-M_{p,q}(g)\|_q\leq\frac{p}{q}\|f-g\|_p,
        \] 
        where $C'_1=C_1^{-\frac{p}{q}}$ depends only on $p$ and $q$. 
        \item When $p=1$, the right-hand inequality in Lemma~\ref{lem:inequ}(ii) still holds, but the left-hand inequality may fail because the uniform convexity argument
used above is unavailable in $\ell^1$; see \cite{Clark}.
        For example, let $f:=(1,0,\cdots)$ and $g:=(\frac{1}{2},\frac{1}{2},0,\cdots)$ be two elements in $\ell^1(X)_1$. Then $\|f-g\|_1=1$ but $1-\Re[f,g]=0$.
    \end{enumerate}
\end{rem}

\begin{defn} \label{def:pkernel}
    For $p\in[1,\infty)$ and a $p$-Higson-Roe function $\eta:X\to\ell^p(X)_{1}$, define a map $\varphi:X\times X\to\mathbb{R}$ by
    \[
    \varphi(x,y):=[|\eta_y|,|\eta_x|]=\sum_{z\in X}|\eta_x(z)|^{p/q}\cdot|\eta_y(z)|=\sum_{z\in X}\phi_z(x)^{p/q}\phi_z(y),
    \]
    where $\frac{1}{p}+\frac{1}{q}=1$ with the convention $\frac{1}{\infty}=0$ and $0^{\frac{1}{\infty}}=0$. 
    We say $\varphi:X\times X\to\mathbb{R}$ is a controlled propagation $p$-kernel induced by the $p$-Higson-Roe function $\eta:X\to\ell^p(X)_{1}$. 

    For $p\in\{0,\infty\}$ and a 1-Higson-Roe function $\eta:X\to\ell^1(X)_{1}$, define $\varphi:X\times X\to\mathbb{R}$ by
    \begin{align*}
    \varphi(x,y) &:=[|\eta_x|,|\eta_y|]=\sum_{z\in X}|\eta_y(z)|^{1/\infty}\cdot|\eta_x(z)|=\sum_{z\in X}\phi_z(y)^{1/\infty}\phi_z(x) \\
    &= \sum_{z\in \supp(\eta_y)}\phi_z(x).
    \end{align*}
    We say $\varphi:X\times X\to\mathbb{R}$ is a controlled propagation $p$-kernel induced by the $1$-Higson-Roe function $\eta:X\to\ell^1(X)_{1}$.

    For $E\in\mathcal{E}$ and $\varepsilon>0$, we say $\varphi$ has $(E,\varepsilon)$-variation if $|1-\varphi(x,y)|<\varepsilon$ whenever $(x,y)\in E$.
\end{defn}

For every $p\in\{0\}\cup[1,\infty]$, the resulting kernel satisfies
    \begin{enumerate}
        \item $\varphi(x,y)\leq1$ and $\varphi(x,x)=1$ for each $x,y\in X$;
        \item $\supp(\varphi):=\{(x,y)\in X\times X:\varphi(x,y)\neq0\}= F\circ F^{-1}$, where $F:=\{(x,y)\in X\times X: x\in X, y\in\supp(\eta_x)\}\in\mathcal{E}$.
    \end{enumerate}
Thus $\varphi\in\mathbb{C}_u[X,\mathcal{E}]$, and
Lemma~\ref{lem:Schur} yields the bounded controlled propagation Schur
multiplier
    \[
        \mathcal{M}_\varphi: B^p_u(X)\to B^p_u(X),\ T\mapsto \varphi\circ T.
    \]
    Observe also that the existence of a controlled propagation $p$-kernel
with respect to $\mathcal E$ forces the coarse structure to be unital.
Indeed, $\varphi(x,x)=1$ for every $x\in X$, so $\Delta_X\subseteq\supp(\varphi)\in\mathcal E$, and hence $\Delta_X\in\mathcal E$. Thus, when $\mathcal E$ is
non-unital, such a kernel cannot have $\mathcal E$-controlled
propagation. In that case, we shall instead use controlled propagation
$p$-kernels for the unitized coarse structure
$\mathcal E^\Delta$ and subsequently compress the associated
multipliers back to $\mathcal E$.
Moreover, we will show that these controlled propagation Schur multipliers are contractive. First, we need the following lemma:

\begin{lem}(\cite[Theorem 5.11]{Pisier}) \label{lem:Pisier}
    For $p\in[1,\infty)$ and $\frac{1}{p}+\frac{1}{q}=1$, let $\varphi\in\ell^\infty(X\times X)$ and let $C\geq0$ be a constant. Then $\mathcal{M}_\varphi$ is a bounded Schur multiplier on $B(\ell^p(X))$ with norm $\leq C$ if and only if there exist a measure space $(\Omega,\mu)$ and elements $(\eta_y)_{y\in X}$ in $L^p(\mu)$ and $(\eta'_x)_{x\in X}$ in $L^q(\mu)$ such that
    \[
    \varphi(x,y)=\langle\ \eta_y,\eta_x'\ \rangle
    \]
    for any $x,y\in X$ and $\sup_{x\in X}\|\eta_x\|_p\cdot \sup_{y\in X}\|\eta'_y\|_q\leq C$. 
\end{lem}

\begin{prop} \label{prop:approximate3}
For every $p\in\{0\}\cup[1,\infty]$, every controlled propagation
$p$-kernel induces a contractive Schur multiplier on $B^p_u(X)$.
Moreover, if $p\in\{0,1,\infty\}$, then for every
$\varphi\in\mathbb C_u[X,\mathcal E]$,
\[
\|\mathcal M_\varphi\|=\|\varphi\|_{\sup}.
\]
\end{prop}

\begin{proof}
Suppose first that $1<p<\infty$, and let $\varphi$ be the
$p$-kernel induced by a $p$-Higson--Roe function
$\eta:X\to\ell^p(X)_1$. If $q$ is the conjugate exponent of $p$,
then
\[
\varphi(x,y)
=\langle\ |\eta_y|,M_{p,q}(|\eta_x|)\ \rangle .
\]
Since
\[
\||\eta_y|\|_p=1,
\ \|M_{p,q}(|\eta_x|)\|_q=1,
\]
Lemma~\ref{lem:Pisier} implies
\[
\|\mathcal M_\varphi\|\leq 1.
\]

We next consider the extreme exponents. Let
$\varphi\in\mathbb C_u[X,\mathcal E]$. For
$p\in\{1,\infty\}$ and $A\in\mathbb C_u[X,\mathcal E]$, the
standard column-sum and row-sum norm formulas, respectively, give
\[
\|\varphi\circ A\|_p
\leq
\|\varphi\|_{\sup}\|A\|_p.
\]
By density,
\[
\|\mathcal M_\varphi\|_{B(B^p_u(X))}
\leq
\|\varphi\|_{\sup},
\ p\in\{1,\infty\}.
\]
Conversely, for every $x,y\in X$, the matrix unit $e_{xy}$ belongs
to $\mathbb C_u[X,\mathcal E]$ and has norm one. Hence
\[
\|\mathcal M_\varphi\|
\geq
\|\mathcal M_\varphi(e_{xy})\|
=
|\varphi(x,y)|.
\]
Taking the supremum over $x,y\in X$ yields
\[
\|\mathcal M_\varphi\|
=
\|\varphi\|_{\sup},
\ p\in\{1,\infty\}.
\]

For $p=0$, recall from Section~2.2 the canonical isometric
isomorphism
\[
\Gamma:B^0_u(X)\to B^\infty_u(X),
\ T\mapsto T^{''},
\]
under which $T$ and $T^{''}$ have the same matrix coefficients. Thus,
if $\mathcal M_\varphi^{(0)}$ and
$\mathcal M_\varphi^{(\infty)}$ denote Schur multiplication by
$\varphi$ at the two exponents, then
\[
\Gamma\circ\mathcal M_\varphi^{(0)}
=
\mathcal M_\varphi^{(\infty)}\circ\Gamma.
\]
Consequently,
\[
\|\mathcal M_\varphi^{(0)}\|
=
\|\mathcal M_\varphi^{(\infty)}\|
=
\|\varphi\|_{\sup}.
\]

Finally, a controlled propagation $p$-kernel satisfies
$\|\varphi\|_{\sup}\leq1$ by Definition~\ref{def:pkernel}.
Hence its associated Schur multiplier is contractive for
$p\in\{0,1,\infty\}$ as well.
\end{proof}

\begin{lem} (cf. \cite[Lemma 3.5]{HR}, \cite[Remark 6.2]{SW17}, \cite[Theorem 3.1]{Sako}, \cite[Theorem 1.2.4]{Willett})\label{lem:propertyA}
    For a uniformly locally finite unital coarse space $(X,\mathcal{E})$, the following conditions are equivalent: 
    \begin{enumerate}
        \item $(X,\mathcal{E})$ has property A;
        \item For every $\varepsilon>0$ and every $E\in\mathcal{E}$, there exists a map $\xi: X\to\ell^1(X)_{1,+}$ such that:
        \begin{enumerate}
            \item $\{(x,y)\in X\times X: x\in X, y\in\supp(\xi_x)\}$ is an entourage in $\mathcal{E}$;
            \item $\|\xi_x-\xi_y\|_1\leq\varepsilon$ for each $(x,y)\in E$.
        \end{enumerate}
        \item For all (or some) $p$ with $1\leq p<\infty$, every $\varepsilon>0$ and every $E\in\mathcal{E}$, there exists a $p$-Higson-Roe function with $(E,\varepsilon)$-variation;
        \item For all (or some) $p$ with $1\leq p<\infty$, every $\varepsilon>0$ and every $E\in\mathcal{E}$, there exists a controlled $p$-partition of unity with $(E,\varepsilon)$-variation;
        \item For all (or some) $p$ with $1<p<\infty$, every $\varepsilon>0$ and every $E\in\mathcal{E}$, there exists a controlled propagation $p$-kernel with $(E,\varepsilon)$-variation;
        \item For every $\varepsilon>0$ and every $E\in\mathcal{E}$, there exists a controlled propagation positive definite kernel $\varphi: X\times X\to\mathbb{C}$ with $(E,\varepsilon)$-variation.
    \end{enumerate}
\end{lem}

\begin{proof}
    (i) $\Rightarrow$ (ii) Suppose $(X,\mathcal{E})$ has property A.
Fix $\varepsilon>0$ and $E\in\mathcal{E}$. Applying Definition~\ref{def:propertyA} with $\varepsilon/2$ in place of $\varepsilon$,
there exist a controlled set $F\in\mathcal{E}$ and a subset
$A\subseteq F\times\mathbb{N}$ satisfying the conditions of that
definition. Define a vector $\zeta_x\in\ell^1(X)_{+}$ by $\zeta_x(y)=\#(A_x(y))$ and a unit vector $\xi_x:=\zeta_x/\|\zeta_x\|_1$ in $\ell^1(X)_{1,+}$, where 
    \[
    A_x(y)=\{n\in\mathbb{N}: (y,n)\in A_x\}=\{n\in\mathbb{N}:(x,y,n)\in A\}.
    \]
    Here, $\ell^1(X)_+$ denotes the set of vectors $\zeta=(\zeta(y))_{y\in X}\in\ell^1(X)$ such that $\zeta(y)\geq 0$ for all $y\in X$.
    Then the set $\{(x,y)\in X\times X: y\in\supp(\xi_x)\}\subseteq F\in\mathcal{E}$ is controlled and for any $(x,y)\in E$, we have
    \begin{align*}
        \|\xi_x-\xi_y\|_1&=\|\frac{\zeta_x}{\|\zeta_x\|_1}-\frac{\zeta_y}{\|\zeta_y\|_1}\|_1=\frac{\|\|\zeta_y\|_1\zeta_x-\|\zeta_x\|_1\zeta_y\|_1}{\|\zeta_x\|_1\|\zeta_y\|_1}\\
        &\leq\frac{|\|\zeta_y\|_1-\|\zeta_x\|_1|\cdot\|\zeta_x\|_1+\|\zeta_x\|_1\cdot\|\zeta_x-\zeta_y\|_1}{\|\zeta_x\|_1\cdot\|\zeta_y\|_1}\\
        &\leq2\frac{\|\zeta_x-\zeta_y\|_1}{\|\zeta_y\|_1}\\
        &\leq2\frac{\#(A_x\bigtriangleup A_y)}{\#(A_y)}\\
        &<\varepsilon.
    \end{align*}
    
    (ii) $\Rightarrow$ (iii) Condition (ii) is precisely the
$p=1$ case of (iii), after replacing each vector by its absolute value if
necessary. We show at the same time that the ``some $p$'' and ``all $p$''
versions of (iii) are equivalent.

Suppose (iii) holds for some $p\in[1,\infty)$, and fix
$q\in[1,\infty)$, $E\in\mathcal{E}$, and $\varepsilon>0$.
By Lemma~\ref{lem:inequ}(i) and Remark~\ref{rem:inequ}(i), there exist
constants $C=C(p,q)>0$ and $\theta=\min\{p/q,1\}>0$
such that
\[
\|M_{p,q}(f)-M_{p,q}(g)\|_q
\leq C\|f-g\|_p^\theta
\]
for all $f,g\in\ell^p(X)_1$. Choose $\delta>0$ so that $C\delta^\theta<\varepsilon$,
and choose a $p$-Higson--Roe function
$\eta:X\to\ell^p(X)_1$ with $(E,\delta)$-variation. Set
\[
\eta':=M_{p,q}\circ\eta:X\to\ell^q(X)_1.
\]
Since the Mazur map preserves supports, $\eta'$ is a
$q$-Higson--Roe function. Moreover, for $(x,y)\in E$,
\[
\|\eta'_x-\eta'_y\|_q
\leq C\|\eta_x-\eta_y\|_p^\theta
<C\delta^\theta
<\varepsilon.
\]
Thus (iii) holds for every $q\in[1,\infty)$. 

    (iii) $\Rightarrow$ (ii) Fix $\varepsilon>0$ and
$E\in\mathcal{E}$, and choose $p\in[1,\infty)$ for which (iii) holds.
Choose a $p$-Higson--Roe function
$\eta:X\to\ell^p(X)_1$ with $(E,\varepsilon/p)$-variation. For each
$x\in X$, define
\[
    \xi_x(y):=M_{p,1}(|\eta_x|)(y)=|\eta_x(y)|^p.
\]
Then $\xi_x\in\ell^1(X)_{1,+}$, and for $(x,y)\in E$, Remark~\ref{rem:inequ}(i) gives
\begin{align*}
    \|\xi_x-\xi_y\|_1
    &=\|M_{p,1}(|\eta_x|)-M_{p,1}(|\eta_y|)\|_1\\
    &\leq p\||\eta_x|-|\eta_y|\|_p\\
    &\leq p\|\eta_x-\eta_y\|_p
    <\varepsilon.
\end{align*}
This proves (ii).
    
    (ii) $\Rightarrow$ (i) is obvious from the proof in \cite[Theorem 3.1]{Sako}.
    
    (iv) $\Leftrightarrow$ (iii)
We first show that the ``some $p$'' and ``all $p$'' versions of (iv)
are equivalent. Suppose that (iv) holds for some
$p\in[1,\infty)$, and fix $q\in[1,\infty)$,
$E\in\mathcal{E}$, and $\varepsilon>0$. Put
\[
\alpha:=\min\{p/q,1\}.
\]
By Lemma~\ref{lem:inequ}(i) and Remark~\ref{rem:inequ}(i), there is a
constant $C=C(p,q)>0$ such that
\[
\|M_{p,q}(f)-M_{p,q}(g)\|_q
\leq C\|f-g\|_p^\alpha
\]
for $f,g\in\ell^p(I)_1$. Choose $\delta>0$ such that $C^q\delta^{q\alpha/p}<\varepsilon$.
Let $\{\phi_i\}_{i\in I}$ be a controlled $p$-partition of unity with
$(E,\delta)$-variation, and define
\[
\phi_i'(x):=\phi_i(x)^{p/q}.
\]
Then $\{\phi_i'\}_{i\in I}$ is a controlled $q$-partition of unity
subordinate to the same controlled covering. Writing
\[
\phi_x(i):=\phi_i(x),
\ \phi_x'(i):=\phi_i'(x)=M_{p,q}(\phi_x)(i),
\]
we obtain, for $(x,y)\in E$,
\begin{align*}
\sum_{i\in I}|\phi_i'(x)-\phi_i'(y)|^q
&=\|\phi_x'-\phi_y'\|_q^q\\
&\leq C^q\|\phi_x-\phi_y\|_p^{q\alpha}\\
&=C^q
\left(
\sum_{i\in I}|\phi_i(x)-\phi_i(y)|^p
\right)^{q\alpha/p}\\
&<C^q\delta^{q\alpha/p}
<\varepsilon.
\end{align*}
Hence (iv) holds for every $q\in[1,\infty)$.

It remains to compare the $p=1$ cases of (iii) and
(iv), that is, to establish the correspondence between
$1$-Higson--Roe functions and controlled $1$-partitions of unity. On the one hand, let $\varepsilon>0$, $E\in\mathcal{E}$, and $\eta: X\to\ell^{1}(X)_1$ be a map as in (iii). Without loss of generality, we assume that each $\eta_x$ is nonnegative. For every $y\in X$, define $\phi_y:X\to[0,1]$ by $\phi_y(x)=\eta_x(y)$. Thus, we have
    \[
    F:=\{(x,y):x\in\supp(\phi_y)\}=\{(x,y):y\in\supp(\eta_x)\}\in\mathcal{E}.
    \]
    Then $F(\boldsymbol{i_X})=\{\supp(\phi_y): y\in s(F)\}$ is a covering of $X$, and $\{\phi_y\}_{y\in X}$ is subordinated to the covering $F(\boldsymbol{i_X})$. Since $F\in\mathcal{E}$, it follows from Lemma~\ref{lem:Eix} that $F(\boldsymbol{i_X})$ is a controlled covering of $X$.
    Moreover, since $X$ is a uniformly locally finite coarse space, there exists $N\in\mathbb{N}$ such that for each $x\in X$, at most $N$ of the numbers $\phi_y(x)$ are nonzero. It is clear that $\sum_{y\in X}\phi_y(x)=\|\eta_x\|_1=1$ for each $x\in X$ and 
    \[
    \sum_{z\in X}|\phi_z(x)-\phi_z(y)|=\|\eta_x-\eta_y\|_1<\varepsilon
    \]
    whenever $(x,y)\in E$.

    On the other hand, let $\varepsilon>0$, $E\in\mathcal{E}$, and $\{\phi_i\}_{i\in I}$ be a controlled 1-partition of unity on $X$ in (iv) which is subordinated to the controlled covering $\{U_i\}_{i\in I}$. Set $F:=\bigcup_i U_i\times U_i\in\mathcal{E}$. After discarding the indices $i$ for which $\varphi_i\equiv0$, each
remaining $U_i$ is nonempty. Choose a point $x_i\in U_i$ for each
remaining $i$, and define
    \[
    \eta_x:=\sum_i\phi_i(x)\delta_{x_i}\in\ell^1(X)_{1,+}
    \]
    for each $x\in X$, where $\delta_{x_i}$ is the characteristic function of $x_i\in X$. If $\phi_i(x)\neq0$, then $x\in\supp(\phi_i)\subseteq U_i$, $(x,x_i)\in U_i\times U_i$, and
    \[
    \{(x,y):x\in X, y\in\supp(\eta_x)\}\subseteq F\in\mathcal{E}.
    \]
    Moreover, we have
    \[
    \|\eta_x-\eta_y\|_1=\sum_{z\in X}\left|\sum_{i,x_i=z}(\phi_i(x)-\phi_i(y))\right|\leq\sum_{z\in X}\left(\sum_{i,x_i=z}|\phi_i(x)-\phi_i(y)|\right)<\varepsilon
    \]
    whenever $(x,y)\in E$.
    
    (v) $\Leftrightarrow$ (iii) Without loss of generality, assume that each $\eta_x$ is nonnegative. Then
    \[
    1-\varphi(x,y)=1-[\eta_y,\eta_x]=1-\Re[\eta_y,\eta_x].
    \]
    From Lemma~\ref{lem:inequ}(ii), we have \[ C_2\| \eta_y-\eta_x \|_p^\alpha \leq 1-\varphi(x,y) \leq \|\eta_y-\eta_x\|_p \] for some constant $C_2>0$ depending only on $p$, and $\alpha=\max\{2,p\}$ with $\frac{1}{p}+\frac{1}{q}=1$. It follows that (v) and (iii) are equivalent.
    
    (v) $\Rightarrow$ (vi) is obvious since a controlled propagation 2-kernel is a controlled propagation positive definite kernel.

    (vi) $\Rightarrow$ (iii) This follows from the $p=2$ case proved in \cite[Theorem 3.1]{Sako}.
    \end{proof}

\begin{prop}\label{prop:endpointpkernel}
Let $(X,\mathcal E)$ be a uniformly locally finite coarse space.
For every $p\in\{0,1,\infty\}$, every $\varepsilon>0$, and every
$R\in\mathcal E^\Delta$, there exists a controlled propagation
$p$-kernel for the unitized coarse space
$(X,\mathcal E^\Delta)$ with $(R,\varepsilon)$-variation.
\end{prop}

\begin{proof}
Set
\[
S:=R\cup R^{-1}\cup\Delta_X\in\mathcal E^\Delta.
\]
Then $S$ is symmetric. For $x\in X$, write
\[
S_x:=\{y\in X:(y,x)\in S\}.
\]
Since $\mathcal E^\Delta$ is uniformly locally finite and
$\Delta_X\subseteq S$, each $S_x$ is finite and nonempty.

Choose $0<\delta<\min\{\varepsilon,1/2\}$,
and define
\[
\eta_x
:=
(1-\delta)\delta_x
+
\frac{\delta}{\#S_x}\chi_{S_x}
\in\ell^1(X)_{1,+}.
\]
Then $\supp(\eta_x)=S_x$ for every $x\in X$, so
\[
\{(x,y)\in X\times X:y\in\supp(\eta_x)\}=S\in\mathcal E^\Delta.
\]
Thus $\eta:X\to\ell^1(X)_{1,+}$ is a $1$-Higson--Roe function for
$(X,\mathcal E^\Delta)$.

For $p=1$, the induced kernel is
\[
\varphi(x,y)
=
\sum_{z\in\supp(\eta_x)}\eta_y(z).
\]
If $(x,y)\in R$, then $y\in S_x=\supp(\eta_x)$, and hence
\[
\varphi(x,y)
\geq
\eta_y(y)
=
1-\delta+\frac{\delta}{\#S_y}
\geq
1-\delta.
\]
Therefore $1-\varphi(x,y)\leq\delta<\varepsilon$.

For $p\in\{0,\infty\}$, the induced kernel is
\[
\varphi(x,y)
=
\sum_{z\in\supp(\eta_y)}\eta_x(z).
\]
If $(x,y)\in R$, then, since $S$ is symmetric,
$x\in S_y=\supp(\eta_y)$, and therefore
\[
\varphi(x,y)
\geq
\eta_x(x)
=
1-\delta+\frac{\delta}{\#S_x}
\geq
1-\delta.
\]
Again, $1-\varphi(x,y)<\varepsilon$.

In every case, Definition~\ref{def:pkernel} shows that $\varphi$ has
$\mathcal E^\Delta$-controlled propagation.
\end{proof}

\begin{rem}\leavevmode
    \begin{enumerate}
        \item Condition~(ii) in Lemma~\ref{lem:propertyA} is the classical
Higson--Roe condition \cite[Lemma 3.5]{HR} for uniformly discrete metric
spaces of bounded geometry. It exhibits property A as a weak form of
amenability. The equivalence $(i)\Leftrightarrow(ii)$ follows from
\cite[Theorem 3.1]{Sako}, where property A is treated for uniformly locally
finite unital coarse spaces.
        \item For a uniformly discrete metric space of bounded geometry, a
controlled $1$-partition of unity is usually called a metric
$1$-partition of unity; see \cite[Theorem 1.2.4]{Willett}. Such partitions
arise in the standard proof that finite asymptotic dimension implies
property A; see \cite[Lemma 4.3]{HR}, \cite[Remark 11.36]{Roe03}, and
\cite[Theorem 4.3.6]{NY}. Metric $p$-partitions were introduced in
\cite[Definition 6.1]{SW17}. 
\item
Proposition~\ref{prop:endpointpkernel} explains the restriction
$1<p<\infty$ in the $p$-kernel formulation of
Lemma~\ref{lem:propertyA}(v). At the extreme exponents
$p\in\{0,1,\infty\}$, the corresponding controlled-kernel
approximation is automatic for unital coarse spaces, and is automatic
after unitization in general. Thus, at these exponents, existence of
such kernels cannot by itself characterize property A.
    \end{enumerate}
\end{rem}

\subsection{Ghosts in $B^p_u(X)$}

Recall from Definition~\ref{def:subalgebras} that an operator $K\in B(\ell^p(X))$ is $\mathcal{P}$-compact if
        \[
        \lim\limits_{F\in\mathcal{F}}\|K-K\chi_F\|_p=0\ and\ \lim\limits_{F\in\mathcal{F}}\|K-\chi_F K\|_p=0,
        \]
        which is equivalent to $\lim\limits_{F\in\mathcal{F}}\|K-\chi_F K\chi_F\|_p=0$. Then it is clear that every $\mathcal{P}$-compact operator is compact. 

\begin{lem} (cf. \cite{CL}, \cite[Lemma 2.16]{Zhang})\label{lem:compact}
    For every coarse space $(X,\mathcal{E})$ under consideration, we have $\mathcal{K}^p(X)\subseteq B^p_u(X)$ and $\mathcal{K}^p(X)$ is an ideal in $B^p_u(X)$. 
\end{lem}

\begin{proof}
    For any $\xi\in\ell^p(X)$, $F\in\mathcal{F}$ and any operator $T\in B(\ell^p(X))$, by definition, we have 
    \[
    (\chi_FT\chi_F\delta_y)(x)=\begin{cases}
        T(x,y) &\text{if}\ x,y\in F,\\
        0 &\text{otherwise},
    \end{cases}
    \]
    the matrix of $\chi_FT\chi_F$ has finite support, and 
    \begin{align*}
        (\chi_FT\chi_F)\xi&=\sum_{y\in F}\xi(y)\cdot\chi_FT\delta_y\\
        &=\sum_{y\in F}\left(\xi(y)\cdot\sum_{x\in F}(T\delta_y)(x)\delta_x\right)\\
        &=\sum_{x\in F}\left(\sum_{y\in F}T(x,y)\xi(y)\right)\cdot\delta_x.
    \end{align*}
    Thus $\chi_FT\chi_F$ is represented on $\ell^p(X)$ by its finite
matrix of coefficients. Since the coarse structure $\mathcal{E}$ is coarse connected, i.e., every finite subset of $X\times X$ is an entourage, we have $\chi_FT\chi_F\in \mathbb{C}_u[X,\mathcal{E}]$. Hence, $\mathcal{K}^p(X)\subseteq B^p_u(X)$. Moreover, by Definition~\ref{def:subalgebras}, it is clear that
$\mathcal{K}^p(X)$ is a Banach subalgebra of $B^p_u(X)$. To show $\mathcal{K}^p(X)$ is an ideal, it remains to prove that for any operator $A\in\mathbb{C}_u[X,\mathcal{E}]$ and $K\in\mathcal{K}^p(X)$, $KA$ and $AK$ belong to $\mathcal{K}^p(X)$. Without loss of generality, we assume that $A$ is a partial translation operator induced by a bijection $t:S\to R$. For any $\varepsilon>0$, there exists $F\in\mathcal{F}$ such that $\|K-\chi_FK\chi_F\|<\frac{\varepsilon}{2\|A\|+1}$. Let $G:=F\cup t^{-1}(F\cap R)$ and $S_F:=\chi_FK\chi_FA$, then we have $\chi_GS_F\chi_G=S_F$ and
    \begin{align*}
        \|KA-\chi_GKA\chi_G\|&\leq\|KA-S_F\|+\|S_F-\chi_GKA\chi_G\|\\
        &\leq\|KA-S_F\|+\|\chi_G(S_F-KA)\chi_G\|\\
        &\leq2\|K-\chi_FK\chi_F\|\cdot\|A\|<\varepsilon.
    \end{align*}
    It follows that $KA\in\mathcal{K}^p(X)$. The proof for $AK$ is analogous.
\end{proof}

More precisely, suppose that $X$ is infinite. For
$p\in\{0\}\cup(1,\infty)$, every compact operator is
$\mathcal{P}$-compact. For $p\in\{1,\infty\}$, however, compact operators
need not be $\mathcal{P}$-compact.

If $K\in B^\infty_u(X)$ is compact, then its restriction
\[
K_0:=K|_{\ell^0(X)}
\]
belongs to $\mathcal{K}^0(X)\subseteq B^0_u(X)$ and satisfies
$K_0^{''}=K$ under the canonical isometric isomorphism
$B^0_u(X)\cong B^\infty_u(X)$. Hence
$K\in\mathcal{K}^\infty(X)$. Similarly, if
$K\in B^1_u(X)$ is compact, then $K^{'}\in B^\infty_u(X)$ is compact, so
$K^{'}\in\mathcal{K}^\infty(X)$ and therefore
$K\in\mathcal{K}^1(X)$. Consequently,
    \begin{enumerate}
            \item $\mathcal{K}^p(X)=K(\ell^p(X))\subseteq B^p_u(X)$, for $p\in\{0\}\cup(1,\infty)$;
            \item $\mathcal{K}^p(X)=K(\ell^p(X))\cap B^p_u(X)\subsetneq K(\ell^p(X))$, for $p\in\{1,\infty\}$.
    \end{enumerate}
    Thus the $\mathcal{P}$-compact operators are precisely the compact
operators that belong to $B^p_u(X)$.
    
\begin{cor}
    By definition of the smallest coarse structure $\mathcal{E}_{min}$ in Example~\ref{ex:min}(i), we have $\mathcal{K}^p(X)=B^p_u(X,\mathcal{E}_{min})$.
\end{cor}

\begin{rem}\label{rem:trivial}\leavevmode
    \begin{enumerate}
        \item Lemma~\ref{lem:compact} and Remark~\ref{rem:exactly} show that $\mathcal{K}^p(X)=B^p_u(X,\mathcal{E}_{min})$ is the smallest nonzero geometric ideal in $B^p_u(X)$;
        \item Since 
        \[
        \|K-\chi_F K\chi_F\|_{\sup}\leq\|K-\chi_F K\chi_F\|_p,
        \]
        every $\mathcal{P}$-compact operator $K$ belongs to
$C_0(X\times X)$ and hence is a ghost in $B^p_u(X)$. We call such ghosts
\emph{trivial}. Consequently, the smallest nonzero geometric ideal
$\mathcal{K}^p(X)$ is also a ghost ideal consisting entirely of trivial
ghosts.
    \end{enumerate}
\end{rem}

\begin{defn} (cf. \cite[Section 3]{CW05})
    For $p\in\{0\}\cup[1,\infty]$, $G\in B^p_u(X)$, and $E, F\in\mathcal{E}$, we define the local testing functions as follows:
    \begin{enumerate}
        \item \begin{align*}
                &g_{E\times F}: X\times X\to\mathbb{R}^+\\
                &g_{E\times F}(x,y)=\|G_{N_E(x)\times N_F(y)}\|_p;
                \end{align*}
        \item \begin{align*}
                &g_{E\times \Delta}: X\to\mathbb{R}^+\\
                &g_{E\times \Delta}(x)=\|G_{N_E(x)\times N_E(x)}\|_p.
                \end{align*}
    \end{enumerate}  
\end{defn}
Since $x\in N_E(x)$ and $y\in N_F(y)$ for any $x,y\in X$, we have
\[
|G(x,y)|\leq\|G_{N_E(x)\times N_F(y)}\|_p=g_{E\times F}(x,y).
\]

\begin{lem} (cf. \cite[Theorem 3.1]{CW05})\label{lem:ghost}
    For $p\in\{0\}\cup[1,\infty]$ and any $G\in B^p_u(X)$, the following are equivalent:
    \begin{enumerate}
        \item $G$ is a ghost element;
        \item $g_{E\times F}\in C_0(X\times X)$ for some $E, F\in\mathcal{E}$ (equivalently, for any $E, F\in\mathcal{E}$);
        \item $g_{E\times\Delta}\in C_0(X)$ for any $E\in\mathcal{E}$;
        \item $\langle\ G\ \rangle\subseteq C_0(X\times X)$;
        \item $\varphi\circ G\in \mathcal{K}^p(X)$ for any controlled propagation function $\varphi\in\ell^\infty(X\times X)$;
        \item The $E$-truncation $G_E\in \mathcal{K}^p(X)$ for all $E\in\mathcal{E}$;
        \item $\mathbb{C}_u(\langle\ G\ \rangle)\subseteq \mathcal{K}^p(X)$.
    \end{enumerate}
\end{lem}

\begin{proof}
    (i) $\Rightarrow$ (ii) For any $E, F\in\mathcal{E}$, since $\mathcal{E}$ is a uniformly locally finite coarse structure, there exist $n, m\in\mathbb{N}$ such that for any $x,y\in X$ we have $\#N_E(x)\leq n, \#N_F(y)\leq m$. As $G\in C_0(X\times X)$, for any $\varepsilon>0$, there exists a finite subset $Y\subseteq X$ such that 
    \[
    |G(x,y)|<\frac{\varepsilon}{mn}
    \]
    for all $(x,y)\in (X\times X)\setminus(Y\times Y)$.
    Now $Z=N_{E\cup E^{-1}\cup F\cup F^{-1}}(Y)$ is a finite subset of $X$, and for any $(x,y)\in (X\times X)\setminus (Z\times Z)$ we have
    \[
    N_E(x)\times N_F(y)\subseteq (X\times X)\setminus(Y\times Y).
    \]
    Hence, for all $(u,v)\in N_E(x)\times N_F(y)$,
    \[
    |G_{N_E(x)\times N_F(y)}(u,v)|<\frac{\varepsilon}{mn}.
    \]
    And for all $(x,y)\in (X\times X)\setminus(Z\times Z)$, we have
    \begin{align*}
        g_{E\times F}(x,y)&=\|G_{N_E(x)\times N_F(y)}\|_p\\
        &\leq mn\sup_{(x,y)\in N_E(x)\times N_F(y)}\{|G_{N_E(x)\times N_F(y)}(u,v)|\}\\
        &<mn\cdot \frac{\varepsilon}{mn}=\varepsilon.
    \end{align*}
    Hence,
    \[
    g_{E\times F}\in C_0(X\times X).
    \]
    (ii) $\Rightarrow$ (i)
Suppose that
\[
g_{E\times F}\in C_0(X\times X)
\]
for some $E,F\in\mathcal{E}$. Since
$x\in N_E(x)$ and $y\in N_F(y)$ for every $x,y\in X$, we have
\[
|G(x,y)|
\leq
\|G_{N_E(x)\times N_F(y)}\|_p
=
g_{E\times F}(x,y).
\]
Hence $G\in C_0(X\times X)$, so $G$ is a ghost.
\\
(i) $\Rightarrow$ (iii)
The proof of (i)$\Rightarrow$(ii) above shows, in fact, that
\[
g_{E\times F}\in C_0(X\times X)
\]
for every $E,F\in\mathcal{E}$. Taking $F=E$, we obtain
\[
g_{E\times\Delta}(x)
=
g_{E\times E}(x,x),\ x\in X.
\]
Since the restriction to the diagonal of a function in
$C_0(X\times X)$ belongs to $C_0(X)$, it follows that
\[
g_{E\times\Delta}\in C_0(X)
\]
for every $E\in\mathcal{E}$.
    \\
    (iii) $\Rightarrow$ (i)
Fix $\varepsilon>0$ and set
\[
A:=\supp_\varepsilon(G),
\ E:=A\cup A^{-1}.
\]
By Lemma~\ref{lem:supp}, $A\in\mathcal E$, and hence $E\in\mathcal E$. Moreover,
$E$ is symmetric. By (iii), $g_{E\times\Delta}\in C_0(X)$,
so there exists a finite subset $Y\subseteq X$ such that $g_{E\times\Delta}(x)<\varepsilon$ for every $x\in X\setminus Y$.

We claim that $A\subseteq Y\times Y$. Suppose otherwise that
$(x,y)\in A\setminus(Y\times Y)$. Then at least one of $x,y$ lies
outside $Y$. Since $(x,y)\in A\subseteq E$ and $E$ is symmetric, we
have $x,y\in N_E(x)$ and $x,y\in N_E(y)$.
If $x\notin Y$, then
\[
|G(x,y)|
\leq
\|G_{N_E(x)\times N_E(x)}\|_p
=
g_{E\times\Delta}(x)
<
\varepsilon,
\]
contradicting $(x,y)\in\supp_\varepsilon(G)$. The same argument,
using $y$, applies when $y\notin Y$. Hence
\[
\supp_\varepsilon(G)=A\subseteq Y\times Y,
\]
so $\supp_\varepsilon(G)$ is finite. Since $\varepsilon>0$ was
arbitrary, $G\in C_0(X\times X)$, and therefore $G$ is a ghost.
\\
    (i) $\Rightarrow$ (iv) By Remark~\ref{rem:cutoff-ai},
$\{\chi_{r(E)}\}_{E\in\mathcal E}$ is a contractive approximate
identity for $B^p_u(X)$. Hence every operator in $\langle\ G\ \rangle$
can be approximated by finite sums of operators of the form
$AGB$ with $A,B\in B^p_u(X)$. By Lemma~\ref{prop:norm}, it remains to prove $AGB\in C_0(X\times X)$ for any $A, B\in B^p_u(X)$. Without loss of generality, we assume that $A, B\in \mathbb{C}_u[X,\mathcal{E}]$, as any operator in $B^p_u(X)$ can be approximated in operator norm by controlled propagation operators. Let $H=AGB$, $E=\supp(A)\cup\supp(A)^{-1}, F=\supp(B)\cup\supp(B)^{-1}$. Then we have
    \begin{align*}
		|H(x,y)|&=\|E_{ii}AGBE_{jj}\|_p\\
		&=\|E_{ii}A\chi_{N_E(x)}G\chi_{N_F(y)}BE_{jj}\|_p\\
		&\leq\|A\|_p\|B\|_p\|G_{N_E(x)\times N_F(y)}\|_p\\
		&=\|A\|_p\|B\|_p g_{E\times F}(x,y)
	\end{align*}
    for any $(x,y)=(x_i,x_j)\in X\times X$. It follows from (ii) that $H\in C_0(X\times X)$.\\
    (iv) $\Rightarrow$ (i) is obvious.\\
    (i) $\Rightarrow$ (v) Denote $E=\supp(\varphi)\in\mathcal{E}$. By Lemma~\ref{lem:CW}(ii), without loss of generality, we assume that $E$ is a partial translation. As $G\in C_0(X\times X)$, we have $G_E\in C_0(X\times X)$ and
    \[
    \lim\limits_{F\in\mathcal{F}}\|G_E-\chi_F G_E\chi_F\|_p=\lim\limits_{F\in\mathcal{F}}\|G_E-\chi_F G_E\chi_F\|_{\sup}=0.
    \]
    Thus, $G_E\in \mathcal{K}^p(X)$, and there is a diagonal operator $D$ such that $\varphi\circ G=DG_E\in \mathcal{K}^p(X)$ (cf. proof of Proposition~\ref{thm:keythm}(iv)).\\
    (v) $\Rightarrow$ (vi) is obvious.\\
    (vi) $\Rightarrow$ (i) Let $E=\supp_\varepsilon(G)$. Then we have $G_\varepsilon=G_E\in \mathcal{K}^p(X)$. Hence $G_\varepsilon\in C_0(X\times X)$ and then $G\in C_0(X\times X)$.\\
    (vi) $\Rightarrow$ (vii) By Proposition~\ref{prop:truncation}(i), we have $\mathbb{C}_u(\langle\ G\ \rangle)=\{T_E: T\in\langle\ G\ \rangle,E\in\mathcal{E}\}$. It follows from (iv) and (vi) that $\mathbb{C}_u(\langle\ G\ \rangle)\subseteq \mathcal{K}^p(X)$.\\
    (vii) $\Rightarrow$ (vi) is obvious as $G_E\in \mathbb{C}_u(\langle\ G\ \rangle)$ for any $E\in\mathcal{E}$.
\end{proof}

\begin{rem}\label{rem:ghost}
    For $p\in\{0\}\cup[1,\infty]$, we have:
    \begin{enumerate}
        \item For any ghost element $G\in B^p_u(X)$, it follows from Lemma~\ref{lem:ghost}(iv) that $\langle\ G\ \rangle$ is a ghost ideal of $B^p_u(X)$. 
        \item It follows from Lemma~\ref{lem:ghost}(vi) that controlled propagation ghost elements must be trivial ghosts.
        \item It follows from Lemma~\ref{lem:ghost}(iv, vii) that an ideal $I$ of $B^p_u(X)$ is a ghost ideal if and only if $\mathbb{C}_u(I)\subseteq \mathcal{K}^p(X)\cap I$.
        \item If $G\in B^p_u(X)$ is a non-trivial ghost, it follows from Lemma~\ref{lem:ghost}(vii) that $\mathbb{C}_u(\langle\ G\ \rangle)\subseteq \mathcal{K}^p(X)\cap\langle\ G\ \rangle$ and then $\overline{\mathbb{C}_u(\langle\ G\ \rangle)}^{\|\cdot\|_p}\subseteq \mathcal{K}^p(X)\cap\langle\ G\ \rangle\subsetneq\langle\ G\ \rangle$. Thus, $\mathbb{C}_u(\langle\ G\ \rangle)$ is not dense in $\langle\ G\ \rangle$, i.e., $\langle\ G\ \rangle$ is not a geometric ideal.
    \end{enumerate}    
\end{rem}

\subsection{Multiplier approximate identities and ideal consequences}

We now pass from controlled $p$-kernels to their algebraic and
ideal-theoretic consequences. Under property A this applies for every
$p\in\{0\}\cup[1,\infty]$, while at the extreme exponents
$p\in\{0,1,\infty\}$ Proposition~\ref{prop:endpointpkernel}
provides the required kernels without any property A assumption.
The arguments below are inspired by
\cite[Lemma~11.17]{Roe03},
\cite[Lemma~4.3]{CW05},
\cite[Corollary~6.5]{SW17}, and
\cite[Theorem~5.11]{Pisier}.

\begin{lem}\label{lem:approximate2}
    Given $p\in\{0\}\cup[1,\infty]$, for any $\varphi\in\ell^\infty(X\times X)$, $A\in\mathbb{C}_u[X,\mathcal{E}]$ with $\supp(A)=E\in\mathcal{E}$ and $N:=2n(E)-1\in\mathbb{N}$, the Schur product $\varphi\circ A$ belongs to $\mathbb{C}_u[X,\mathcal{E}]$, and we have
    \[
    \|\varphi\circ A\|_p\leq\sup_{(x,y)\in E}|\varphi(x,y)|\cdot N\|A\|_p.
    \]
\end{lem}

\begin{proof}
    It is clear that $\supp(\varphi\circ A)\subseteq\supp(A)\in\mathcal{E}$ and $\varphi\circ A\in\mathbb{C}_u[X,\mathcal{E}]$. It follows from Lemma~\ref{lem:approximate1} and Lemma~\ref{prop:norm} that 
    \begin{align*}
    \|\varphi\circ A\|_p&\leq N\sup_{(x,y)\in E}|\varphi(x,y)\cdot A(x,y)|\\
    &\leq\sup_{(x,y)\in E}|\varphi(x,y)|\cdot N\sup_{(x,y)\in E}|A(x,y)|\\
    &\leq\sup_{(x,y)\in E}|\varphi(x,y)|\cdot N\|A\|_p.
    \end{align*}
\end{proof}

\begin{defn} (cf. \cite[Corollary 11.18]{Roe03})
    A net $\{\varphi_i\}_{i\in I}\subseteq\ell^\infty(X\times X)$ is a multiplier approximate identity for $B^p_u(X)$ if $M_{\varphi_i}:T\mapsto\varphi_i\circ T$ are all bounded Schur multipliers with $\sup_{i\in I}\|M_{\varphi_i}\|< \infty$ and $\varphi_i\circ T\to T$ for every $T\in B^p_u(X)$.

    If, in addition, $\varphi_i\in\mathbb{C}_u[X,\mathcal{E}]$ for all $i\in I$,
we say that $\{\varphi_i\}_{i\in I}$ is a multiplier approximate
identity with controlled propagation.
\end{defn}

Our definition is slightly different from that in
\cite[Corollary~11.18]{Roe03}. There, for a bounded-geometry metric
space coarsely embeddable in Hilbert space, the uniform Roe
$C^*$-algebra is shown to admit such an approximate identity, with the
associated Schur multipliers contractive and completely positive.

\begin{thm}\label{thm:multappr}
Let $p\in\{0\}\cup[1,\infty]$. Suppose either that
$(X,\mathcal E)$ has property A or that
$p\in\{0,1,\infty\}$. Then $B^p_u(X)$ has a multiplier approximate
identity with controlled propagation.
\end{thm}

\begin{proof}
    Suppose first that $(X,\mathcal E)$ is unital. For each
$E\in\mathcal E$ and $\varepsilon>0$, choose a controlled propagation
$p$-kernel $\varphi^{(E,\varepsilon)}$ with
$(E,\varepsilon)$-variation. Such kernels exist by
Lemma~\ref{lem:propertyA} when $1<p<\infty$ under the property A
hypothesis, and by Proposition~\ref{prop:endpointpkernel} when
$p\in\{0,1,\infty\}$ without any property A assumption. Index these kernels by $\mathcal{E}\times(0,\infty)$,
with
\[
(E,\varepsilon)\preceq(E',\varepsilon')
\ \Leftrightarrow\ E\subseteq E'
\ \text{and}\ 
\varepsilon'\leq\varepsilon.
\]
Since $\mathcal{E}$ is closed under finite unions, this is a directed
set. Proposition~\ref{prop:approximate3} gives $\|\mathcal{M}_{\varphi^{(E,\varepsilon)}}\|\leq1$
for every $(E,\varepsilon)$.

We claim that
$\{\varphi^{(E,\varepsilon)}\}$ is a multiplier approximate identity.
Let $T\in B^p_u(X)$ and $\delta>0$. Choose
$A\in\mathbb{C}_u[X,\mathcal{E}]$ such that
\[
\|T-A\|_p<\frac{\delta}{3},
\]
and set
\[
E_0:=\supp(A),
\ N:=2n(E_0)-1.
\]
Choose $\varepsilon_0>0$ so that
\[
\varepsilon_0N\|A\|_p<\frac{\delta}{3}.
\]
If $E\supseteq E_0$ and $0<\varepsilon\leq\varepsilon_0$, then
$\varphi^{(E,\varepsilon)}$ has
$(E_0,\varepsilon_0)$-variation. Hence Lemma~\ref{lem:approximate2} gives
\[
\|A-\varphi^{(E,\varepsilon)}\circ A\|_p
\leq
\varepsilon_0N\|A\|_p
<\frac{\delta}{3}.
\]
Using the contractivity of the corresponding Schur multiplier,
\begin{align*}
\|T-\varphi^{(E,\varepsilon)}\circ T\|_p
&\leq
\|T-A\|_p
+\|A-\varphi^{(E,\varepsilon)}\circ A\|_p\\
&\qquad
+\|\mathcal{M}_{\varphi^{(E,\varepsilon)}}(A-T)\|_p\\
&<
2\frac{\delta}{3}+\frac{\delta}{3}
=\delta.
\end{align*}
Therefore
\[
\lim_{(E,\varepsilon)}
\|T-\varphi^{(E,\varepsilon)}\circ T\|_p=0.
\]
    
    Suppose now that $\mathcal{E}$ is non-unital. Apply the unital case to $(X,\mathcal{E}^\Delta)$, and for
    $E\in\mathcal{E}^\Delta$, $\varepsilon>0$, and
    $F\in\mathcal{E}$, put
    \[
    \psi^{(E,\varepsilon,F)}
    :=\chi_{r(F)}\varphi^{(E,\varepsilon)}\chi_{r(F)}.
    \]
    Here, the right-hand side is viewed as a kernel, so that
    \[
    \psi^{(E,\varepsilon,F)}(x,y)
    =\chi_{r(F)}(x)\varphi^{(E,\varepsilon)}(x,y)\chi_{r(F)}(y).
    \] 
    We first check controlled propagation. Since
    $\supp(\varphi^{(E,\varepsilon)})\in\mathcal{E}^\Delta$, there is
    $H\in\mathcal{E}$ such that
    \[
    \supp(\varphi^{(E,\varepsilon)})\subseteq H\cup\Delta_X.
    \]
    Moreover,
    $\Delta_{r(F)}\subseteq F\circ F^{-1}\in\mathcal{E}$. Hence,
    \[
    \supp(\psi^{(E,\varepsilon,F)})
    \subseteq H\cup\Delta_{r(F)}\in\mathcal{E}.
    \]
    Thus each $\psi^{(E,\varepsilon,F)}$ has
    $\mathcal{E}$-controlled propagation.
    By contractivity of the multipliers associated to the kernels
    $\varphi^{(E,\varepsilon)}$,
    \[
    \mathcal{M}_{\psi^{(E,\varepsilon,F)}}(T)   
    =\chi_{r(F)}\mathcal{M}_{\varphi^{(E,\varepsilon)}}(T)\chi_{r(F)},
    \]
    and therefore
    \[
    \|\mathcal{M}_{\psi^{(E,\varepsilon,F)}}\|\leq 1.
    \]
    Index these kernels by
    $\mathcal{E}^\Delta\times\mathbb{R}^+\times\mathcal{E}$, ordered by
    \[
    (E,\varepsilon,F)\preceq(E',\varepsilon',F')
    \ \Leftrightarrow\ E\subseteq E',\ \varepsilon'\leq\varepsilon,\ F\subseteq F'.
    \]
    Let $T\in B_u^p(X,\mathcal{E})$ and $\delta>0$. By the unital
    argument, choose $(E_0,\varepsilon_0)$ such that for every
    $E\supseteq E_0$ and $0<\varepsilon\leq\varepsilon_0$,
    \[
    \|T-\varphi^{(E,\varepsilon)}\circ T\|_p<\delta/2.
    \]
    
    By Remark~\ref{rem:cutoff-ai},
$\{\chi_{r(F)}\}_{F\in\mathcal E}$ is a contractive approximate
identity for $B^p_u(X,\mathcal E)$. Hence choose $F_0\in\mathcal{E}$ such that
    \[
    \|T-\chi_{r(F_0)}T\chi_{r(F_0)}\|_p<\delta/4.
    \]
    For $F\supseteq F_0$, we have $\chi_{r(F)}\chi_{r(F_0)}=\chi_{r(F_0)}$, and hence
    \[
    \|T-\chi_{r(F)}T\chi_{r(F)}\|_p
    \leq 2\|T-\chi_{r(F_0)}T\chi_{r(F_0)}\|_p<\delta/2.
    \]
    Consequently, for every $(E,\varepsilon,F)\succeq(E_0,\varepsilon_0,F_0)$,
    \[
    \|T-\psi^{(E,\varepsilon,F)}\circ T\|_p\leq \|T-\chi_{r(F)}T\chi_{r(F)}\|_p+\|\chi_{r(F)}(T-\varphi^{(E,\varepsilon)}\circ T)\chi_{r(F)}\|_p<\delta.
    \]
    Thus $\{\psi^{(E,\varepsilon,F)}\}$ is a multiplier approximate
    identity for $B_u^p(X,\mathcal{E})$ with controlled propagation.
\end{proof}

\begin{thm}\label{thm:geomideal}
    For $p\in\{0\}\cup[1,\infty]$, if $B^p_u(X)$ has a multiplier approximate identity with controlled propagation, then all ideals of $B^p_u(X)$ are geometric ideals.
\end{thm}

\begin{proof}
    Let $I$ be an ideal of $B^p_u(X)$ and let $T\in I$. By the definition of a
multiplier approximate identity with controlled propagation, for every
$\varepsilon>0$ there exists
$\varphi\in\mathbb{C}_u[X,\mathcal{E}]$ such that $\|T-\varphi\circ T\|_p<\varepsilon$. By Proposition~\ref{prop:truncation}(i), we have $\mathbb{C}_u(I)=\{\varphi\circ T: T\in I, \varphi\in\mathbb{C}_u[X,\mathcal{E}]\}$. Thus, $\mathbb{C}_u(I)$ is dense in $I$, so $I$ is a geometric ideal.
\end{proof}

\begin{cor}\label{cor:geomideal}
Let $p\in\{0\}\cup[1,\infty]$. If either
$(X,\mathcal E)$ has property A or
$p\in\{0,1,\infty\}$, then every ideal of $B^p_u(X)$ is geometric.
\end{cor}

\begin{prop}\label{prop:trivial}
    For $p\in\{0\}\cup[1,\infty]$, if all ideals of $B^p_u(X)$ are geometric ideals, then all ghosts in $B^p_u(X)$ are trivial ghosts, i.e., the ideal $I^p_G(X)$ of all ghosts coincides with the smallest geometric ideal $\mathcal{K}^p(X)$.
\end{prop}

\begin{proof}
Recall that $\mathcal{E}_{\min}$ consists precisely of the finite
subsets of $X\times X$. An operator $T\in B^p_u(X)$ is a ghost if and
only if $\supp_\varepsilon(T)$
is finite for every $\varepsilon>0$. Hence
\[
I^p_G(X)
=
\widetilde{I}(\mathcal{E}_{\min}).
\]
By Proposition~\ref{prop:ghostly}, this ideal has essential support
$\mathcal{E}_{\min}$.

Now assume that every ideal of $B^p_u(X)$ is geometric. Then
$I^p_G(X)$ is geometric. Proposition~\ref{prop:lattice3} gives a unique
geometric ideal with essential support $\mathcal{E}_{\min}$, namely
\[
I(\mathcal{E}_{\min})
=
B^p_u(X,\mathcal{E}_{\min}).
\]
By Corollary~6.13,
\[
B^p_u(X,\mathcal{E}_{\min})
=
\mathcal{K}^p(X).
\]
Therefore
\[
I^p_G(X)=\mathcal{K}^p(X),
\]
as required.
\end{proof}



\begin{cor}\label{cor:extreme}
For $p\in\{0,1,\infty\}$, $B^p_u(X)$ has a multiplier approximate
identity with controlled propagation, every ideal of $B^p_u(X)$ is
geometric, and
\[
I_G^p(X)=\mathcal K^p(X).
\]
\end{cor}

\begin{proof}
The first assertion follows from Theorem~\ref{thm:multappr}.
Theorem~\ref{thm:geomideal} then implies that every ideal is
geometric, and Proposition~\ref{prop:trivial} gives
\[
I_G^p(X)=\mathcal K^p(X).
\]
\end{proof}

\begin{rem}
In particular, property A implies that every ghost in the
$\ell^p$ uniform Roe algebra is trivial. For
$p\in[1,\infty)$ and a bounded-geometry discrete metric space
$(X,d)$ equipped with its bounded coarse structure
$\mathcal E_d$, this is consistent with
\cite[Proposition~2.4]{BV}.

Fix $p\in\{0\}\cup[1,\infty]$ and consider the conditions
\[
\text{(i)}\Rightarrow
\text{(ii)}\Rightarrow
\text{(iii)}\Rightarrow
\text{(iv)},
\]
where
\begin{enumerate}
\item[(i)] $(X,\mathcal E)$ has property A;
\item[(ii)] $B^p_u(X)$ has a multiplier approximate identity
with controlled propagation;
\item[(iii)] every ideal of $B^p_u(X)$ is geometric;
\item[(iv)] every ghost in $B^p_u(X)$ is trivial, equivalently $I_G^p(X)=\mathcal K^p(X)$.
\end{enumerate}
For $p=2$ and bounded-geometry metric spaces, Roe--Willett
\cite{RW} proved that (iv) implies (i). For
$p\in(1,\infty)\setminus\{2\}$, this converse remains open in the
present generality.

At the extreme exponents $p\in\{0,1,\infty\}$, however,
Corollary~\ref{cor:extreme} shows that
(ii)--(iv) hold for every uniformly locally finite
coarse space. Since uniformly locally finite coarse spaces without
property A exist---for example, coarse disjoint unions of expander
graphs---the implication (iv)$\Rightarrow$(i) fails
in general at these exponents.
\end{rem}

\subsection{Full ideal classification}

\begin{cor}[Full ideal classification and exponent independence] \label{cor:fullclassification}
Let $(X,\mathcal{E})$ be a uniformly locally finite coarse space.

\begin{enumerate}
\item
Suppose $(X,\mathcal{E})$ has property A. Then, for every
$p\in\{0\}\cup[1,\infty]$, every nonzero ideal of
$B^p_u(X,\mathcal{E})$ is uniquely of the form
\[
B^p_u(X,\mathcal{J})
\]
for an ideal $\mathcal{J}$ of $\mathcal{E}$. Consequently,
with the formal zero element $\mathbf{0}$ understood as in
Section~\ref{section4},
\[
I[B^p_u(X)]
\cong
\mathcal{J}[\mathcal{E}].
\]
Moreover, every essential-support block is a singleton:
\[
I[\mathcal{J}]
=
\{B^p_u(X,\mathcal{J})\},
\ I(\mathcal{J})=\widetilde{I}(\mathcal{J}).
\]

\item
For $p\in\{0,1,\infty\}$, all the conclusions in \emph{(i)} hold
without assuming property A.

\item
For $p\in[1,\infty]$, if $(X,\mathcal{E})$ has property A, then every
ideal of $F^p_{\mathrm{red}}(G(X))$ is dynamical. More precisely, every
nonzero ideal is uniquely of the form
\[
I_U
\]
for an invariant open subset $U\subseteq G(X)^{(0)}$, and
\[
I_U=\widetilde{I}_U.
\]
For $p\in\{1,\infty\}$, the same conclusions hold without assuming
property A.
\item
Suppose $(X,\mathcal{E})$ has property A. Then for any
$p,q\in\{0\}\cup[1,\infty]$, there is a canonical lattice
isomorphism
\[
\Theta_{p,q}: I[B^p_u(X)]\to I[B^q_u(X)]
\]
given by
\[
\Theta_{p,q}(0)=0,
\ \Theta_{p,q}\big(B^p_u(X,\mathcal{J})\big)
=
B^q_u(X,\mathcal{J})
\]
for every ideal $\mathcal{J}$ of $\mathcal{E}$.
If $p,q\in\{0,1,\infty\}$, the same conclusion holds without assuming
property A.

\item
Suppose $(X,\mathcal{E})$ has property A. Then for any
$p,q\in[1,\infty]$, the assignment
\[
I^{(p)}_U\mapsto I^{(q)}_U
\]
defines a canonical lattice isomorphism
\[
\mathbb{I}[F^p_{\mathrm{red}}(G(X))]
\cong
\mathbb{I}[F^q_{\mathrm{red}}(G(X))],
\]
where $I^{(r)}_U$ denotes the ideal of
$F^r_{\mathrm{red}}(G(X))$ associated with the invariant open subset
$U\subseteq G(X)^{(0)}$. If $p,q\in\{1,\infty\}$, the same conclusion
holds without assuming property A.
\end{enumerate}
\end{cor}

\begin{proof}
For (i), Corollary~\ref{cor:geomideal} shows that every ideal of
$B^p_u(X)$ is geometric. Proposition~\ref{prop:lattice3} and
Theorem~\ref{thm:geometric} therefore show that every nonzero ideal is
uniquely of the form $B^p_u(X,\mathcal{J})$.

For a fixed $\mathcal{J}$, Proposition~\ref{prop:ghostly} identifies
$I(\mathcal{J})$ and $\widetilde{I}(\mathcal{J})$ as, respectively,
the smallest and largest ideals in the block $I[\mathcal{J}]$. Since
every ideal is geometric and Proposition~\ref{prop:lattice3} gives a
unique geometric ideal with essential support $\mathcal{J}$, the block
contains only that ideal. Hence
\[
I(\mathcal{J})=\widetilde{I}(\mathcal{J}).
\]

Part (ii) follows in exactly the same way from
Corollary~\ref{cor:extreme}.

For (iii), use the isometric isomorphism
\[
\Psi:F^p_{\mathrm{red}}(G(X))\to B^p_u(X)
\]
together with Corollary~\ref{cor:isomorphic}. Under the hypotheses in
(i) or (ii), every ideal on the uniform Roe algebra side
is geometric, so every ideal on the groupoid side is dynamical.
Proposition~\ref{prop:lattice4} then gives the unique parametrization by
invariant open subsets. Finally, Proposition~\ref{prop:restrictive}
shows that $\widetilde{I}_U$ has the same essential support $U$ as
$I_U$; since there is only one ideal with that essential support, we
have
\[
I_U=\widetilde{I}_U.
\]

Part (iv) follows from (i) by identifying the ideal
lattices for the two exponents with the same lattice
$\mathcal{J}[\mathcal{E}]$. The assertion for
$p,q\in\{0,1,\infty\}$ follows instead from (ii).

Similarly, (iii) identifies
$\mathbb{I}[F^p_{\mathrm{red}}(G(X))]$ with
$\mathcal{O}[G(X)]$ by sending $I^{(p)}_U$ to $U$. Composing these
identifications for $p$ and $q$ gives (v). The endpoint
statement follows from the last assertion of (iii).
\end{proof}

\appendix
\section{Coarse structures via partial coverings}\label{appen1}

\begin{defn}\cite[Definition 2.3.1]{LV} \label{def:LV}
    Given a coarse structure $\mathcal{E}$ on a set $X$, a partial covering $\boldsymbol{a}$ of $X$ is controlled if $\mathrm{diag}(\boldsymbol{a}) \in \mathcal{E}$. The family of all controlled partial coverings associated with $\mathcal{E}$ is denoted by $\mathfrak{C}(\mathcal{E})$, i.e., 
    \[
    \mathfrak{C}(\mathcal{E})=\{\boldsymbol{a}: \diag(\boldsymbol{a})\in \mathcal{E}\}.
    \]
    If we want to specify that a partial covering is controlled with respect to a certain coarse structure $\mathcal{E}$, we say that it is $\mathcal{E}$-controlled.
\end{defn}

The family of controlled partial coverings $\mathfrak{C}(\mathcal{E})$ associated with a coarse structure $\mathcal{E}$ (not necessarily containing the diagonal) satisfies the following (cf. \cite[Page 31]{LV}):
\begin{enumerate}
    \item if $\boldsymbol{a}\in\mathfrak{C}(\mathcal{E})$ then $[\boldsymbol{a}]\subseteq\mathfrak{C}(\mathcal{E})$;
    \item if $\boldsymbol{a}\in\mathfrak{C}(\mathcal{E})$ and $\boldsymbol{a'}$ is a refinement of $\boldsymbol{a}$, then $\boldsymbol{a'}\in\mathfrak{C}(\mathcal{E})$;
    \item if $\boldsymbol{a},\boldsymbol{b}\in\mathfrak{C}(\mathcal{E})$, then $St(\boldsymbol{a},\boldsymbol{b}), \boldsymbol{a}\cup\boldsymbol{b}\in\mathfrak{C}(\mathcal{E})$;
    \item $\{\{x,y\}\}\in\mathfrak{C}(\mathcal{E})$ for every $x,y\in X$.
\end{enumerate}

\begin{rem}
    \cite{LV} considers unital coarse structures (i.e., coarse structures containing the diagonal), whereas we allow coarse structures that need not be unital.
\end{rem}

\begin{lem}\cite[Lemma 2.1.1]{LV} \label{lem:LV}
    \begin{align*}
		\diag(St(\boldsymbol{a},\boldsymbol{b}))&=\diag(\boldsymbol{b})\circ \diag(\boldsymbol{a})\circ \diag(\boldsymbol{b}),\\
		\diag(\boldsymbol{a}\cup\boldsymbol{b})&=\diag(\boldsymbol{a})\cup \diag(\boldsymbol{b}).
	\end{align*}
\end{lem}

\begin{prop}
    For any equivalence class $[\boldsymbol{a}]$, there is a maximum element $\boldsymbol{m_a}\in[\boldsymbol{a}]$ and a minimum element $\boldsymbol{s_a}\in[\boldsymbol{a}]$ such that $\boldsymbol{s_a}\leq\boldsymbol{a'}\leq\boldsymbol{m_a}$ for any $\boldsymbol{a'}\in[\boldsymbol{a}]$.
\end{prop}

\begin{proof}
    Let $\boldsymbol{s_a}=\{\{x,y\}: (x,y)\in\diag(\boldsymbol{a}) \}$. It is straightforward to show that $\boldsymbol{s_a}\in [\boldsymbol{a}]$ and $\boldsymbol{s_a}$ is a minimum partial covering in $[\boldsymbol{a}]$.
    
    On the other hand, let $\boldsymbol{m_a}:=\{A\subseteq X:\forall x,y\in A,\,
(x,y)\in\diag(\boldsymbol{a})\}$.
Then $\boldsymbol{m_a}$ is a maximum partial covering in
$[\boldsymbol{a}]$.
\end{proof}

\begin{ex}
    Let $X=\mathbb{Z}$.
    \begin{align*}
        \boldsymbol{a_1}=&\{\{1,2\},\{1,3\},\{2,3\},\{3,4\},\{1\},\{2\},\{3\},\{4\},\{5\}\},\\
        \boldsymbol{a_2}=&\{\{1,2\},\{1,3\},\{2,3\},\{3,4\},\{5\}\},\\
        \boldsymbol{a_3}=&\{\{1,2,3\},\{1,2\},\{1,3\},\{2,3\},\{3,4\},\{1\},\{2\},\{3\},\{4\},\{5\}\},\\
        \boldsymbol{a_4}=&\{\{1,2,3\},\{3,4\},\{5\}\}   
    \end{align*}
    are equivalent partial coverings; $\boldsymbol{a_1},\boldsymbol{a_2}$ are minimum elements and
$\boldsymbol{a_3},\boldsymbol{a_4}$ are maximum elements of their equivalence
class.
\end{ex}

\begin{rem}
The example shows that minima and maxima in $[\boldsymbol a]$ need
not be unique, since refinement is only a preorder. They become unique
after restricting to partial coverings in which no distinct members
contain one another; compare $\boldsymbol{a_2}$ and
$\boldsymbol{a_4}$ above.
\end{rem}

If $\boldsymbol{a}\leq\boldsymbol{b}$, then it is clear that $\diag(\boldsymbol{a})\subseteq\diag(\boldsymbol{b})$. The converse need not hold, but we have the following:

\begin{lem}\label{lem:sm}
    If $\diag(\boldsymbol{a})\subseteq\diag(\boldsymbol{b})$, then $\boldsymbol{a'}\leq\boldsymbol{m_b}$ for any $\boldsymbol{a'}\in[\boldsymbol{a}]$ and $\boldsymbol{s_a}\leq\boldsymbol{b'}$ for any $\boldsymbol{b'}\in[\boldsymbol{b}]$.  
\end{lem}

\begin{proof}
    If $\diag(\boldsymbol{a})\subseteq\diag(\boldsymbol{b})$, we have $\boldsymbol{s_a}\leq\boldsymbol{s_b}$ and $\boldsymbol{m_a}\leq\boldsymbol{m_b}$. Thus, we have $\boldsymbol{a'}\leq\boldsymbol{m_b}$ for any $\boldsymbol{a'}\in[\boldsymbol{a}]$ and $\boldsymbol{s_a}\leq\boldsymbol{b'}$ for any $\boldsymbol{b'}\in[\boldsymbol{b}]$.
\end{proof}

\begin{lem} (cf. \cite[Lemma 2.3]{DH}) \label{lem:DH}
    For any two relations $E, F\subseteq X\times X$ and two partial coverings $\boldsymbol{a}, \boldsymbol{b}$ of $X$, if $E\subseteq \diag(\boldsymbol{a})$ and $F\subseteq \diag(\boldsymbol{b})$, then
    \[
    E\circ F\subseteq \diag(St(\boldsymbol{b},\boldsymbol{a}\cup\boldsymbol{b})),\ F\circ E\subseteq \diag(St(\boldsymbol{b},\boldsymbol{a}\cup\boldsymbol{b})).
    \]
\end{lem}

\begin{proof}
    We have
    \begin{align*}
		E\circ F&\subseteq \diag(\boldsymbol{a})\circ \diag(\boldsymbol{b})\\
		&=\left(\bigcup_{A\in\boldsymbol{a}}(A\times A)\right)\circ\left(\bigcup_{B\in\boldsymbol{b}}(B\times B)\right)\\
		&=\bigcup_{A\in\boldsymbol{a},B\in\boldsymbol{b}}(A\times A)\circ(B\times B)\\
		&\subseteq\bigcup_{A\in\boldsymbol{a},B\in\boldsymbol{b}, A\cap B\neq\emptyset}(A\cup B)\times (A\cup B) \\
        &=\diag(\{A\cup B: A\in\boldsymbol{a}, B\in\boldsymbol{b}, A\cap B\neq\emptyset\}).
	\end{align*}
    Note that if $A\in\boldsymbol{a}$, $B\in\boldsymbol{b}$, and $A\cap B\neq\emptyset$, then $A\subseteq St(B,\boldsymbol{a})$.
    Thus, we also have
    \begin{align*}
		\{A\cup B: A\in\boldsymbol{a}, B\in\boldsymbol{b}, A\cap B\neq\emptyset\} &\leq \{St(B,\boldsymbol{a})\cup B: B\in\boldsymbol{b}\} \\
        &\leq \{St(B,\boldsymbol{a})\cup St(B,\boldsymbol{b}): B\in\boldsymbol{b}\} \\
        &= St(\boldsymbol{b},\boldsymbol{a}\cup\boldsymbol{b}).
	\end{align*}
    Hence
    \[
		E\circ F \subseteq \diag(St(\boldsymbol{b},\boldsymbol{a}\cup\boldsymbol{b})).
    \]
    Similarly, we have $F\circ E\subseteq \diag(St(\boldsymbol{b},\boldsymbol{a}\cup\boldsymbol{b}))$. In fact, the inequalities also hold for $\diag(St(\boldsymbol{a},\boldsymbol{a}\cup\boldsymbol{b}))$.        
\end{proof}

\begin{lem}\label{lem:Eix}
    For any $E\in\mathcal{E}$, we have $\diag(E(\boldsymbol{i_X}))=E\circ E^{-1}\in\mathcal{E}$. In particular, $E(\boldsymbol{i_X})\in\mathfrak{C}(\mathcal{E})$. 
    
    If $\Delta_{s(E)}\subseteq E$, then  $E(\boldsymbol{i_X})\in\mathfrak{C}(\mathcal{E})$ implies $E\subseteq E\circ E^{-1}\in\mathcal{E}$. 
\end{lem}

\begin{proof}
    For any $E\subseteq X\times X$, we have
    \begin{align*}
		E\circ E^{-1}&=\{(y,z): \exists x\in X,\ (y,x),(z,x)\in E\}\\
		&=\{(y,z): y,z\in E_{x},\ x\in s(E)\}\\
		&=\bigcup_{x\in s(E)}(E_{x}\times E_{x})\\
		&=\diag(E(\boldsymbol{i_X})).
	\end{align*} 
    Since $E\in\mathcal{E}$ and $\mathcal{E}$ is closed under inverses and
compositions,
\[
\diag(E(\boldsymbol{i_X}))
=E\circ E^{-1}\in\mathcal{E}.
\]
Hence
\[
E(\boldsymbol{i_X})\in\mathfrak{C}(\mathcal{E}).
\]
    
    If $\Delta_{s(E)}\subseteq E$, then from $\diag(E(\boldsymbol{i_X}))\in\mathcal{E}$, we have $E\subseteq (E\cup\Delta_{s(E)})\circ (E^{-1}\cup\Delta_{s(E)})=E\circ E^{-1}=\diag(E(\boldsymbol{i_X}))\in\mathcal{E}$.
\end{proof}

\begin{prop}\label{prop:cs}
    Given a family $\mathfrak{C}$ of partial coverings of $X$ satisfying:
    \begin{enumerate}
        \item[(C0)] if $\boldsymbol{a}\in\mathfrak{C}$ then $[\boldsymbol{a}]\subseteq\mathfrak{C}$;
        \item[(C1)] if $\boldsymbol{a}\in\mathfrak{C}$ and $\boldsymbol{a'}$ is a refinement of $\boldsymbol{a}$, then $\boldsymbol{a'}\in\mathfrak{C}$;
        \item[(C2)] if $\boldsymbol{a},\boldsymbol{b}\in\mathfrak{C}$, then $St(\boldsymbol{a},\boldsymbol{b}), \boldsymbol{a}\cup\boldsymbol{b}\in\mathfrak{C}$;
        \item[(C3)] $\{\{x,y\}\}\in\mathfrak{C}$ for every $x,y\in X$, 
    \end{enumerate}
    let
    \[
    \mathcal{E}(\mathfrak{C})=\{E\subseteq X\times X: \exists \boldsymbol{a}\in\mathfrak{C},E\subseteq \diag(\boldsymbol{a})\}.
    \]
    Then $\mathcal{E}(\mathfrak{C})$ is a coarse structure on $X$. 
\end{prop}

\begin{proof}
    It is clear that $\mathcal{E}(\mathfrak{C})$ is closed under taking subsets.
    For any $E, F\in\mathcal{E}(\mathfrak{C})$ and $\boldsymbol{a}, \boldsymbol{b}\in\mathfrak{C}$ such that $E\subseteq \diag(\boldsymbol{a}), F\subseteq  \diag(\boldsymbol{b})$, we have
    \begin{itemize}
        \item $E^{-1}\subseteq \diag(\boldsymbol{a})$,
        \item $E\circ F\subseteq \diag(St(\boldsymbol{b}, \boldsymbol{a}\cup\boldsymbol{b}))$,
        \item $E\cup F\subseteq \diag(\boldsymbol{a})\cup \diag(\boldsymbol{b})= \diag(\boldsymbol{a}\cup\boldsymbol{b})$,
    \end{itemize}
    so $E^{-1}, E\circ F, E\cup F\in\mathcal{E}(\mathfrak{C})$.
    
    Finally, for any singleton set $\{(x, y)\}$, where $x,y\in X$, we have $\{(x, y)\}\subseteq\diag( \{\{x,y\}\} )$ so $\{(x,y)\}\in\mathcal{E}(\mathfrak{C})$.
\end{proof}

\begin{rem}
    When $\boldsymbol{i_X}\in \mathfrak{C}$, the conditions (C0-C2) above are equivalent to the following conditions in \cite{LV}:
    \begin{enumerate}
        \item if $\boldsymbol{a}\in\mathfrak{C}$, then $\boldsymbol{a}\cup\boldsymbol{i_X}\in\mathfrak{C}$;
        \item if $\boldsymbol{a}\in\mathfrak{C}$ and $\boldsymbol{a'}$ is a refinement of $\boldsymbol{a}$, then $\boldsymbol{a'}\in\mathfrak{C}$;
        \item if $\boldsymbol{a},\boldsymbol{b}\in\mathfrak{C}$, then $St(\boldsymbol{a},\boldsymbol{b})\in\mathfrak{C}$. 
    \end{enumerate}
    In this case, $\mathcal{E}(\mathfrak{C})$ is a unital coarse structure.
\end{rem}

\begin{prop} (cf. \cite[Proposition 2.3.3]{LV}) \label{prop:LV}
    For any set $X$, there is a one-to-one correspondence
    \[\begin{tikzcd}
  \{\text{coarse structures on $X$}\} \arrow[d, "\lambda_1", shift left=1ex] \\
  \{\text{families of partial coverings of $X$ satisfying (C0)-(C3)}\} \arrow[u, "\lambda_2", shift left=1ex]
    \end{tikzcd}\]
\end{prop}

\begin{proof}
    Given a family $\mathfrak{C}$ of partial coverings of $X$ satisfying (C0)-(C3), we have
    \[
    \mathfrak{C}(\mathcal{E}(\mathfrak{C}))=\{\boldsymbol{a}: \diag(\boldsymbol{a})\in\mathcal{E}(\mathfrak{C})\}=\{\boldsymbol{a}: \exists \boldsymbol{b}\in\mathfrak{C}, \diag(\boldsymbol{a})\subseteq \diag(\boldsymbol{b})\}.
    \]
    It is obvious that $\mathfrak{C}\subseteq\mathfrak{C}(\mathcal{E}(\mathfrak{C}))$. On the other hand, for any $\boldsymbol{a}\in\mathfrak{C}(\mathcal{E}(\mathfrak{C}))$, it follows from Lemma~\ref{lem:sm} that if $\boldsymbol{b}\in\mathfrak{C}$ and $\diag(\boldsymbol{a})\subseteq \diag(\boldsymbol{b})$, then $\boldsymbol{a}\leq\boldsymbol{m_b}\in\mathfrak{C}$. Thus, $\mathfrak{C}(\mathcal{E}(\mathfrak{C}))\subseteq\mathfrak{C}$, hence $\mathfrak{C}(\mathcal{E}(\mathfrak{C}))=\mathfrak{C}$.

    Now let $\mathcal{E}$ be any coarse structure on $X$. Then \[
    \mathcal{E}(\mathfrak{C}(\mathcal{E}))=\{E\subset X\times X: \exists\boldsymbol{a}\in\mathfrak{C}(\mathcal{E}), E\subseteq\diag(\boldsymbol{a})\in\mathcal{E}\}.\] It is clear that $\mathcal{E}(\mathfrak{C}(\mathcal{E}))\subseteq\mathcal{E}$. It remains to show that $\mathcal{E}\subseteq\mathcal{E}(\mathfrak{C}(\mathcal{E}))$. For any $E\in\mathcal{E}$, we have $\Delta_{s(E)}=\{(x,x): x\in s(E)\}\subseteq E^{-1}\circ E\in\mathcal{E}$ and $E\subseteq(E\cup\Delta_{s(E)})\circ(E\cup\Delta_{s(E)})^{-1}\in\mathcal{E}$. Let $\boldsymbol{a}=(E\cup\Delta_{s(E)})(\boldsymbol{i_X})$. It follows from Lemma~\ref{lem:Eix} that $E\subseteq(E\cup\Delta_{s(E)})\circ(E\cup\Delta_{s(E)})^{-1}=\diag(\boldsymbol{a})\in\mathcal{E}$. Hence, $\mathcal{E}\subseteq\mathcal{E}(\mathfrak{C}(\mathcal{E}))$.
\end{proof}

\section{Limit coarse spaces and limit operators}\label{appen2}

In this section, we recall some basic properties of the coarse groupoid
$G(X)$ from \cite{STY,Roe03} and introduce limit coarse spaces and limit
operators. The construction is motivated by the metric approach of
\cite{SW17}.

The appendix proceeds in three stages. We first describe the limit coarse
space associated with an ultrafilter by means of partial translations and
local coordinates. We then define limit operators and prove their
boundedness and multiplicativity. Finally, we identify these limit
operators with the regular representations of the coarse groupoid. This
last identification is the input used in Section~\ref{section5}.

Let $E\in\Gamma_\mathcal{E}$ be a partial translation. Since $r$ and $s$
are injective on $E$, setting $S:=s(E)$ and $R:=r(E)$ determines a
bijection
\[
t:S\to R
\]
such that
\[
E=\Gamma_t:=\{(t(x),x):x\in S\}.
\]
Conversely, we call any bijection $t:S\to R$ whose graph is controlled a
partial translation.

Recall that every uniformly locally finite coarse structure $\mathcal{E}$ can be generated by $\Gamma_{\mathcal{E}}$, the collection of all partial translations in $\mathcal{E}$. Hence
\[
G(X):=\bigcup_{E\in\mathcal{E}}\overline{E}^{\beta(X\times X)}=\bigcup_{\Gamma_t\in\Gamma_{\mathcal{E}}}\overline{\Gamma_t}^{\beta(X\times X)}\subseteq\beta(X\times X),
\]
and
\[
    G(X)^{(0)}=\bigcup_{ A\subseteq X,\Delta_A\in\mathcal{E}}\overline{A}^{\beta X}=\bigcup_{ E\in\mathcal{E}}\overline{r(E)}^{\beta X}=\bigcup_{ \Gamma_t\in\Gamma_{\mathcal{E}}}\overline{r(\Gamma_t)}^{\beta X}\subseteq\beta X.
\]
This description allows us to analyze the structure of the coarse groupoid
$G(X)$ using partial translations. It was shown in \cite[Lemma 2.7]{STY} that the map $(r,s):G(X)\to\beta X\times\beta X$ is injective, as a set $G(X)$ identifies with $\bigcup_{E\in\mathcal{E}}\overline{E}^{\beta X\times\beta X}$, and the groupoid operations are the restrictions of the pair groupoid operations from $\beta X\times\beta X$.

\begin{defn}\cite[Definition 3.2]{SW17}\label{def:compatible}
    Fix an ultrafilter $\omega\in G(X)^{(0)}$. A partial translation $t:S\to R$ on $X$ is compatible with $\omega$ if $\omega(S)=1$ (i.e. $\omega$ is in the closure of $S$ in $\beta X$). If $\omega$ is compatible with $t$, regard $t$ as a map
$t:S\to\beta X$ and define
    \[
    t(\omega):=\lim_{\omega}t\in G(X)^{(0)}
    \]
    to be the ultralimit of $t$ along $\omega$, or the $\omega$-limit of $t$. 
\end{defn}

For a fixed ultrafilter $\omega\in G(X)^{(0)}$, an ultrafilter $\alpha\in G(X)^{(0)}$ is compatible with $\omega$ if there exists a partial translation $t$ that is compatible with $\omega$ and satisfies $t(\omega)=\alpha$. Denote this relation by $\alpha\overset{t}{\sim}\omega$, or by $\alpha\sim\omega$ when the partial translation $t$ is understood.

Denote by $X(\omega)$ the collection of all ultrafilters in $G(X)^{(0)}\subseteq\beta X$ that are compatible with $\omega$.

\begin{lem} (cf. \cite[Remark 3.3 and Lemma C.3]{SW17}, \cite[discussion in 10.18-10.24]{Roe03})\label{lem:compatible}
    For a given ultrafilter $\omega\in G(X)^{(0)}$, the following are equivalent:
    \begin{enumerate}
        \item The ultrafilter $\alpha\in G(X)^{(0)}$ is compatible with $\omega\in G(X)^{(0)}$, i.e., $\alpha\in X(\omega)$;
        \item There exists a partial translation $t:S\to R$ such that $\omega(S)=1$ and such that for any $Z\subseteq X$, we have $\alpha(Z)=1$ if and only if $\omega(t^{-1}(Z\cap R))=1$;
        \item There exists a partial translation $t:S\to R$ and \[\delta\in\overline{\Gamma_t}^{\beta(X\times X)}:=\overline{\{(t(x),x):x\in S\}}^{\beta(X\times X)}\] such that $r(\delta)=\alpha$ and $s(\delta)=\omega$; equivalently,
            \[
            (\alpha,\omega)\in\overline{\Gamma_t}^{\beta X\times\beta X}=\overline{\{(t(x),x):x\in S\}}^{\beta X\times\beta X};
            \]
        \item There exists $E\in\mathcal{E}$ such that $(\alpha,\omega)\in\overline{E}^{\beta X\times\beta X}$;
        \item $\alpha\in r(G(X)_\omega)$.    
    \end{enumerate} 
\end{lem}

\begin{proof}
    (i) $\Rightarrow$ (ii) follows from
\cite[Remark 3.3]{SW17} by unpacking the definition of the ultralimit $\lim_\omega$;
    
      (ii) $\Rightarrow$ (iii)
Since $\omega(S)=1$, choose a net $\{x_\lambda\}_{\lambda\in\Lambda}$
in $S$ converging to $\omega$. We claim that
\[
t(x_\lambda)\to\alpha.
\]
Indeed, if $Z\subseteq X$ satisfies $\alpha(Z)=1$, then by (ii), we have $\omega\big(t^{-1}(Z\cap R)\big)=1$.
Hence $x_\lambda\in t^{-1}(Z\cap R)$ eventually, and therefore
$t(x_\lambda)\in Z$ eventually. This proves the claim.

Passing to a subnet if necessary, $(t(x_\lambda),x_\lambda)$ converges in
$\overline{\Gamma_t}^{\,\beta(X\times X)}$ to some $\delta$.
By continuity of the range and source maps, we have $r(\delta)=\alpha$ and $s(\delta)=\omega$.
Thus (iii) holds.
    
    (iii) $\Rightarrow$ (iv) and (iv) $\Rightarrow$ (v) are obvious; 
    
    (v) $\Rightarrow$ (i) by \cite[Lemma C.3]{SW17}.
\end{proof}

\begin{rem}\label{rem:compatible}\leavevmode
    \begin{enumerate}
        \item For a given $\omega\in X\subseteq G(X)^{(0)}$ (i.e., $\omega$ is a principal ultrafilter), it follows from (ii) that if an ultrafilter $\alpha\in G(X)^{(0)}$ is compatible with $\omega$, then $\alpha$ is necessarily a principal ultrafilter. Moreover, we have $\{(\alpha,\omega)\}\in\mathcal{E}$ for any $\alpha\in X\subseteq G(X)^{(0)}$ (since $\mathcal{E}$ is coarse connected). Combining with (iv) above, we have $\alpha\in X(\omega)$. Thus, $X(\omega)$, the collection of ultrafilters compatible with $\omega$, consists precisely of all principal ultrafilters on $X$, i.e., all points of $X$ itself.
        \item For ultrafilters in $G(X)^{(0)}=\bigcup_{ A\subseteq X,\Delta_A\in\mathcal{E}}\overline{A}^{\beta X}$, it follows from (iii) that compatibility $\sim$ is an equivalence relation on $G(X)^{(0)}$, and $X\subseteq G(X)^{(0)}$ is one of the equivalence classes. Moreover, for each $\omega\in G(X)^{(0)}$, by \cite[Lemma C.3]{SW17}, the map
        \[
        F:X(\omega)\to G(X)_\omega,\ \alpha\mapsto(\alpha,\omega)
        \] is a bijection. $X(\omega)$ is the smallest invariant subset of $G(X)^{(0)}$ containing $\omega$. Since any nonempty open subset of $G(X)^{(0)}$ necessarily contains points in $X$, thus $X$ is the smallest nonempty invariant open subset of
$G(X)^{(0)}$, and it is dense. Consequently, every nonempty invariant
open subset $U\subseteq G(X)^{(0)}$ contains $X$ and, being invariant,
is a union of compatibility classes. Hence one may write
\[
U
=
X\sqcup\bigsqcup_{\omega\in I}X(\omega),
\]
where
\[
I\subseteq\partial G(X)^{(0)}
\]
contains one representative from each compatibility class occurring
in $U\setminus X$.
        \item For a partial translation $t:S\to R$, we have
        \[
        \overline{\Gamma_t}^{\beta X\times\beta X}=\bigcup_{\omega\in G(X)^{(0)}}\{(t(\omega),\omega):\omega(S)=1\}.
        \]
    \end{enumerate}
\end{rem}

The next lemma from \cite{SW17} gives the following essential uniqueness property: if $\omega$ is an ultrafilter, then
any two partial translations $t_1, t_2$ that are compatible with $\omega$ and such that $t_1(\omega)=t_2(\omega)$ are essentially the same, where ``essentially'' means ``off'' a set of $\omega$-measure zero.

\begin{lem}\cite[Lemma 3.4]{SW17}
    Let $(X,\mathcal{E})$ be a coarse space, and $\omega$ be an ultrafilter in $G(X)^{(0)}\subseteq\beta X$. Let $t_1: S_1\to R_1$ and $t_2: S_2\to R_2$ be two partial translations compatible with $\omega$ such that $t_1(\omega)=t_2(\omega)$. If
    \[
    S:=\{x\in S_1\cap S_2: t_1(x)=t_2(x)\},
    \]
    then $\omega(S)=1$.
\end{lem}

For a given ultrafilter $\omega$, and $\alpha\in X(\omega)$, let $[t_\alpha]$ denote the class of partial translations witnessing the compatibility of $\alpha$ with $\omega$, i.e.,
\[
[t_\alpha]:=\{t:S\to R: \Gamma_t\in\Gamma_\mathcal{E},\ \alpha\overset{t}{\sim}\omega\}.
\]
By the preceding lemma, any two members of $[t_\alpha]$ agree on a set of $\omega$-measure one. 

We say that a compatible family for $\omega$ is a collection of partial translations $\{t_\alpha:S_\alpha\to R_\alpha\}_{\alpha\in X(\omega)}$ indexed by $X(\omega)$, where each $t_\alpha$ is a partial translation chosen from $[t_\alpha]$.

It follows from Remark~\ref{rem:compatible} that $X(\omega)=X$ for any $\omega\in X$. The set $X(\omega)$ should be thought of as the collection of ultrafilters that are at a controlled ``distance'' from $\omega$. We now equip $X(\omega)$ with a natural coarse structure.

\begin{defn}\label{def:limitcoarse1}
    Fix an ultrafilter $\omega\in G(X)^{(0)}$. For each $E\in\mathcal{E}$, define a subset $E_\omega$ of $X(\omega)\times X(\omega)$ by
    \[
    E_\omega:=\{(\alpha,\beta)\in X(\omega)\times X(\omega): \exists \Gamma_t\in\Gamma_{\mathcal{E}},\ \Gamma_t\subseteq E\ \text{such that}\ \alpha\overset{t}{\sim}\beta\}.
    \]
    Set
\[
\mathcal{E}_\omega:=\{E_\omega:E\in\mathcal{E}\}.
\]
\end{defn}

\begin{rem}\leavevmode\label{rem:ultra1}
    \begin{enumerate}
        \item We claim $E_\omega=(X(\omega)\times X(\omega))\cap\overline{E}^{\beta X\times\beta X}$. On the one hand, for $(\alpha,\beta)\in E_\omega$, it follows from Lemma~\ref{lem:compatible} that $(\alpha,\beta)\in\overline{\Gamma_t}^{\beta X\times\beta X}\subseteq\overline{E}^{\beta X\times\beta X}$. On the other hand, from Lemma~\ref{lem:CW}(ii), it follows that $E=\Gamma_{t_1}\sqcup\Gamma_{t_2}\sqcup\cdots\sqcup\Gamma_{t_l}$, where $\{t_i\}^{l}_{i=1}$ is a finite family of partial translations. Hence, we have
        \[
        \overline{E}^{\beta X\times\beta X}=\overline{\Gamma_{t_1}}^{\beta X\times\beta X}\cup\overline{\Gamma_{t_2}}^{\beta X\times\beta X}\cup\cdots\cup\overline{\Gamma_{t_l}}^{\beta X\times\beta X}.
        \]
        Thus, for $(\alpha,\beta)\in (X(\omega)\times X(\omega))\cap\overline{E}^{\beta X\times\beta X}$, we have $(\alpha,\beta)\in\overline{\Gamma_{t_i}}^{\beta X\times\beta X}\subseteq\overline{E}^{\beta X\times\beta X}$ for some $1\leq i\leq l$. Thus $(\alpha,\beta)\in E_\omega$.
        \item It is clear that $E_\alpha=E_\omega$ for each $\alpha\in X(\omega)$. Then we have  
        \[
        \overline{E}^{\beta X\times\beta X}=\bigcup_{\omega\in G(X)^{(0)}}E_\omega=E\sqcup\bigcup_{\omega\in \partial G(X)^{(0)}}E_\omega
        \]
        and
        \[
        G(X)=(X\times X)\sqcup G(X)_{\partial G(X)^{(0)}}
        \]
        where $X\times X$ is an open subgroupoid of $G(X)$.
        When $E=\Gamma_t$ is a partial translation,
        \[
        \overline{\Gamma_t}^{\beta X\times\beta X}=\bigcup_{\omega\in G(X)^{(0)}}(\Gamma_t)_\omega=\Gamma_t\sqcup\bigcup_{\omega\in \partial G(X)^{(0)}}(\Gamma_t)_\omega.
        \]
        It follows from Remark~\ref{rem:compatible}(iii) that \begin{align*}(\Gamma_t)_\omega &=\bigcup_{\gamma\in X(\omega)}\{(t(\gamma),\gamma):\gamma(S)=1\} \\ &=\bigcup_{\gamma\in \overline{S}\cap X(\omega)}\{(t(\gamma),\gamma)\}.\end{align*}
    \end{enumerate} 
\end{rem}

\begin{prop} \label{B7}
    For each $\omega\in G(X)^{(0)}$, the collection $\mathcal{E}_\omega$ of subsets in $X(\omega)\times X(\omega)$ satisfies: 
    \begin{enumerate}
        \item For any $A_\omega,B_\omega\in\mathcal{E}_\omega$, the sets $A_\omega^{-1}$, $A_\omega\circ B_\omega$ and $A_\omega\cup B_\omega$ are entourages;
        \item Every finite subset of $X(\omega)\times X(\omega)$ is contained in some entourage of $\mathcal{E}_\omega$.
    \end{enumerate}
\end{prop}

\begin{proof}
    (i) It is clear from Definition~\ref{def:limitcoarse1} that $A_\omega^{-1}=(A^{-1})_\omega$ and $A_\omega\cup B_\omega=(A\cup B)_\omega$ are entourages in $\mathcal{E}_\omega$. Decompose $A=\bigcup_iA_i$ and $B=\bigcup_jB_j$ into finitely many partial translations. Then
    \[
    A_\omega\circ B_\omega=\bigcup_i(A_i)_\omega\circ\bigcup_j(B_j)_\omega=\bigcup_{i,j}(A_i)_\omega\circ(B_j)_\omega,
    \]
    which is a finite union. By definition, $(A_i)_\omega\circ (B_j)_\omega=(A_i\circ B_j)_\omega$ are entourages in $\mathcal{E}_\omega$. Hence $A_\omega\circ B_\omega$ is also an entourage in $\mathcal{E}_\omega$.
    
    (ii) It remains to prove that $\{(\alpha,\beta)\}$ is contained in some entourage of $\mathcal{E}_\omega$ for any $\alpha,\beta\in X(\omega)$. Choose partial translations $t$ and $s$ such that
\[
\alpha\overset{t}{\sim}\omega,
\ \beta\overset{s}{\sim}\omega.
\] Then $\alpha\overset{t\circ s^{-1}}{\sim}\beta$ off a set of $\omega$-measure zero. It follows from Lemma~\ref{lem:compatible} that $(\alpha,\beta)\in\overline{\Gamma_{t\circ s^{-1}}}^{\beta X\times\beta X}$, where $t\circ s^{-1}$ is a partial translation. Thus, from Remark~\ref{rem:ultra1}(i), we have $(\alpha,\beta)\in (\Gamma_{t\circ s^{-1}})_\omega$, i.e., $\{(\alpha,\beta)\}$ is contained in the entourage $(\Gamma_{t\circ s^{-1}})_\omega$ of $\mathcal{E}_\omega$.
\end{proof}

\begin{defn}\label{def:limitcoarse2}
    For a non-principal ultrafilter
$\omega\in\partial G(X)^{(0)}$, define the downward closure of
$\mathcal{E}_\omega$ by
\[
\widetilde{\mathcal{E}_\omega}
:=\{D:D\subseteq E_\omega\text{ for some }E\in\mathcal{E}\}.
\]
By Proposition~\ref{B7}, this is a coarse structure on $X(\omega)$. We call
\[
(X(\omega),\widetilde{\mathcal{E}_\omega})
\]
the limit coarse space of $(X,\mathcal{E})$ at $\omega$.
\end{defn}

\begin{rem}\label{rem:ultra2}
    Fix a non-principal ultrafilter $\omega\in \partial G(X)^{(0)}$.
    \begin{enumerate}
        \item For any $\alpha\in X(\omega)$, we have $(X(\alpha),\widetilde{\mathcal{E}_\alpha})=(X(\omega), \widetilde{\mathcal{E}_\omega})$.
        \item
It follows from Remark~\ref{rem:ultra1}(i) that
\[
E_\omega
=
(\Gamma_{t_1})_\omega\cup\cdots\cup(\Gamma_{t_l})_\omega,
\]
where
\[
E=\Gamma_{t_1}\sqcup\cdots\sqcup\Gamma_{t_l}
\]
and the $t_i$ are partial translations. For any
$\beta\in X(\omega)$, each $(\Gamma_{t_i})_\omega$ contains at most
one pair with source $\beta$. Hence
\[
\#\{\alpha\in X(\omega):(\alpha,\beta)\in E_\omega\}\leq l.
\]
Applying the same argument to $E_\omega^{-1}=(E^{-1})_\omega$ gives
\[
\#\{\alpha\in X(\omega):(\beta,\alpha)\in E_\omega\}\leq l.
\]
Therefore
\[
\#N_{E_\omega}(\beta)\leq l+1,
\ \#N_{E_\omega^{-1}}(\beta)\leq l+1,
\]
and hence
\[
n(E_\omega)\leq l+1.
\]
Thus the limit coarse space
$(X(\omega),\widetilde{\mathcal E_\omega})$ is uniformly locally
finite.
        \item
If $(X,d)$ is a strongly discrete bounded-geometry metric space with
its bounded coarse structure, then
$(X(\omega),\widetilde{\mathcal E_\omega})$ is metrizable and its
coarse structure is the bounded coarse structure of the limit metric
space $(X(\omega),d_\omega)$ of
\cite[Definition~3.8]{SW17}.

\item
If $X=G$ is a discrete group equipped with a strongly discrete
bounded-geometry left-invariant metric, then every limit metric space
is isometric to $G$ with the original metric
\cite[Example~3.14 and Lemma~B.1]{SW17}. In particular, this applies
to $\mathbb Z^N$ with a metric induced by a norm on $\mathbb R^N$.
    \end{enumerate}
\end{rem} 

\begin{prop} (cf. \cite[Proposition 3.10]{SW17})\label{prop:capture}
    Let $\omega$ be a non-principal ultrafilter in $G(X)^{(0)}$, and $\{t_\alpha: S_\alpha\to R_\alpha\}$ a compatible family for $\omega$.
    For each finite subset $F$ of $X(\omega)$, there exists a subset $Y\subseteq X$ with $\omega(Y)=1$, and such that for each $y\in Y$ there is a finite subset $G(y)$ of $X$ such that the map
        \[
        f_y: F\to G(y),\ \alpha\mapsto t_\alpha(y)
        \]
        is a bijection.
\end{prop}

\begin{proof}
    For each $\alpha\neq\beta\in X(\omega)$, i.e., $\alpha\overset{t_\alpha}{\sim}\omega$ and $\beta\overset{t_\beta}{\sim}\omega$, the set
    \[
    Y_{\alpha\beta}:=\{x\in S_\alpha\cap S_\beta:t_\alpha(x)\neq t_\beta(x)\}
    \]
    has $\omega$-measure one. 
    Indeed, if
\[
\{x\in S_\alpha\cap S_\beta:
t_\alpha(x)=t_\beta(x)\}
\]
had $\omega$-measure one, then the two partial translations would have
the same $\omega$-limit, so
\[
\alpha=t_\alpha(\omega)=t_\beta(\omega)=\beta,
\]
contrary to $\alpha\neq\beta$. Since
$S_\alpha\cap S_\beta$ has $\omega$-measure one and $\omega$ is an
ultrafilter, it follows that $Y_{\alpha\beta}$ has
$\omega$-measure one.
    Set
    \[
    Y:=\left(\bigcap_{\alpha\in F}S_\alpha\right)\cap\left(\bigcap_{\alpha\neq\beta\in F}Y_{\alpha\beta}\right).
    \]
    As this is a finite intersection of subsets of $X$ of $\omega$-measure one, it also has $\omega$-measure one. For each $y\in Y$, define
    \[
    G(y):=\{t_\alpha(y):\alpha\in F\}
    \]
    and define $f_y:F\to G(y)$ by $f_y(\alpha)=t_\alpha(y)$; these sets and maps have the desired properties.
\end{proof}

It is clear that the existence of a bijection does not depend on the choice of the compatible family $\{t_\alpha\}$. We call the collection $\{f_y: F\to G(y)\}_{y\in Y}$ with the properties above a local coordinate system (with respect to $\{t_\alpha\}$). The maps $f_y: F\to G(y)$ are called local coordinates.

\begin{defn} (cf. \cite[Definition 4.4]{SW17}, \cite[Definition 3.9]{Zhang})\label{def:limitoperator}
    Let $\omega$ be a non-principal ultrafilter in $G(X)^{(0)}$, $p\in\{0\}\cup[1,\infty]$ and $T\in B^p_u(X,\mathcal{E})$. Fix a compatible family $\{t_\alpha\}_{\alpha\in X(\omega)}$ for $\omega$.
    
    The limit operator of $T$ at $\omega$, denoted by $\Phi_\omega(T)$, is the $X(\omega)\times X(\omega)$-indexed matrix with entries in $\mathbb{C}$ defined by
    \[
    \Phi_\omega(T)_{\alpha\beta}:=\lim_{x\to\omega}T_{t_\alpha(x)t_\beta(x)}=\overline{T}(\alpha,\beta).
    \]   
\end{defn}

\begin{rem}\leavevmode\label{rem:ultra3}
    \begin{enumerate}
        \item The limit operator $\Phi_\omega(T)$ does not depend on the choice of compatible family, since all partial translations in $[t_\alpha]$ are essentially the same in the ``direction'' of $\omega$.
        \item At this stage, it is not yet clear that $\Phi_\omega(T)$ defines a bounded operator on $\ell^p(X(\omega))$.
\end{enumerate} 
\end{rem}

\begin{prop}\label{lem:norm}
    Let $\omega$ be a non-principal ultrafilter in $G(X)^{(0)}$, $p\in\{0\}\cup[1,\infty]$ and $T\in B^p_u(X,\mathcal{E})$. Let $F$ be a finite subset of $X(\omega)$ and $\varepsilon>0$. Let $\{t_\alpha\}_{\alpha\in X(\omega)}$ be a compatible family of partial translations for $\omega$.

    Then there exists a local coordinate system $\{f_y:F\to G(y)\}_{y\in Y}$ such that for each $y\in Y$, if
    \[
    U:\ell^p(F)\to\ell^p(G(y)),\ (U\xi)(x):=\xi(f_y^{-1}(x))
    \]
    is the linear isometry induced by $f_y$, then we have
    \[
    \|U^{-1}\chi_{G(y)}T\chi_{G(y)}U-\chi_F\Phi_\omega(T)\chi_F\|_p<\varepsilon,
    \]
    where $\chi_F\Phi_\omega(T)\chi_F$ is viewed as a finite
$F\times F$ matrix with entries in $\mathbb{C}$, acting on $\ell^p(F)$ by matrix multiplication.
\end{prop}

\begin{proof}
    For each $\alpha,\beta\in X(\omega)$, the set  
    \[
    V_{\alpha\beta}=\{x\in S_\alpha\cap S_\beta:|T_{t_\alpha(x)t_\beta(x)}-\Phi_\omega(T)_{\alpha\beta}|<\varepsilon/|F|^2\}
    \]
    has $\omega$-measure one. Now set
    \[
    Y:=\left(\bigcap_{\alpha\in F}S_\alpha\right)\cap\left(\bigcap_{\alpha\neq\beta\in F}Y_{\alpha\beta}\right)\cap\left(\bigcap_{\alpha,\beta\in F}V_{\alpha\beta}\right),
    \]
    which again has $\omega$-measure one as it is a finite intersection of sets with $\omega$-measure one. The choice of the sets $V_{\alpha\beta}$ ensures that $Y$ has the required properties. 
\end{proof}

\begin{lem}\cite[Lemma 2.9]{Zhang}\label{lem:submatrix}
Let $p\in[1,\infty)$, and let
$A=[a_{xy}]_{x,y\in X}$ be an $X\times X$ matrix with entries in
$\mathbb{C}$. Suppose there exists $M>0$ such that, for every finite
subset $F\subseteq X$, multiplication by the finite submatrix
$[a_{xy}]_{x,y\in F}$ defines an operator $A_F$ on $\ell^p(F)$ with $\|A_F\|_p\leq M$.
Then multiplication by $A$ defines a bounded operator on $\ell^p(X)$,
and
\[
\|A\|_p
=
\sup_{F\in\mathcal{F}}\|A_F\|_p.
\]
\end{lem}

\begin{lem}\label{lem:multi}
    Let $\omega$ be a non-principal ultrafilter in $G(X)^{(0)}$, and $\{t_\alpha:S_\alpha\to R_\alpha\}$ a compatible family for $\omega$. For each finite subset $F\subseteq X(\omega)$, if $E\in\mathcal{E}$, $\gamma\in X(\omega)$ and $E_\omega^\gamma:=\{\delta\in X(\omega):(\gamma,\delta)\in E_\omega\}\subseteq F$ then there exists a local coordinate system $\{f_y: F\to G(y)\}_{y\in Y}$ such that 
    \[
    f_y(E_\omega^\gamma)=E^{t_\gamma(y)}:=\{x\in X:(t_\gamma(y),x)\in E\},
    \]
    for each $y\in Y$.
\end{lem}

\begin{proof}
    Suppose $E=\Gamma_{s_1}\sqcup\Gamma_{s_2}\sqcup\cdots\sqcup\Gamma_{s_l}$, where $\{s_i:S_i\to R_i\}^{l}_{i=1}$ is a finite family of partial translations. Define $I:=\{i:\omega(S_\gamma\cap t_\gamma^{-1}(R_i))=1\}$, $u_i:=s_i^{-1}\circ t_\gamma$ for $i\in I$. Then $E_\omega^\gamma=\{\eta_i:=u_i(\omega)\in X(\omega), i\in I\}$. Set
    \[
    C_i=\begin{cases}
        S_\gamma\cap t_\gamma^{-1}(R_i)\cap S_{\eta_i}\cap\{y:u_i(y)=t_{\eta_i}(y)\} &i\in I,\\
        X\setminus(S_\gamma\cap t_\gamma^{-1}(R_i)) &i\notin I.
    \end{cases}
    \]
    Then the set $C_i$ always has $\omega$-measure one, for $1\leq i\leq l$. Now set
    \[
    Y:=\left(\bigcap_{\alpha\in F}S_\alpha\right)\cap\left(\bigcap_{\alpha\neq\beta\in F}Y_{\alpha\beta}\right)\cap\left(S_\gamma\cap\bigcap_{i}C_i\right),
    \]
    which again has $\omega$-measure one as it is a finite intersection of sets with $\omega$-measure one. The choices of $C_i$ guarantee that 
\[
E^{t_\gamma(y)}
=
\{u_i(y):i\in I\}
=
\{t_{\eta_i}(y):i\in I\}
=
f_y(E^\gamma_\omega)
\]
    for any $y\in Y$.
\end{proof}

\begin{prop}\label{lem:norm1}
    For each non-principal ultrafilter $\omega\in \partial G(X)^{(0)}$, $p\in\{0\}\cup[1,\infty]$ and $T\in B^p_u(X,\mathcal{E})$, by matrix multiplication, the limit operator $\Phi_\omega(T)$ defines a bounded operator in $B^p_u(X(\omega),\widetilde{\mathcal{E}_\omega})$. Moreover, the map
    \[
    \Phi_\omega: B^p_u(X,\mathcal{E})\to B^p_u(X(\omega),\widetilde{\mathcal{E}_\omega})
    \]
    is a contractive homomorphism. 
\end{prop}

\begin{proof}
     Let $T\in B^p_u(X,\mathcal{E})$. We first show that
$\Phi_\omega(T)$ defines a bounded operator on $\ell^p(X(\omega))$ and
that
\[
\|\Phi_\omega(T)\|_p\leq\|T\|_p.
\]

Suppose first that $p\in[1,\infty)$. Let
$F\subseteq X(\omega)$ be finite and let $\varepsilon>0$. By
Proposition~\ref{lem:norm}, there is a finite subset $G\subseteq X$ and
an isometric identification
\[
U:\ell^p(F)\to\ell^p(G)
\]
such that
\[
\|U^{-1}\chi_GT\chi_GU
-\chi_F\Phi_\omega(T)\chi_F\|_p
<\varepsilon.
\]
Hence
\[
\|\chi_F\Phi_\omega(T)\chi_F\|_p
\leq
\|\chi_GT\chi_G\|_p+\varepsilon
\leq
\|T\|_p+\varepsilon.
\]
Since $\varepsilon>0$ is arbitrary,
\[
\|\chi_F\Phi_\omega(T)\chi_F\|_p\leq\|T\|_p
\]
for every finite $F\subseteq X(\omega)$. Lemma~\ref{lem:submatrix} therefore gives a bounded operator
$\Phi_\omega(T)$ on $\ell^p(X(\omega))$ satisfying
\[
\|\Phi_\omega(T)\|_p\leq\|T\|_p.
\]

Now suppose $p=\infty$. Applying the preceding finite-compression
estimate to finite subsets and using the formula for the norm of a
finite matrix on $\ell^\infty$, we obtain, for every
$\alpha\in X(\omega)$,
\[
\sum_{\beta\in X(\omega)}
|\Phi_\omega(T)_{\alpha\beta}|
\leq\|T\|_\infty.
\]
Thus multiplication by $\Phi_\omega(T)$ defines a bounded operator on
$\ell^\infty(X(\omega))$, and
\[
\|\Phi_\omega(T)\|_\infty
=
\sup_{\alpha\in X(\omega)}
\sum_{\beta\in X(\omega)}
|\Phi_\omega(T)_{\alpha\beta}|
\leq\|T\|_\infty.
\]

It remains to treat $p=0$. Let
\[
\Gamma:B^0_u(X,\mathcal{E})\to B^\infty_u(X,\mathcal{E}),
\ T\mapsto T^{''},
\]
be the canonical isometric isomorphism from
Section~\ref{subsection1}. Choose
$A_n\in\mathbb{C}_u[X,\mathcal{E}]$ such that
\[
\|A_n-T\|_0\to0.
\]
Then
\[
\|A_n^{''}-T^{''}\|_\infty
=
\|A_n-T\|_0
\to0.
\]
By the $p=\infty$ case,
\[
\Phi_\omega(A_n^{''})\to\Phi_\omega(T^{''})
\]
in operator norm on $\ell^\infty(X(\omega))$.

Each $\Phi_\omega(A_n^{''})$ has controlled propagation, and therefore
leaves $\ell^0(X(\omega))$ invariant. Since $\ell^0(X(\omega))$ is
closed in $\ell^\infty(X(\omega))$, the operator
$\Phi_\omega(T^{''})$ also leaves $\ell^0(X(\omega))$ invariant. Let
\[
S:=\Phi_\omega(T^{''})|_{\ell^0(X(\omega))}.
\]
The restrictions
\[
\Phi_\omega(A_n^{''})|_{\ell^0(X(\omega))}
\]
are controlled propagation operators and converge in operator norm to
$S$. Hence
\[
S\in
B^0_u(X(\omega),\widetilde{\mathcal{E}}_\omega).
\]
Moreover, $T$ and $T^{''}$ have the same matrix coefficients, so for
$\alpha,\beta\in X(\omega)$,
\[
S_{\alpha\beta}
=
\Phi_\omega(T^{''})_{\alpha\beta}
=
\lim_{x\to\omega}
T_{t_\alpha(x)t_\beta(x)}
=
\Phi_\omega(T)_{\alpha\beta}.
\]
Thus $S=\Phi_\omega(T)$ as a matrix operator on
$\ell^0(X(\omega))$, and
\[
\|\Phi_\omega(T)\|_0
\leq
\|\Phi_\omega(T^{''})\|_\infty
\leq
\|T^{''}\|_\infty
=
\|T\|_0.
\]

Therefore, for every $p\in\{0\}\cup[1,\infty]$, the matrix
$\Phi_\omega(T)$ defines a bounded operator and
\[
\|\Phi_\omega(T)\|_p\leq\|T\|_p.
\]
Since linearity is immediate from the matrix coefficients, we obtain a
contractive linear map
\[
\Phi_\omega:
B^p_u(X,\mathcal{E})\to B(\ell^p(X(\omega))).
\]
     By definition, for any $A\in \mathbb{C}_u[X,\mathcal{E}]$ with $\supp(A)\subseteq E\in\mathcal{E}$, we have $\supp(\Phi_\omega(A))\subseteq E_\omega$. Hence, $\Phi_\omega(A)\in\mathbb{C}_u[X(\omega),\widetilde{\mathcal{E}_\omega}]$. Thus, we have 
     \[
     \Phi_\omega(B^p_u(X,\mathcal{E}))\subseteq B^p_u(X(\omega),\widetilde{\mathcal{E}_\omega}).
     \]
     It remains to prove that $\Phi_\omega$ is multiplicative. For any given $A, B\in \mathbb{C}_u[X,\mathcal{E}]$ with $\supp(A)\subseteq E\in\mathcal{E}$ and $\alpha,\beta\in X(\omega)$, $F:=E_\omega^\alpha\cup\{\alpha,\beta\}$ is a finite subset of $X(\omega)$. It follows from Lemma~\ref{lem:multi} that there exists a local coordinate system $\{f_y: F\to G(y)\}_{y\in Y}$ such that 
    \[
    E^{t_\alpha(y)}=\{t_\gamma(y):\gamma\in E_\omega^\alpha\},
    \]
    for each $y\in Y$. Thus, 
    \begin{align*}
        (AB)_{t_\alpha(y)t_\beta(y)}&=\sum_{z\in X}A_{t_\alpha(y)z}B_{zt_\beta(y)}\\
        &=\sum_{z\in E^{t_\alpha(y)}}A_{t_\alpha(y)z}B_{zt_\beta(y)}\\
        &=\sum_{\gamma\in E_\omega^\alpha}A_{t_\alpha(y)t_\gamma(y)}B_{t_\gamma(y)t_\beta(y)},
    \end{align*}
    which is a finite sum. Hence,
    \begin{align*}
    \Phi_\omega(AB)_{\alpha\beta}&=\lim_{y\to\omega}(AB)_{t_\alpha(y)t_\beta(y)}=\lim_{y\to\omega}\sum_{\gamma\in E_\omega^\alpha}A_{t_\alpha(y)t_\gamma(y)}B_{t_\gamma(y)t_\beta(y)}\\
    &=\sum_{\gamma\in E_\omega^\alpha}\left(\lim_{y\to\omega}A_{t_\alpha(y)t_\gamma(y)}\right)\left(\lim_{y\to\omega}B_{t_\gamma(y)t_\beta(y)}\right)\\
    &=\sum_{\gamma\in E_\omega^\alpha}\Phi_\omega(A)_{\alpha\gamma}\Phi_\omega(B)_{\gamma\beta}=\sum_{\gamma\in X(\omega)}\Phi_\omega(A)_{\alpha\gamma}\Phi_\omega(B)_{\gamma\beta}\\
    &=(\Phi_\omega(A)\Phi_\omega(B))_{\alpha\beta},
    \end{align*}
    i.e., $\Phi_\omega(AB)=\Phi_\omega(A)\Phi_\omega(B)$. For $S, T\in B^p_u(X)$, we have $S_n, T_n\in\mathbb{C}_u[X,\mathcal{E}]$ such that $S_n\to S$, $T_n\to T$ and $S_nT_n\to ST$. Since $\Phi_\omega$ is continuous, we have 
    \[
    \Phi_\omega(ST)=\lim_n\Phi_\omega(S_nT_n)=\lim_n\Phi_\omega(S_n)\Phi_\omega(T_n)=\Phi_\omega(S)\Phi_\omega(T).
    \]
\end{proof}

The algebra isomorphism 
\begin{align*}
\Psi_X:C_c(G(X))&\to \mathbb{C}_u[X,\mathcal{E}]\\
        f&\mapsto T_{f},
\end{align*}
also yields the following.

\begin{prop} \label{prop:limitgroupoid}
    For each non-principal ultrafilter $\omega\in \partial G(X)^{(0)}$, define
    \[
    \Psi_{X(\omega)}:C_c(G(X))\to \mathbb{C}_u[X(\omega),\widetilde{\mathcal E}_\omega]
    \]
    by
    \[
    \Psi_{X(\omega)}(f)=T_{f,\omega},\ T_{f,\omega}(\alpha,\beta):=f(\alpha,\beta),
    \]
    for $\alpha,\beta\in X(\omega)$. Then $\Psi_{X(\omega)}$ is an algebra homomorphism and
    \[
    \Psi_{X(\omega)}(f)=\Phi_\omega(\Psi_X(f))
    \]
    for every $f\in C_c(G(X))$.
\end{prop}

\begin{proof}
    Fix a non-principal ultrafilter $\omega\in\partial G(X)^{(0)}$. For $f\in C_c(G(X))$, choose $E\in\mathcal E$ such that
    \[
    \supp(f)\subseteq \overline{E}^{\,\beta(X\times X)}.
    \]
    By Remark B.6(i),
    \[
    \supp(T_{f,\omega})\subseteq E_\omega\in\widetilde{\mathcal E}_\omega.
    \]
    Since $f$ is bounded, it follows that $T_{f,\omega}\in \mathbb{C}_u[X(\omega),\widetilde{\mathcal E}_\omega]$.

    Under the identification of the restriction of the coarse groupoid to $X(\omega)$ with a subgroupoid of the pair groupoid $X(\omega)\times X(\omega)$, convolution corresponds to matrix multiplication. Hence $\Psi_{X(\omega)}$ is an algebra homomorphism.

    Finally, if $\alpha,\beta\in X(\omega)$ and
    $\{t_\gamma\}_{\gamma\in X(\omega)}$ is a compatible family for $\omega$, then
    \[
    \begin{aligned}
    \Phi_\omega(\Psi_X(f))_{\alpha\beta}
    &=\lim_{x\to\omega}\Psi_X(f)(t_\alpha(x),t_\beta(x))\\
    &=\lim_{x\to\omega}f(t_\alpha(x),t_\beta(x))\\
    &=f(\alpha,\beta)=T_{f,\omega}(\alpha,\beta),
    \end{aligned}
    \]
    where the third equality follows from continuity of $f$. Therefore
    \[
    \Phi_\omega(\Psi_X(f))=\Psi_{X(\omega)}(f).
    \]
\end{proof}

\begin{cor} (cf. \cite[Lemma C.3]{SW17})\label{lem:norm2}
Let $\omega$ be an ultrafilter in $G(X)^{(0)}$, and let $F:X(\omega)\to G(X)_\omega,\;\alpha\mapsto(\alpha,\omega)$ be the bijection defined in Remark~\ref{rem:compatible}(ii). For $p\in[1,\infty]$, if 
\[ U:\ell^p(G(X)_\omega)\to\ell^p(X(\omega)),\; (Uv)(\alpha):=v(\alpha,\omega) \] is the invertible isometry induced by $F$, and 
\[ \Psi_{X(\omega)}:C_c(G(X))\to \mathbb{C}_u[X(\omega),\widetilde{\mathcal{E}_\omega}] \]
is the algebra homomorphism above, then \[ U\lambda_\omega(f)U^{-1}=\Psi_{X(\omega)}(f)=\Phi_\omega(\Psi_X(f)) \] for all $f\in C_c(G(X))$ (set $\Phi_\omega=id$ when $\omega$ is a principal ultrafilter).
\end{cor}

\begin{rem}\leavevmode\label{rem:norm}
    \begin{enumerate}
        \item When $\omega=x\in X$ is a principal ultrafilter, we have
        \[
        \|\lambda_x(f)\|_{B(\ell^p(G(X)_x))}=\|\Psi_X(f)\|_{B(\ell^p(X))};
        \]
        \item When $\omega\in \partial G(X)^{(0)}$ is a non-principal ultrafilter, we have
        \[
        \|\lambda_\omega(f)\|_{B(\ell^p(G(X)_\omega))}=\|\Psi_{X(\omega)}(f)\|_{B(\ell^p(X(\omega)))}=\|\Phi_\omega(\Psi_X(f))\|_{B(\ell^p(X(\omega)))}.
        \]
    \end{enumerate}
\end{rem}

By Proposition~\ref{prop:isometric}, the algebraic isomorphism
\[
\Psi_X:C_c(G(X))\to\mathbb C_u[X,\mathcal E]
\]
extends isometrically to
\[
\Psi:F^p_{\mathrm{red}}(G(X))\to B^p_u(X).
\]
For each $\omega\in\partial G(X)^{(0)}$, the regular representation
extends contractively to
\[
\widetilde{\lambda}_\omega:
F^p_{\mathrm{red}}(G(X))\to B(\ell^p(G(X)_\omega)).
\]
Set
\[
\widetilde{\Psi}_{X(\omega)}
:=
\Phi_\omega\circ\Psi:
F^p_{\mathrm{red}}(G(X))\to B(\ell^p(X(\omega))).
\]
By Proposition~\ref{prop:isometric} and Corollary~\ref{lem:norm2},
this is the continuous extension of
$\Psi_{X(\omega)}$ from Proposition~\ref{prop:limitgroupoid}.
Moreover, for every
$f\in F^p_{\mathrm{red}}(G(X))$,
\[
\widetilde{\Psi}_{X(\omega)}(f)
=
\Phi_\omega(\Psi(f))
=
U\widetilde{\lambda}_\omega(f)U^{-1}.
\]
Equivalently, the following diagram commutes:
    \[
    \begin{tikzcd}
        F^p_{red}(G(X))\ar[r,"\Psi", "\cong"']\ar[d,"\widetilde{\lambda}_{\omega}"]\ar[dr,"\widetilde{\Psi}_{X(\omega)}" description]& B^p_u(X)\ar[d,"\Phi_\omega"]\\
        B(\ell^p(G(X)_\omega))\ar[r,"\Ad_U"',"\cong"]& B(\ell^p(X(\omega)))
    \end{tikzcd}
    \]

\section*{Acknowledgments}

The first author is supported by the Research Start-up Fund of Jilin University (grant number 419080524100).

\bibliographystyle{plain}
\bibliography{ref}

\end{document}